\documentclass[pdflatex,sn-mathphys-ay]{sn-jnl}

\usepackage{graphicx}%
\usepackage{multirow}%
\usepackage{amsmath,amssymb,amsfonts}%
\usepackage{amsthm}%
\usepackage{mathrsfs}%
\usepackage[title]{appendix}%
\usepackage{xcolor}%
\usepackage{textcomp}%
\usepackage{manyfoot}%
\usepackage{booktabs}%
\usepackage{algorithm}%
\usepackage{algorithmicx}%
\usepackage{algpseudocode}%
\usepackage{listings}%
\usepackage{lscape}%
\usepackage{nomencl}
\usepackage[printonlyused,withpage]{acronym}%
\makenomenclature
\usepackage{mathtools}%
\usepackage{dsfont}%
\usepackage[capitalise]{cleveref}%
\usepackage{array}%
\usepackage{cancel}

\newcommand{\R}{\mathrm{R}}
\newcolumntype{?}{!{\vrule width 0.8pt}}

\newcommand{\up}[1]{$^{\mathrm{#1}}$}
\newcommand{\I}{\mathrm{I}}
\newcommand{\J}{\mathcal{J}}
\newcommand{\T}[1]{\mathbf{T}_{#1}}
\newcommand{\U}{\mathcal{U}}
\renewcommand{\S}[1]{\mathrm{S}_{#1}}
\newcommand{\V}{\mathcal{V}}
\renewcommand{\d}{\mathrm{d}}
\renewcommand{\t}{\mathrm{t}}
\newcommand{\y}{\mathrm{y}}
\newcommand{\hy}{\hat{\mathrm{y}}}
\newcommand{\Y}{\mathrm{Y}}

\newcommand{\mT}{\mathcal{T}}
\newcommand{\N}{\mathrm{N}}
\newcommand{\E}{\mathrm{E}}
\newcommand{\mE}{\mathcal{E}}
\newcommand{\mA}{\mathrm{A}}
\newcommand{\MA}{\mathsf{A}}
\newcommand{\MAp}{\mathsf{A}^\prime}
\newcommand{\sX}{\mathsf{X}}
\newcommand{\sx}{\mathsf{x}}
\newcommand{\sY}{\mathsf{Y}}
\newcommand{\sy}{\mathsf{y}}
\newcommand{\MB}{\mathsf{B}}
\newcommand{\e}{\mathrm{e}}
\newcommand{\me}{{e}}
\newcommand{\he}{\hat{{e}}}
\newcommand{\tol}{\mathsf{tol}}
\newcommand{\sfT}{\mathcal{T}}
\newcommand{\h}{\tau}
\newcommand{\f}{{f}}
\newcommand{\F}{\mathrm{F}}
\newcommand{\g}[1]{\gamma_{(#1\arrowvert n)}}
\newcommand{\Gp}[1]{\gamma^{\prime}_{(#1\arrowvert n)}}
\newcommand{\G}[1]{\Gamma_{(#1\arrowvert n)}}
\newcommand{\gf}[1]{\gamma_{(#1)}}
\newcommand{\Gf}[1]{\Gamma_{(#1)}}

\renewcommand{\r}[1]{r_{(#1\arrowvert n)}}
\newcommand{\eps}[1]{\varepsilon_{(#1\arrowvert n)}}
\newcommand{\Eps}[1]{\xi_{(#1\arrowvert n)}}
\renewcommand{\a}{\kappa}
\newcommand{\tnmh}{\t_{n-1/2}}
\newcommand{\ynmh}{\y_{n-1/2}}
\newcommand{\Tnmh}[1]{\mT_{n-1/2,#1}}
\newcommand{\Ynmh}[1]{\Y_{n-1/2,#1}}
\newcommand{\nf}[2]{\Phi^{{[#1]}}_{#2}}
\newcommand{\nfc}[2]{\Psi^{{[#1]}}_{#2}}
\renewcommand{\Im}{\mathrm{Im}}
\renewcommand{\Re}{\mathrm{Re}}
\newcommand{\eC}{\mathcal{C}}
\newcommand{\C}[1]{\mathrm{C}_{#1}}
\newcommand{\Mat}[2]{\mathsf{Mat}_{#1}\left(#2\right)}
\newcommand{\GL}[2]{\mathrm{GL}_{#1}\left(#2\right)}
\newcommand{\mi}{\mathrm{i}} 
\newcommand{\op}{\mathrm{p}}

\theoremstyle{thmstyleone}%
\newtheorem{theorem}{Theorem}
\newtheorem{proposition}[theorem]{Proposition}%
\newtheorem{lemma}[theorem]{Lemma}

\theoremstyle{thmstyletwo}%

\theoremstyle{thmstylethree}%
\newtheorem{definition}{Definition}%

\begin{document}

\title[Article Title]{Error estimation for numerical approximations of ODEs via composition techniques. Part II: BDF methods}


\author*[1]{\fnm{Ahmad} \sur{Deeb}}
\email{ahmad.deeb@ku.ac.ae}
\author[1]{\fnm{Denys} \sur{Dutykh}}
\email{denys.dutykh@ku.ac.ae}
\author[2]{\fnm{Maryam} \sur{Al Zohbi}}
\email{maryam.al-zohbi@univ-cotedazur.fr}

\affil*[1]{\orgdiv{Department of Mathematics}, \orgname{Khalifa University of Science and Technology}, \orgaddress{\city{Abu Dhabi}, \postcode{PO Box 127788}, \country{UAE}}}

\affil[2]{\orgdiv{LJAD}, \orgname{University Cote D'Azur}, \orgaddress{\city{Nice}, \postcode{06100}, \country{France}}}

\abstract{Integration of Ordinary Differential Equations (ODEs) using \ac{BDF} methods with $\op$ backwards steps achieves order $\op$ accuracy if specific conditions are met. This work extends the composition technique with complex coefficients to the implicit \ac{BDF} schemes, increasing the approximation order by one without additional backward points. The imaginary part of the composed flow provides an error estimate of order $\op+1$. Linear stability analysis reveals that the composed schemes break the Dahlquist barrier, achieving stability up to order eight. The computational performance of the composed flow outperforms \ac{BDF} schemes when using the same number of backward points, allowing for higher accuracy with lower CPU time. For non-uniform meshes, the ratio of consecutive time steps, which influences stability, appears as a parameter in the roots of algebraic equations relative to the composed flow. Having a complex root with a real positive part implies a lower bound to this ratio depending on the order. For example, the bound is $0.4506$ for order three and $0.6806$ for order four. Numerical tests demonstrate the effectiveness of this technique in improving the accuracy and stability compared to \ac{BDF} methods.}


\maketitle

\section{Introduction}\label{sec1}

In this second part study, we are interested in providing a numerical solution and a local error estimate for the following Cauchy problem using the composition technique of \ac{BDF} schemes:
\begin{equation}
 \label{diff-system}
 \frac{\d\y}{\d\t} = \f(\t,\y), \quad \t \in \I = ]\t_0,\T{}[,\qquad \y(\t_0) = \y_0,
\end{equation}
with
\begin{equation}
\y:\left\rvert
 \begin{array}{ccc}
  \I\subset \mathds{R} &\rightarrow &\U \subset \mathds{R}^d\\
  \t &\mapsto &\y(\t),
 \end{array}
 \right.
\qquad
\f: \left\rvert
\begin{array}{ccc}
  \I \times \U&\rightarrow &\mathds{R}^d\\
  (\t,\y) &\mapsto &\f(\t,\y).
 \end{array}
 \right.
 \end{equation}
In real problems, the right-hand side $\f$ is defined for real arguments and returns real values. In this manuscript, we consider the cases where $\f$ could also be extended to complex subsets and defined in $\J\times \V$, where $\I\subset \J \subset \mathds{C}$ and $\U\subset \V\subset \mathds{C}^d$. The extension, if possible, is done naturally.

\subsection*{Numerical approximations via different classes of schemes}

The numerical integration of System \eqref{diff-system} has been extensively investigated \cite{butcher-2009}, and a wide range of numerical schemes has been proposed to approximate solutions for stiff and non-stiff problems \cite{book:hairer, book:hairer2}.
Classical numerical schemes are decomposed into two types: one-step methods such as, for instance, the versatile explicit fourth-order \ac{RK} scheme \cite{Runge_1895,butcher_1964}, and \ac{LMS} methods, such as the second-order \ac{BDF}2 resulting from approximating the first derivative by difference formulas. The topic of developing efficient and high-order numerical methods for time integration remains significant and of interest in current research \cite{Diaz-2021, Vu-2024, Kropielnicka-2024, comp1, comp2, comp3, comp4}.

The accuracy and stability of numerical methods are crucial for ensuring reliable numerical solutions \cite{axelsson-69}. Explicit \ac{RK} schemes with $s$ stages are easily available to produce numerical solutions for different orders of approximations \cite{butcher_1963,butcher_1996,verwer_1996}. The region of stability of these methods is small, imposing a small time step when approximating solutions of highly stiff problems. \ac{ERK} methods \cite{iserles_2008, Vermeire-23} have been developed to adapt the local time step and improve the performance of the simulation when using explicit schemes. The idea consists of having two approximations of the solution with different orders of approximation, allowing for an error estimate could be deduced. Bogacki and Shampine \cite{bogacki_shampine} developed a third-order \ac{ERK} method with four stages, whilst a fifth-order \ac{ERK} method was presented by Dormand and Prince \cite{dromand_prince}.
For more details on \ac{ERK} methods, refer to \cite{book:butcher, book:tomas}.
The \ac{IRK} methods are then developed to produce numerical schemes with larger regions of stability, such as the Lobatto methods \cite{Jay-15} that are based on the trapezoidal quadrature rule, the Radau \cite{radau-88,hairer-99} and Gau\ss-RK methods \cite[pp 34]{hairer2002geometric} which are collocation-type methods. Embedded techniques were also developed for collocation methods \cite{pia-17}.
Other types of numerical methods exist, such as \ac{ETD} methods \cite{Pope_1963,hochbruck2010exponential,cox_2002, Maset-21} that are based on integrating exactly, if possible, the linear part of the equation.
Other classes of numerical integrators are based on \ac{DSR} \cite{ahmad_bpl_2014,ahmad_comp_bpl_sfg_2015,ahmad_icnpaa_2016} such as \ac{BPL} that were applied to solving stiff and non-stiff problems \cite{DEEB_2022_bpl,ahmad_pgd_pade, deeb:stab-NS}.

Another class of numerical methods that has been developed for solving the Cauchy problem is the so-called \ac{LMS} schemes, as seen in \cite[Chapter 3]{book:hairer} and \cite{kirlinger-04,CiCP-35-1327}, such as the Adams-Bashforth and Milne-Simpson methods. These methods were developed to improve the stability of simulations for stiff problems \cite{book:hairer2,okounghae-12,xiao-14}. They require additional starting points to provide numerical approximations on subsequent layers. Dahlquist \cite{dahlquist-56,dahlquist-63} explored and advanced the stability and convergence of these classes based on the generating polynomials of corresponding \ac{LMS} methods.

\subsection*{Stability of \ac{BDF} schemes}

When having a uniform mesh, explicit \ac{BDF} schemes are stable up to the second order, while implicit ones are stable up to the sixth order \cite{hairer-83}. When applied to PDEs, the second-order implicit \ac{BDF} scheme was studied and used for solving parabolic-type problems. Bokanowski \cite{bokanowski-21} studied their stability, Wang \cite{wang-21} provided error estimation, Akrivis \cite{arkrivis-13,arkrivis-20} applied \ac{BDF} schemes to the non-linear cases, Xu \cite{Xu-23} provide a long error estimate for Delay differential equations, and Gragg and Stettar \cite{gragg-64} used one or more off-grid points to increase the order of approximations.

\ac{BDF} schemes were also developed for non-uniform meshes. Among others, Becker \cite{becker-98} proposed a variable time stepping technique for a parabolic problem.
Stability is altered in non-uniform meshes and has been studied in the case of linear diffusion equations using a class of \ac{DOC}\cite{Liao-20}. An upper bound on the ratio between two consecutive time steps is established to maintain stability.

\subsection*{Increasing the order by composition techniques}

Consider a situation where one has a basic one-step numerical scheme with low-order accuracy and wants to provide approximations with higher orders. Composition techniques enable us to construct new numerical schemes of higher order. It was introduced in several works: Yoshida \cite{yoshida-1990} used the composition to introduce symplectic integrators for \ac{ODE}s, Suzuki \cite{suzuki-1990} improved the precision of the approximation of solution for two-body problems using the composition of basic numerical flows and McLachlan \cite{Mc-1995} considered a new type of differential equation $\dot{y} = A+B$, where vector fields $A$ and $B$ can be integrated exactly. Then, their solutions were composed to produce a general solution to the initial problem. Blanes \emph{et al.} \cite{blanes-01,casas-2006,casas-2008} have developed numerical methods with high-order approximations by composing basic numerical integrators with the adjoint, which is defined by the inverse flow with a negative time step, or by composing several times second or fourth-order symmetric methods. Casas \emph{et al.} \cite{casas_2021_complex} have used the composition of one-step methods with complex coefficients to develop pseudo-symplectic and pseudo-symmetric methods. Complex coefficients were used first by Castella \emph{et al.} \cite{castella_2009} for solving parabolic equations, while Blanes \emph{et al.} \cite{blanes-10} employed them in splitting techniques for enhancing the order of approximations \cite{blanes-99}, as negative real coefficients were inevitable for composing numerical methods with order higher than two \cite{blanes-10}.

However, the composition of \ac{LMS} methods has been proposed in a cyclic way, where several schemes are composed recursively.
Hansen \cite{Hansen-69} established new predictor-corrector methods by the cyclic composition. Different works have been performed in the context of cyclic composition, see \emph{e.g.} \cite{donelson-71,bickart-73,cash-77,tendler-78}, to enhance the stability of numerical simulations for highly stiff differential equations. For example, a composition is obtained by computing first an approximation using the second-order Simpson rule, then marching in time using the third-order Adams-Moulton, and finalizing it with the Simpson rule again.

\subsection*{Error estimation and adaptivity}

Error estimation is a must in numerical simulation provided via one-step \cite{johnson-88,choi-96,soderlind-02,yan-21,jaradat-23,wang-23} or \ac{LMS} methods \cite{baron-17,soderlind-21,meng-23}.
It improves the stability and performance of simulations, as \ac{ATS} could be employed \cite{soderlind-06, CiCP-13-461}.
Having a numerical scheme of order $\op$ that produces an approximation $\y_n$, the exact local error $\me(\t_n) \coloneqq \|\y(t_n) - \y_n\| $ could be approximated by $\me_n \simeq \C{}\times \h_n^{\op+1}$,
where $\C{}$ is a positive constant to be found. Finding this error constant for every numerical scheme is of capital importance. Hairer \emph{et al.} \cite{book:hairer} presented a list of error constants for \ac{RK} and \ac{LMS} methods. Other works have also studied the local error for \ac{RK} \cite{khashin-14,hochbruck-18} and for \ac{LMS} methods \cite{lambert-90,soderlind-21}, and use this error for \ac{ATS} technique for solving the Cahn--Hilliard problem \cite{chen-19} and the Allen--Cahn problem \cite{Liao-20} using the second order \ac{BDF}.

Blanes \emph{et al.} \cite{blanes-19} established an error estimate of the approximation obtained by a composition technique of one-step methods when having real coefficients. This error estimate is used for adaptive mesh refinement. When it comes to \ac{LMS}, adaptivity affects the stability as mentioned above. Liao \cite{Liao-20} establishes that the ratio should be smaller than $3.561$ to maintain energy stability using \ac{BDF}2 and smaller than $2.414$ to satisfy the discrete maximum principle for the Allen-Cahn equations. Li \cite{Li-22} shows that the \ac{BDF}2 scheme is stable for arbitrary time grids, while for \ac{BDF}3, the ratio of adjacent time steps must be smaller than $2.553$. Liao extends the ratio results for \ac{BDF}2 and \ac{BDF}3 to a class of diffusion equations \cite{Liao-21, Liao-23}.

In recent works, Laadhari \emph{et al.} \cite{deeb:AML1,deeb:casson} have employed the composition technique with complex coefficients of (\textit{a}) the second order \ac{CN} and (\textit{b}) the implicit \ac{BDF}2, both tailored for time marching of the membranes' dynamic simulation of red blood cells in a quasi-Newtonian blood flow. The numerical outputs of this composition are complex values, whose real parts provide the fourth-order approximations for the \ac{CN} and the third-order when developed for the \ac{BDF}2. The imaginary parts provide the third-order error estimation for both resulting approximations,
enabling dynamical simulation via \ac{ATS} process. This technique of composition using complex coefficients
has been generalized in the first part of this work \cite{deeb:part1} to any one-step method.

\subsection*{Novelties}
In this manuscript, we extend the composition technique introduced for the implicit \ac{BDF}2 scheme in \cite{deeb:casson} to any \ac{BDFp} scheme with $\op \in \S{1}^{8}$. The necessary condition for increasing the order of accuracy by one in the new (composed) scheme is the use of coefficients that are roots of associated algebraic equations. Among these roots, at least one is a complex number with a positive real part.

Performing the composition with the complex root requires computations in the complex plane, producing outputs with complex values. The real parts of these outputs represent approximations with an additional order of accuracy. While computations involving both real and imaginary parts increase the computational cost, numerical tests reveal that the composed flow of order $\op$, which uses $\op-1$ backward points, competes with classical \ac{BDFp} schemes (using $\op$ backward points) in achieving the same accuracy without additional CPU time for $\op=3$. Moreover, it outperforms classical \ac{BDFp} schemes for $\op > 3$, or for any order when the same number of backward points is used.

This composition technique also breaks the Dahlquist barrier of linear stability for implicit \ac{BDFp} schemes, which is limited to order \textbf{six}. The new composed schemes are stable up to order \textbf{eight}. Additionally, the imaginary parts of the outputs are shown to serve as error estimates for the resulting approximations. These error estimates are then utilized in the \ac{ATS} technique.

We further demonstrate that the root of the algebraic equations for compositions of order $\op \in \S{2}^{8}$ will have a positive real part—necessary for stable time stepping—if the ratio of two consecutive time steps satisfies a lower bound. For instance, in the case of double composing \ac{BDF}1 (a scheme of order two), the root will always have a positive real part, regardless of the ratio, aligning with the result in \cite{Li-22}. However, for the composition of \ac{BDF}2 (a scheme of order three), the root will have a positive real part only if the ratio exceeds $0.4506$. Similarly, for the fourth-order composed flow, the ratio must be greater than $0.6806$. Other lower bounds for composed flows up to order eight are presented in \cref{tab7}.

The structure of the paper will be as follows. The next Section presents the mathematical framework of \ac{BDF} schemes and tools we use to develop the process of composing them in \cref{sec-comp}. \cref{sec-nov} headlines the outcomes of this work. \cref{proof_theorem2} presents the proof of the main result.
\cref{sec-conv} demonstrates the convergence error of the composition.
\cref{sec-stab} presents the linear stability analysis of the composed flow, while \cref{sec_vts} presents the variable time stepping strategy.
\cref{sec-num} highlights some numerical tests that stress the usefulness of complex parts in providing error estimates and discusses the computational efficiency that results from involving complex parts.
We end with discussions, main conclusions, and some perspectives in \cref{sec13}.

\section{Notations and preliminaries}

For every $\t \in \I$, we denote by $\varphi_{\t}$ the exact flow of the Cauchy problem \eqref{diff-system}. It is defined by the following map:
\begin{equation}
 \varphi_{\t}: \left\rvert
 \begin{array}{ccl}
  \mathds{R}^d &\longrightarrow &\mathds{R}^d\\
  \y_0 & \mapsto & \varphi_{\t}(\y_0) = \y(\t),
 \end{array}
 \right.
\end{equation}
such that $\y(\t)$ is the solution to the \ac{IVP} \eqref{diff-system}.
However, it is not possible for almost all equations encountered in modeling problems to have the exact flow written explicitly in terms of elementary or even special functions. Therefore, numerical methods are to be applied in order to provide numerical approximations to exact flows \cite{book:tomas}.
These approximations are sought on a discrete set of points $\{\t_0,\t_1,\ldots,\t_n\} \subset \I$ with $\t_{i-1}<\t_{i}, \, \forall i\in\S{0}^n\ \equiv \{0,1,\ldots,n\}$ where $\y_{n}$  denotes the approximation to $\y(\t_n)$ and $\h_{n} \ \coloneqq\  \t_{n}-\t_{n-1}$ denotes the $n$\up{th} time step size after instant $\t_{n-1}$. Having equidistant points $\t_{n-i}$ with fixed time step $\h$ and approximations $\y_{n-i}$ of $\y(\t_{n-i})$ for $i\in\S{1}^\op\ \equiv \{1,\ldots,\op\}$ (as illustrated in the below figure),
\begin{figure}[ht]
 \centerline{
 \includegraphics[width=5.1cm]{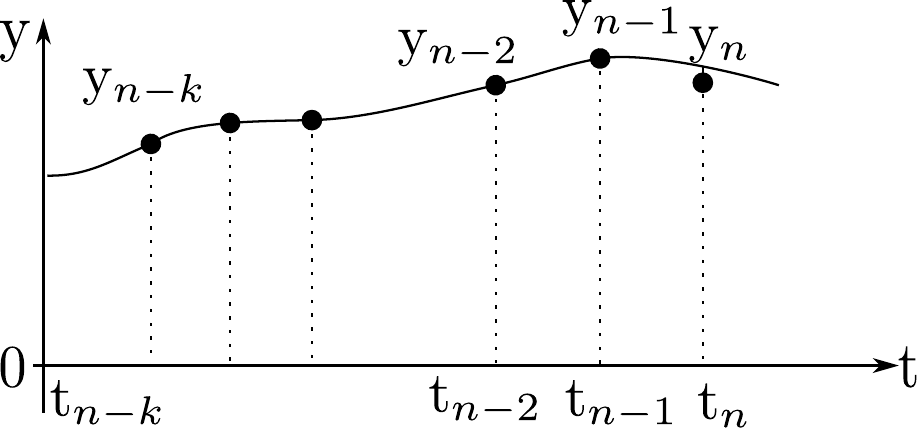}}
 \caption{Illustration of the interpolation polynomial $q$. }
\end{figure}
we present the general form of an \ac{LMS} method that computes $\y_{n}$ by solving the following equation \cite{henrici-62}:
\begin{equation}
\label{multistep-formula}
 \sum_{i=0}^{\op} \alpha_i \y_{n-i}   = \h\sum\limits_{j=0}^{\op}\beta_j \f(\t_{n-j},\y_{n-j}),
\end{equation}
where $\alpha_i$ and $\beta_i$ are coefficients defining the \ac{LMS} method. If $\beta_0\neq 0$, the scheme is said to be implicit. The generating polynomials are defined as follows:
\begin{equation}
 \label{generating-pol}
 \begin{array}{cc}
  \varrho(\zeta) \ \coloneqq\  \sum\limits_{i=0}^{\op}\alpha_i \zeta^i, \quad
  \sigma(\zeta) \ \coloneqq\  \sum\limits_{i=0}^{\op}\beta_i\zeta^i.
  \end{array}
\end{equation}
We have the following
\begin{theorem}
\label{thm1}{\cite[Theorem 2.4, pp 370]{book:hairer}}
An \ac{LMS} method defined by \cref{multistep-formula} is of order $\op$, if and only if one of the following conditions is satisfied:
 \begin{itemize}
  \item[i)] $\sum\limits_{j=0}^{\op}\alpha_j=0$ \text{and} $\sum\limits_{j=0}^{\op}\alpha_j j^m= m \sum\limits_{j=0}^{\op}\beta_j j^{m-1}$,  $\forall\,m\in \S{1}^{\op},$
  \item[ii)] $\varrho(\e^\h) - \h\sigma(\e^\h) =  \mathcal{O}(\h^{\op+1})$ for $\h \rightarrow 0,$
  \item[iii)] $\displaystyle\frac{\varrho(\zeta)}{\log(\zeta)} -\sigma(\zeta) = \mathcal{O}\big((\zeta-1)^\op\big)$ for $\zeta \rightarrow 1.$
 \end{itemize}
\end{theorem}
The second point in this Theorem is used later to prove that the proposed composition of a scheme of order $\op$ is precisely of order $\op+1$.

\subsection{\ac{BDFp}}
Assume we know the approximation of the solution on a set of points $\{ \t_{n-\op},\ldots, \t_{n-1}\}$, and we are looking to find an approximation $\y_{n}$ of $\y(\t_n)$. We consider the polynomial $q(\t)$ interpolating the points
$\big\{ (\t_i,\y_i) \,\arrowvert \, i\,\in \S{n-\op}^{n} \big\}$.

\subsubsection{Fixed time step}
This polynomial could be expressed in terms of backward differences
\begin{equation}
 \nabla^0\y_{n} = \y_{n}, \quad \nabla^{j+1}\y_{n} = \nabla^j\y_{n} - \nabla^j\y_{n-1},
\end{equation}
for all $\t\in [\t_{n-1},\t_n]$, such that $\t_n = \t_{n-1}+\h$, as follows (see \cite[pp 357]{book:hairer}):
\begin{equation}
\label{interpolation_difference}
 q(\t) \equiv q(\t_{n-1} + s\h) = \sum\limits_{j=0}^{\op} (-1)^j \binom{-s+1}{j}\nabla^j\y_{n}\,, \quad \forall s\in [0,1],
\end{equation}
where $\displaystyle \binom{s}{j} \ \coloneqq\ \frac{s!}{j!(s-j)!}$ is the binomial coefficient.
The implicit \ac{BDFp} consists of determining the unknown $\y_{n}$ such that the classical interpolation polynomial $q(\t)$ satisfies the following condition:
\begin{equation}
\label{cond_BDF}
 \left.\frac{\d }{\d \t}q(\t)\right\rvert_{\t=\t_n} = \f(\t_n,\y_{n}).
\end{equation}
Applying this condition to \cref{interpolation_difference}, we get the following relation:
\begin{equation}
 \sum\limits_{j=0}^{\op} \eta_j^* \nabla^j \y_{n} = \h\f(\t_n,\y_{n}), \quad \text{with } \eta_j^* \ \coloneqq\  (-1)^j \left.\frac{\d}{\d s}\binom{-s+1}{j}\right\rvert_{s=1},
\end{equation}
where $\left. (.)\right\rvert_{s=1}$ represents the evaluation operator at $s=1$. Though, we have $\eta_0^* = 0$ and  $\eta_j^* = 1/j$ for $j\in\S{1}^{\op}$. By replacing these coefficients and using difference formulas, we have:
\begin{align}
\label{coefficients-alpha_i}
\begin{aligned}
  \sum\limits_{j=1}^\op \frac{1}{j} \nabla^j \y_{n}
   &= \sum\limits_{j=1}^{\op} \frac{1}{j} \left[ \sum\limits_{i=0}^{j} (-1)^i \binom{j}{i}\y_{n-i} \right], \\
  &\text{by permuting the} \sum \\
  &=  \sum\limits_{i=0}^{\op} \left[ (-1)^i  \sum\limits_{j=i+1}^{\op}\frac{1}{j}\binom{j}{i}\right] \y_{n-i} = \sum_{i=0}^{\op} \gf{i} \y_{n-i}.
  \end{aligned}
\end{align}
We present in \cref{tab0} coefficients $\{\gf{i} \,\vert\, i\,\in \S{0}^{\op}\}$ for $\op\,\in\S{1}^{5}$.
\begin{table}[ht]
\begin{center}
\begin{minipage}{\textwidth}
\caption{Coefficients $\gf{i}$ of \ac{BDF} with fixed time step.}\label{tab0}%
\begin{tabular*}{\textwidth}{@{\extracolsep\fill}c?ccccccc@{\extracolsep\fill}}
\toprule
 $\op$ &  $\gf{0}$  & $\gf{1}$ & $\gf{2}$ & $\gf{3}$ &  $\gf{4}$ &  $\gf{5}$ \\
\midrule
 1 & 1 &-1\\ 
 2 & $\cfrac{3}{2}$ & -2 & $\cfrac{1}{2}$  \\ 
 3 & $\cfrac{11}{6}$ & -3 & $\cfrac{3}{2}$ & $-\cfrac{1}{3}$  \\ 
 4 & $\cfrac{25}{12}$ & -4 & 3 & $-\cfrac{4}{3}$ & $\cfrac{1}{4}$  \\ 
 5 & $\cfrac{137}{60}$ & -5 & 5 & $-\cfrac{10}{3}$ & $\cfrac{5}{4}$ & $ -\cfrac{1}{5}$  \\
\bottomrule
\end{tabular*}
\end{minipage}
\end{center}
\end{table}
One has explicitly $\{\gf{j}\,\vert\,j\in\S{0}^{\op}\}$, $\y_{n}$ is the solution of the following equation.
\begin{equation}
\label{BDFp}
 \sum_{j=0}^{\op} \gf{j} \y_{n-j}   = \h \f(\t_{n},\y_{n}).
\end{equation}
We present another way to find coefficients $\gf{i}$. It is based on the point $(ii)$ in \cref{thm1}. We consider a linear equation ($\f(\t,\y)\equiv \y$). Though all preceding solutions are given as a function of $\y_{n}$ such that: $\y_{n-j} = \y_{n} \e^{-j\h}$. Substituting $\y_{n-j}, \forall \,j\in\S{0}^{\op}$ in \cref{BDFp} allows us to write:
\begin{eqnarray*}
 \cancel{\y_{n}} \sum\limits_{j=0}^{\op} \gf{j}\e^{-j\h} = \h\,\cancel{\y_{n}} \quad\Longrightarrow
 \quad
 \sum\limits_{j=0}^\op \gf{j} \left( 1 + \sum\limits_{m=1}^{\infty}\frac{(-j\h)^m}{m!}\right) &=& \h,\\
 \sum\limits_{j=0}^\op \gf{j} +
 \h \sum\limits_{j=0}^\op-i\,\gf{j} +
 \sum\limits_{m=2}^{\infty} \frac{1}{m!}\left( \sum\limits_{j=0}^\op (-j)^m\gf{j} \right)\h^m &=& \h.
\end{eqnarray*}
To obtain the $\op$\up{th} order approximation, coefficients $\{\gf{j} \,\vert\, j \in \S{0}^{\op}\}$ must verify the following linear system:
\begin{equation}
\label{lin-sys-BDF-constant}
\left(
 \begin{array}{ccccc}
  1 & 1 & 1 & \ldots & 1\\
  0 &1 & 2 & \ldots & \op\\
  0 &1 & 2^2 & \ldots & \op^2\\
  \vdots &\vdots &\vdots &\ddots &\vdots \\
  0 &1 & 2^\op & \ldots & \op^\op
 \end{array}
\right) \cdot
\left(
\begin{array}{c}
 \gf{0} \\ \gf{1} \\  \gf{2} \\ \vdots \\ \gf{\op} \\
\end{array}
\right)
=
\left(
\begin{array}{c}
 0\\-1\\ 0 \\ \vdots\\ 0
\end{array}
\right).
\end{equation}
Solving the linear system produces coefficients for any order $\op$, including those were listed in \cref{tab0} for $\op \in \S{1}^{5}$. We end by presenting the difference operator, $L$, associated to a \ac{BDFp} with a fixed time step $\h$:
\begin{equation}
 \label{difference-operator}
 L\big(\y,\t,\h) \coloneqq \sum_{j=0}^{\op} \Big(\gf{j}\y(\t -j\h) - \h\frac{\d \y}{\d\t}(\t)  \Big).
\end{equation}
The difference operator $L$ plays a crucial role in determining the local error and order of convergence. If one can write the difference operator $L$ in its Taylor development as follows:
\begin{equation}
 \label{diff_oper_p}
 L\big(\y,\t,\h) = \sum_{j=0}^{\infty} \E_j \h^j\frac{\d^j \y}{\d\t^j}(\t),
\end{equation}
thus according to \cref{thm1} the scheme is of order $\op$ if coefficients $\E_j = 0$ for $j\in\S{0}^\op$. We can use Lemma 2 in \cite[page 369]{book:hairer} to find the following error using $\E_{\op+1}$:
\begin{equation}
 \y(\t_{n+1})-\y_{n+1} = \frac{1}{\g{0}}\E_{\op+1} \h^{\op+1}\frac{\d^{\op+1} \y}{\d\t^{\op+1}}(\t_n) + O(\h^{\op+2}).
\end{equation}
This will be of use in \cref{proof_theorem2}. In the next subsection, we present the procedure for computing coefficients $\g{j}$ in the case of a variable time step.

\subsubsection{Variable time stepping}
When having a variable step size $\h_n = \t_n-\t_{n-1}$, the polynomial $q(\t)$ that interpolates the set $\big\{(\t_i,\y_i) \lvert\, i\,\in \S{0}^{\op}\big\}$ can be expressed as follows (see \cite[pp 400]{book:hairer}):
\begin{equation}
\label{BDF-vts}
 q(\t) \equiv \sum\limits_{j=0}^{\op} \prod\limits_{i=0}^{j-1}(\t-\t_{n-i})\times \delta^j\y[\t_n,\t_{n-1},\ldots,\t_{n-j} ],
\end{equation}
where the divided differences $\delta^j\y[\t_n,\dots,\t_{n-j}]$ are defined recursively by:
\begin{equation*}
 \begin{array}{rcl}
  \delta^0\y[\t_n] &\coloneqq&\y_{n},\\[6pt]
  \delta^j\y[\t_n,..,\t_{n-j}] &\coloneqq& \displaystyle\frac{\delta^{j-1}\y[\t_n,..,\t_{n-j+1}] - \delta^{j-1}\y[\t_{n-1},..,\t_{n-j}]}{\t_n-\t_{n-j}} \equiv \sum\limits_{i=0}^j c_{i,j} \y_{n-i}.
 \end{array}
\end{equation*}
Applying condition \eqref{cond_BDF} on \cref{BDF-vts} leads us to the variable step size formula:
\begin{equation*}
 \sum\limits_{j=1}^{\op} \h_n\prod\limits_{i=1}^{j-1}(\t_n-\t_{n-i}) \times \delta^j \y[\t_n,\t_{n-1},\ldots,\t_{n-j} ] = \h_n \f(\t_n,\y_{n}).
\end{equation*}
Replacing the formula for $\delta^j\y[\t_n,\ldots,\t_{n-j}]$ in the above equation and permuting the sum leads us to find $\y_{n}$ that is the solution of:
\begin{equation}
\label{bdfk_adapt}
 \sum\limits_{j=0}^{\op} \g{j}\y_{n-j} = \h_n\f(\t_n,\y_{n}),
\end{equation}
where coefficients $\g{j}$ are given below :
\begin{equation}
\label{formula_gi_adapt}
\left\lbrace
 \begin{array}{rcl}
  \g{j} &=& \sum\limits_{m=j}^{\op} b_mc_{j,m}, \quad\text{with}\\
  b_m &=& \h_n\prod\limits_{\ell=1}^{m-1}(\t_n-\t_{n-\ell}), m \in\S{1}^{\op} ,\quad b_0=0, \quad \text{and}\\
  c_{j,m} &=& \prod\limits_{\substack{\ell=0\\ \ell\neq j}}^m \cfrac{1}{\t_{n-j} - \t_{n-\ell}},\quad j\in\S{0}^{\op},\quad m\in \S{\max(1,j)}^{\op}.
 \end{array}
 \right.
\end{equation}
\cref{alg:gi} in \ref{sec_app1}, labeled $\mathsf{Coeff}$, exhibits the computation of $\{\g{i}\,\vert\,i\in\S{0}^{\op}\}$ for any ordered set of points $\big\{\t_i\: \wedge\: \t_i<\t_{i+1}\,\vert\, i \in \S{n-\op}^{n}\: \wedge \: n\geqslant \op \big\}$.

\subsection{Fixed point algorithm}

In nonlinear equations, $\y_{n}$ is obtained by solving an iterative fixed point problem $\F_\op\big([\g{0},\ldots,\g{\op}],\y_{n})=\y_{n}$, with:
\begin{equation}
 \label{BDF_fpp}
 \F_\op\,\Bigl([\g{0},\ldots,\g{\op}],\y_{n}\Bigr) \ \coloneqq\  \frac{1}{\g{0}} \Big(\h_n\f(\t_n,\y_{n}) - \sum\limits_{i=1}^{\op} \g{i}\y_{n-i}\Big).
\end{equation}
The algorithm produces a series of approximations $\y_{n,\ell}$ to $\y(\t_n)$, though it requires the initialization $\y_{n,0}$, which is equal to $\y_{n-1}$ in most considered cases. This is represented by the following system:
\begin{equation}
\label{iter_FP}
 \left\lbrace
 \begin{array}{cl}
  \y_{n,0} &= \y_{n-1},\\
  \y_{n,\ell+1} &= \F_\op\big([\g{0},\ldots,\g{\op}],\y_{n,\ell}\big).
 \end{array}
 \right.
\end{equation}
The series of iterations is stopped when two consecutive approximations are close in norm to to a pre-defined user tolerance $\tol$ parameter:
\begin{equation*}
 \me_{n,\ell+1} =\|\y_{n,\ell+1} -\y_{n,\ell}\|< \tol.
\end{equation*}
The fixed point algorithm converges if the following inequality is verified when $\mathrm{x}$ belongs to a neighborhood of the exact solution $\y(\t_n)$:
$$\displaystyle\left\lvert\frac{\partial \F_\op}{\partial \y_n}(\mathrm{x})\right\rvert<1.$$
To accelerate the convergence rate, one can use Aitken's delta square scalar acceleration process \cite{ramiere-15,martinez-23}.

\subsection{Numerical flow}
\label{sec:nf}
We define the numerical flow associated with the \ac{BDFp}:
\begin{definition}\label{nf-bdf}
We define $\mT_{n-1,\op}\ \coloneqq\ (\t_{n-\op},\ldots,\t_{n-1})\in \mathds{R}^{\op}$ and $\Y_{n-1,\op} \ \coloneqq\  (\y_{n-\op},\ldots,\y_{n-1}) \in {(\mathds{R}^d)}^{\op}$. The numerical flow $\nf{\op}{\h}$ of \ac{BDFp} is defined below:
\begin{equation}
 \nf{\op}{\h_n}:
 \left\lvert
 \begin{array}{ccl}
   \mathds{R}^{\op}\times{(\mathds{R}^d)}^{\op}\ & \longrightarrow\ &  \mathds{R}^{\op}\times{(\mathds{R}^d)}^{\op},\\
  (\mT_{n-1,\op},\Y_{n-1,\op})\ &\longmapsto\ & (\mT_{n,\op},\Y_{n,\op}). 
 \end{array}
 \right.
\end{equation}
It has outputs $\mT_{n,\op}\ \coloneqq\ (\t_{n-\op+1},\ldots,\t_{n}) \in \mathds{R}^{\op}$ with $\t_n = \t_{n-1}+\h_n$ and $\Y_{n,\op}\ \coloneqq\ (\y_{n-\op+1},\ldots,\y_{n})$ with $\y_{n}$ solution to \cref{bdfk_adapt}.
\end{definition}
The numerical flow associated to \ac{BDFp} scheme is implemented in \cref{alg:BDFk} in \ref{sec_app2}, relying on Function $\mathsf{Coeff}$ schematized in \cref{alg:gi}. The purpose of presenting these formulations here is to prepare the ground for the proposed algorithm.

\section{Results/Findings}
\label{sec-nov}

Consider the vector $\mT_{n-1,\op}$ of $\op$ time instants where at every instant $\t_{n-i}$, $i \in \S{1}^{\op}$, there is an approximation of the solution of order $\op$.
\begin{figure}[!ht]
\centerline{
\setlength{\unitlength}{1cm}
 \begin{picture}(10,3)(1,0)
  \put(1,1.5){\vector(1,0){9}}
  \put(1.,0.){\vector(0,1){3}}
  \put(.8,3.05){$\Im(\t)$}
  \put(10,1.4){$\Re(\t)$}
  \put(1.5,1.5){\circle*{.1}}
  \put(1.2,1.2){$\t_{n-\op}$}
  \put(1.5,1.){\line(0,-1){.5}}
  \put(1.8,0.8){$\h_{n-\op+1}$}
  \put(1.5,.7){\line(1,0){1.5}}
  \put(3,1.){\line(0,-1){.5}}
  \put(3,1.5){\circle*{.1}}
  \put(2.7,1.2){$\t_{n-\op+1}$}
  \put(6,1.5){\circle*{.1}}
  \put(5.7,1.2){$\t_{n-1}$}
  \put(6,1.5){\vector(1,1){1}}
  \put(6,2.){\rotatebox{45}{$\a_1\h_n$}}
  \put(7,2.5){\circle*{.1}}
  \put(7.,2.7){$\tnmh$}
  \put(8.5,1.5){\circle*{.1}}
  \put(8.2,1.2){$\t_{n}$}
  \put(6,1.){\line(0,-1){.5}}
  \put(7.2,0.8){$\h_n$}
  \put(6,0.7){\line(1,0){2.5}}
  \put(8.5,1.){\line(0,-1){.5}}
 \end{picture}}
 \caption{Illustration of the composition technique in the complex plane.}
\end{figure}
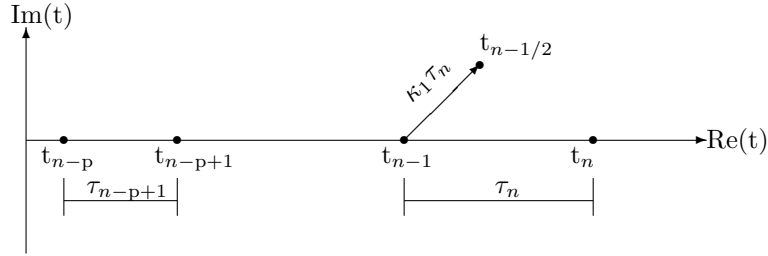
Finding an approximation $\y_n$ of $\y(\t_n)$ at $\t_n \equiv \t_{n-1}+\h_n$ by composing twice the numerical flow $\nf{\op}{}$ consists to advance first to an intermediate point between $\t_{n-1}$ and $\t_n$, denoted by $\tnmh\ \coloneqq\ \t_{n-1} + \a_1\h_n$. Here $\a_1$ is to be determined by solving an algebraic equation related to the set of points $\{t_{n-i}\,\vert\,i\in\S{1}^{\op}\}$.
We denote by $(\Tnmh{\op},\Ynmh{\op})$ the outputs of the numerical flow such that $\Ynmh{\op}$ contains $\ynmh$ to be an approximation of $\y(\tnmh)$.
\[\begin{cases}
(\Tnmh{\op},\Ynmh{\op})&\ \coloneqq\ \nf{\op}{\a_1\h_n}\big(\mT_{n-1,\op},\Y_{n-1,\op}\big),\\
 \Tnmh{\op}&\ \coloneqq\  (\t_{n-\op+1},\ldots,\t_{n-1},\tnmh),\\
\Ynmh{\op}&\ \coloneqq\  (\y_{n-\op+1},\ldots,\y_{n-1},\ynmh).
\end{cases}\]
We apply again the numerical flow $\nf{\op}{}$ having $(\Tnmh{\op},\Ynmh{\op})$ to be its inputs, and integrating with the time step ${\a_2\h_n}\ \coloneqq\ (1-\a_1)\h_n$. We denote by:
\[
\begin{cases}
\big(\mT^{\prime}_{n,\op},\Y^{\prime}_{n,\op} \big)&\ \coloneqq\  \nf{\op}{\a_2\h_n}\big(\Tnmh{\op},\Ynmh{\op} \big),\\
 \mT^{\prime}_{n,\op}&\ \coloneqq\  (\t_{n-\op+2},\ldots,\t_{n-1},\tnmh,\t_n),\\
 \Y^{\prime}_{n,\op}&\ \coloneqq\  (\y_{n-\op+2},\ldots,\y_{n-1},\ynmh,\hy_n).
\end{cases}\]
We also define the following ratios:
\begin{equation*}
 \begin{cases}
  \eps{j} &\ \coloneqq\  \displaystyle\frac{\tnmh - \t_{n-j}}{\tnmh - \t_{n-1}},\\
  \r{j} &\ \coloneqq\  \displaystyle\frac{\t_{n-1}-\t_{n-j}}{\t_n-\t_{n-1}},
 \end{cases}
\quad \forall j\in \S{1}^{\op}
\end{equation*}
and by using $\a_1 \coloneqq \displaystyle \frac{\tnmh-\t_{n-1}} {\t_n - \t_{n-1}}$, we have:
\begin{equation}
 \label{eps_i}
 \eps{j} \equiv 1+\frac{\r{j}}{\a_1}, \quad \forall j\in \S{1}^{\op}.
\end{equation}
We enunciate the following
\begin{theorem}
\label{comp_bdfk}
Consider having $\big(\mT_{n-1,\op},\Y_{n-1,\op} \big)$ and the numerical flow $\nf{\op}{}$ given in \cref{nf-bdf} in \cref{sec:nf}. If $\a_1$ and $\a_2 \in \mathds{C}$ verifying:
 \begin{equation}
 \label{sysa1a2}
 \begin{cases}
   \a_1+\a_2 = 1,\\
     \eps{\op}  \a_1^2 +\g{0}\a_2^2=0
 \end{cases}
 \end{equation}
with $\g{0} \ \coloneqq\  \displaystyle\sum\limits_{i=1}^{\op}\frac{1}{\eps{i}}$
and if
$$\y(\t_{n-i}) - \y_{n-i} = \mathcal{O}(\h_{n-i}^{\op+1})\quad \forall i\,\in\,\S{1}^{\op},$$
then the following composition:
\begin{equation}
\label{composition_BDFk}
 \nf{\op}{\a_2\h_n}\circ \nf{\op}{\a_1\h_n}\big(\mT_{n-1,\op},\Y_{n-1,\op} \big)
\end{equation}
produces the couple $(\mT^{\prime}_{n,\op},\Y^{\prime}_{n,\op})$ with $\hy_n$, in $\Y^{\prime}_{n,\op}$, approximating $\y(\t_n)$ at $\t_n = \t_{n-1}+\h_n$ such that:
\begin{equation}
\label{error_Rehyn}
 \y(\t_n) - \Re(\hy_n) = \mathcal{O}(\h_n^{\op+2}).
\end{equation}
In addition, we have the following error estimation:
 \begin{equation}
  \label{error-estimate}
 \left\rvert \y(\t_n) - \Re(\hy_n)\right\rvert \sim \eC_\op\times \left\rvert \Im(\hy_n)\right\rvert
 \end{equation}
 with $\eC_\op$ is to be determined in \cref{error_constant}.
\end{theorem}

This theorem shows that the composition employs complex arithmetic operations and \textbf{produces $\hy_n$ such that its real part approximates the solution with an additional order of accuracy to the basic numerical flow}. One chooses the next time step $\h_n = \t_{n}-\t_{n-1}$, \cref{comp_bdfk} shows that $\a_1$ is the solution of an algebraic equation depending on coefficients $\r{j}, j\in\S{1}^{\op}$. It provides also in $\Im(\hy_n)$ an error estimate of order $\op+1$ to the approximation $\Re(\hy_n)$.
The proof of this theorem will be given in \cref{proof_theorem2}. We present in \cref{nf_def} the composition, denoted by $\nfc{\op+1}{}$, in a compact form.

\begin{definition}
\label{nf_def}
Consider the \ac{IVP} defined by \eqref{diff-system} and having the numerical flow $\nf{\op}{}$ associated to a \ac{BDFp} scheme. Consider having the first $\op$ approximations $\y_0,\ldots,\y_{\op+1}$ to $\y(\t_0),\ldots, \y(\t_{\op+1})$. We define the new (composed) numerical flow by:
 \begin{equation}
  \label{nf_comp}
\nfc{\op+1}{\h_n}:\left\lvert
 \begin{array}{ccl}
   \mathds{R}^{\op}\times(\mathds{R}^{d})^{\op} & \rightarrow &  \mathds{R}^{\op}\times(\mathds{R}^{d})^{\op}\\
  (\mT_{n-1,\op},\Y_{n-1,\op}) &\mapsto & (\mT_{n,\op},\Y_{n,\op}), 
 \end{array}
 \right.
 \end{equation}
 where $\mT_{n,\op}\ \coloneqq\ (\t_{n-\op+1},\ldots,\t_n)$, $\Y_{n,\op}\ \coloneqq\ (\y_{n-\op+1},\ldots,\y_{n-1},\y_{n})$ such that $\y_{n} = \Re(\hy_n)$ with $\hy_n$ obtained by the composition in \cref{comp_bdfk}.
\end{definition}
We present the process of composition in \cref{alg:BDFk_BDFk}. After calculating  coefficients $\r{j}$ and elaborating algebraic expressions $\eps{\op}$ and $\g{0}$ that are functions of $\a_1$, roots $\a_1$ and $\a_2$ are to be determined by solving the algebraic system \eqref{sysa1a2} with an appropriate software ({\sf Maple, Python, Matlab, Mathematica, ...}). After choosing the complex root with a positive real part, the composition is established, and the real part of $\hy_n$ is known to approximate $\y(\t_n)$. Once the approximation is obtained by the composition, we show that the imaginary part $\Im(\hy_n)$ is an error estimate of the produced approximation as stated in the last point of Theorem \ref{comp_bdfk}. Regarding the linear stability of the composition, we have the following result.
\begin{theorem}
\label{thm3}
 The numerical flow $\nfc{\op}{}$ is $\mA(\vartheta$)-stable up to order $\op=8$ with stability angles reported in \cref{tab4}.
\end{theorem}

\begin{algorithm}[!ht]
 \caption{Adaptive numerical simulation for \eqref{diff-system} using $\nfc{\op+1}{}$\label{alg-adap}}
\begin{algorithmic}
\Require $\f$, $\t_0$, $\y_0$, $\T{}$, $\h$, $\tol$, $\op$
\For{$i\gets 1$ to $\op-1$}
\State $\h_i \gets \h$
\State $\t_i\gets \t_{i-1} +\h_i$
\State Construct $\y_i$, approximation to $\y(\t_i)$ of order $\op$
\Comment{(See    \cite{deeb:part1})}
\EndFor
\State $\h_n \gets \h$
\State $\t_n \gets \t_{\op-1} +\h_n$
\While{$\t_n\leqslant \T{}$}
\State $\mT_{n-1,\op}\gets [\t_{n-\op},\ldots,\t_{n-1}] $
\State $\Y_{n-1,\op}\gets [\y_{n-\op},\ldots,\y_{n-1}]$
\State $(\mT_{n,\op},\Y_{n,\op})\gets \nfc{\op+1}{\h_n}(\mT_{n-1,\op},\Y_{n-1,\op})$
\State $\me_n \gets \eC_\op \times \Im(\hy_n)$ 
\State $\h_{n+1} \gets \h_n \times  \left(\displaystyle\frac{\tol}{|\me_n|}\right)^{\frac{1}{\op+2}}$
\Comment{(update the time step)}
\State $\h_{n+1} \gets \min\left( \max\left( \h_{n+1}, \frac{\h_n}{\ell_\op} \right) , \h_n\times \ell_\op \right)$
\Comment{($\ell_\op$ in \cref{bounding_timestep})}
\State $\t_{n+1} \gets \t_n + \h_{n+1}$
\State $n\gets n+1$
\EndWhile
\end{algorithmic}
\end{algorithm}
The proof of \cref{thm3} is done graphically in \cref{sec-stab}. To this end, we present the full outcome of this work in \cref{alg-adap}. The Algorithm presents the process of computing approximations using the composed flow with \ac{ATS}.
The loop \textbf{for} aims to initialize processing $\nfc{\op+1}{}$ by finding approximations of $\y(\t_j)$ of order $\op$ with $\t_j = \t_0 + j \h$, for $j\in \S{1}^{\op-1}$. It could be done using the composition proposed in Part I \cite{deeb:part1}.
The time step is updated regarding the rate between the user tolerance $\tol$ and the imaginary part representing an error estimate. A bound limit on the ratio between two consecutive time steps is applied to ensure the validity of the time marching.
We will show by numerical experiments in \cref{sec-conv} the convergence error of the composition and in \cref{sec-num} that the imaginary part is a reliable estimation of the local error.

\section{Procedure of composing \ac{BDFp}}
\label{sec-comp}

In this section, having the numerical flow $\nf{\op}{}$ of \ac{BDFp} scheme, we establish the composition procedure generating $\nfc{\op+1}{}$ to increase by one the order of approximation. We first generate coefficients $\{\g{i}\,\vert\,i\in \S{0}^{\op}\}$, using the last $\op$ points $\{\t_{n-i}\,\vert\,i\in\S{1}^{\op}\}$ and $\tnmh\ \coloneqq\ \t_{n-1}+\a_1\h_n$ to integrate the equation by the numerical flow $\nf{\op}{\a_1\h_n}$, though finding approximation $\ynmh$ to $\y(\tnmh)$. Then, we generate coefficients $\{\G{i}\,\vert\,i\in\S{0}^{\op+1}\}$ using the last $\op+1$ points $\{\t_{n-i}\,\wedge\, \tnmh\,\vert\,i\in\S{1}^{\op}\}$ and $\t_n \equiv\tnmh+(1-\a_1)\h_n $ to integrate with $\nf{\op+1}{\a_2\h_n}$ finding  $\y_n$. We select the step $\a_1\h_n$ such that it cancels $\y_{n-\op}$ in the second integration, \emph{i.e,} the coefficient $\G{\op+1}$ in front of $\y_{n-\op}$ must vanish.

\subsection{Intermediate solution: first step}
First, \ac{BDFp} is used to compute an intermediate approximation $\ynmh$ of $\y(\tnmh)$. For that, we solve the following equation:
\begin{equation}
\label{bdf_ynmh}
 \g{0} \ynmh + \sum\limits_{i=1}^{\op}\g{i}\y_{n-i} = (\tnmh-\t_{n-1})\f(\tnmh,\ynmh)
\end{equation}
that is represented by the associated numerical flow:
\begin{equation}
\label{first_flow}
 \Big(\Tnmh{\op},\Ynmh{\op}\Big) = \nf{\op}{\a_1\h_n} \Big(\mT_{n-1,\op},\Y_{n-1,\op}\Big),
\end{equation}
where $\Tnmh{\op}$
and $\Ynmh{\op}$ are defined in \cref{sec-nov}.
For every coefficient $\a_1$ and time step $\h_n$, coefficients $\{\g{i}\,\vert\,i\in\S{0}^{\op}\}$ could be obtained using the technique presented in \cite[Appendix G.4]{book:tomas}. Consider a linear differential equation $\f(\t,\y)=\y$ with $\y(\t) = \ynmh\, \e^{-(\tnmh-t)}$ being its exact analytical solution. Consider also $\y_{n-i}, i\in\S{1}^{\op}$ are known exactly; $\y_{n-i} = \ynmh\,\e^{-(\tnmh - t_{n-i})}, i\in\S{1}^{\op}$, thus, they are to be substituted in \cref{bdf_ynmh}. After simplifying by $\ynmh$, we obtain:
\begin{equation*}
 \g{0} + \sum\limits_{i=1}^{\op} \g{i} \e^{-(\tnmh-\t_{n-i})} = \tnmh-\t_{n-1}.
\end{equation*}
Writing the Taylor series expansion of $ \e^{-(\tnmh-\t_{n-i})}$ in the above equation and equating terms up to order $\op$, we obtain the following linear system generating the solution for unknowns $\g{i}$:
\begin{equation}
 \label{linear_system_gi}
  \MA_{n,\op}\cdot
 \left(
 \begin{array}{c}
\g{0}\\
 \g{1}\\
 \g{2}\\
\vdots\\
\g{\op}
\end{array}
 \right)
 =
 \left(
 \begin{array}{c}
0\\
 -1\\
 0\\
\vdots\\
0
 \end{array}
 \right),
\,\,
\MA_{n,\op}\ \coloneqq\
 \left(
 \begin{array}{ccccc}
  1&1&1& \ldots &1\\
  0&\eps{1} & \eps{2} &\ldots &\eps{\op}\\
  0&\eps{1}^2 & \eps{2}^2 &\ldots &\eps{\op}^2\\
  \vdots & \vdots& \vdots &\vdots &\vdots\\
  0&\eps{1}^{\op} & \eps{2}^{\op} &\ldots &\eps{\op}^{\op}
 \end{array}
 \right),
 \end{equation}
where $\{\eps{i}\,\vert\,i\in\S{1}^{\op}\}$ are already defined in \cref{eps_i}.
System \eqref{linear_system_gi} can be solved symbolically using Cramer's rule and the determinant formulas of Vandermonde matrices. We have:
\begin{equation}
\label{formula_gi}
\g{i} =
\left\lbrace
 \begin{array}{cl}
    \sum\limits_{j=1}^{\op} [\eps{j}]^{-1} ,& \text{if }  i=0,\\
    (-1)^{\op}[\eps{i}]^{-1}\prod\limits_{\substack{j=1\\j\neq i}}^{\op}\eps{j} [\eps{i}-\eps{j}]^{-1}, & \text{if }i\in\S{1}^{\op}.
 \end{array}
\right.
\end{equation}
Therefore, $\ynmh$, in the case of a linear equation, is given by:
\begin{equation}
\label{formula-ynmh}
 \ynmh = \frac{1}{(\tnmh-\t_{n-1}) - \g{0}}\sum\limits_{i=1}^{\op} \g{i}\y_{n-i},
\end{equation}
For now, let us chose $\a_1 \in \mathds{C}$  such that $0<\Re({\a_1})<1$ and $-1<\Im(\a_1)<1$.
Before going to the second integration, we establish the following Lemma:
\begin{lemma}
\label{lemma_1}
Taking $\eps{j}$ defined by \cref{eps_i} and $\g{j}$ solution to System \eqref{linear_system_gi}, we have the following identity:
 \begin{equation}
  \sum\limits_{j=1}^{\op} \eps{j}^{\op+1} \g{j} = (-1)^{\op}\prod\limits_{j=1}^{\op} \eps{j}.
\end{equation}
\end{lemma}

\begin{proof}
First, we replace formulas of $\g{j}$ from \cref{formula_gi} in the above formula:
\begin{eqnarray*}
 \sum\limits_{j=1}^{\op} \eps{j}^{\op+1} \g{j} &=&(-1)^{\op}\sum\limits_{j=1}^{\op} \eps{j}^{\op+1}[\eps{j}]^{-1}\prod\limits_{\substack{m=1\\m\neq j}}^{\op}\eps{m} [\eps{j}-\eps{m}]^{-1},
\end{eqnarray*}
then, we factorize by $\prod\limits_{j=1}^{\op}\eps{j}$:
\begin{eqnarray*}
 &=& (-1)^{\op}\prod\limits_{j=1}^{\op}\eps{j} \sum\limits_{j=1}^{\op} \frac{\eps{j}^{\op-1}}{\prod\limits_{\substack{m=1\\m\neq j}}^{\op}[\eps{j}-\eps{m}]}.
 \end{eqnarray*}
 We now take the common denominator in the sum to get:
 \begin{eqnarray*}
 \sum\limits_{j=1}^{\op} \eps{j}^{\op+1} \g{j}&=& (-1)^{\op}\prod\limits_{j=1}^{\op}\eps{j} \frac{\displaystyle \sum\limits_{j=1}^{\op} (-1)^{j-1}\eps{j}^{\op-1}\prod\limits_{\substack{m=1\\m\neq j}}^{\op-1}\prod\limits_{\substack{\ell=m+1\\ \ell\neq j}}^{\op}[\eps{m}-\eps{\ell}]}{\displaystyle \prod\limits_{m=1}^{\op-1}\prod\limits_{\ell=m+1}^{\op}[\eps{m}-\eps{\ell}]}\\
 &=& (-1)^{\op}\prod\limits_{j=1}^{\op}\eps{j},
\end{eqnarray*}
which is true because of
\begin{equation*}
{\displaystyle \sum\limits_{j=1}^{\op}(-1)^{j-1}\eps{j}^{\op-1}\prod\limits_{\substack{m=1\\m\neq j}}^{\op-1}\prod\limits_{\substack{\ell=m+1\\ \ell\neq j}}^{\op}[\eps{m}-\eps{\ell}]}={\displaystyle \prod\limits_{m=1}^{\op-1}\prod\limits_{\ell=m+1}^{\op}[\eps{m}-\eps{\ell}]}.
\end{equation*}
The last identity can be proven by induction.
\end{proof}

\subsection{Second jump}
\label{sec-secnd-jump}
After computing $\ynmh$, we have $\op+1$ points with $\y_{n-\op},\ldots,\y_{n-1}$. We define:
\begin{equation*}
\t_n = \tnmh + (1-\a_1) \h_n \equiv  \t_{n-1}+\h_n,
\end{equation*}
and having the scheme of \ac{BDF}($\op+1$) as in \cref{nf-bdf}, we apply its numerical flow as follows:
\begin{equation*}
 \Big(\mT^{\prime}_{n,\op+1},\Y^{\prime}_{n,\op+1}\Big) = \nf{\op+1}{(1-\a_1)\h_n} \Big(\Tnmh{\op+1},\Ynmh{\op+1}\Big).
\end{equation*}
Here, $\Tnmh{\op+1} = (\t_{n-\op},\ldots,\t_{n-1},\tnmh)$ and $\mT^{\prime}_{n,\op+1}=(\t_{n-\op+1},\ldots,\t_{n-1},\tnmh,\t_n)$ $\in~\mathds{R}^{\op+1}$, $\Ynmh{\op+1} = (\y_{n-\op},\ldots,\y_{n-1},\ynmh) $ and $\Y^{\prime}_{n,\op+1}=(\y_{n-\op+1},\ldots,\y_{n-1},\ynmh,\hy_n)$ $\in (\mathds{R}^{d})^{\op+1}$,
with $\hy_n$ solution to the following equation:
\begin{equation}
\label{bdf_yn}
 \G{0} \hy_n + \G{1} \ynmh + \sum\limits_{i=1}^{\op} \G{i+1}\y_{n-i} = (\t_{n}-\tnmh) \f(\t_n,\hy_n).
\end{equation}
To find coefficients $\{\G{i}\,\vert\,i\in\S{0}^{\op+1}\}$ that fulfill the $(\op+1)$\up{th} order conditions, we proceed in the same strategy as before and consider the linear equation $\f(\t,\y) = \y$, where $\ynmh$ and $\y_{n-i}$ for $i\in\S{1}^{\op}$ are presented below:
\begin{align*}
\ynmh &= \frac{\hy_n \sum\limits_{i=1}^{\op} \g{i} \e^{-(\t_n-\t_{n-i})}} {(\tnmh-\t_{n-1}) - \g{0}},&
\y_{n-i}& = \hy_n \e^{-(\t_n-\t_{n-i})},\quad i\in\S{1}^{\op}.
\end{align*}

Replacing all the above elements in \cref{bdf_yn} and simplifying by $\hy_n$, leads us to the following relation:
\begin{equation*}
 \G{0} + \G{1}\frac{ \sum\limits_{i=1}^{\op} \g{i} \e^{-(\t_n-\t_{n-i})}} {(\tnmh-\t_{n-1}) - \g{0}} + \sum\limits_{i=1}^{\op} \G{i+1} \e^{-(\t_n-\t_{n-i})} = \t_{n}-\tnmh.
\end{equation*}
Now, multiplying by $[(\tnmh-\t_{n-1}) - \g{0}]$ on both sides and writing the Taylor series expansion of $\e^{-(\t_n-\t_{n-i})}$, gives us the following equation:
\begin{align*}
 -\G{0}\g{0} + \G{0}(\tnmh-\t_{n-1})\hspace{4cm} \\
 + \left[\G{1}\g{1} - \G{2}\g{0} + \G{2} (\tnmh-\t_{n-1})\right] \left[  \sum\limits_{i=0}^{}\frac{(-1)^i}{i!}(\t_n-\t_{n-1})^i\right]  \\
 +\quad \cdots\hspace{11cm} \\
 + \left[\G{1}\g{\op} - \G{\op+1}\g{0} + \G{\op+1} (\tnmh-\t_{n-1})\right]  \left[  \sum\limits_{i=0}^{}\frac{(-1)^i}{i!}(\t_n-\t_{n-\op})^i\right] \\
 =-\g{0}(\t_n-\tnmh) + (\t_n-\tnmh) (\tnmh-\t_{n-1}).\hspace{3cm}
\end{align*}
To simplify notations, we define:
\begin{equation}
\label{Eps_i}
 \Eps{i} \ \coloneqq\  \frac{\t_n - \t_{n-i}}{\h_n} = 1-\a_1 + \a_1\eps{i} = 1+\r{i}, \quad i\in\S{1}^{\op}.
\end{equation}
Though, $(\t_n - \t_{n-i}) \equiv \Eps{i}\h_n$ for $i\in\S{1}^{\op}$. We recall that :
\begin{align*}
(\tnmh-\t_{n-1}) &= \a_1\h_n, & (\t_n-\tnmh) &= (1-\a_1)\h_n
\end{align*}
and we use them to rearrange the Taylor expansion as follows:
\begin{eqnarray}
&&\G{1}\Big(\sum\limits_{i=1}^{\op}\g{i}\Big)  -\Big(\G{0}+\sum\limits_{i=2}^{\op+1}\G{i}\Big)\g{0}\nonumber\\
&&+\Big(\a_1\G{0}- \G{1}\sum\limits_{i=1}^{\op}\Eps{i}\g{i}+\sum\limits_{i=2}^{\op+1}\left[\begin{array}{c}\a_1\\+\g{0}\Eps{i-1}\end{array}\right]\G{i} \Big)\h_n  \nonumber\\
 &&+\sum\limits_{j=2}^{\infty}(-1)^{j-1}\left(- \G{1}\sum\limits_{i=1}^{\op}\Eps{i}^j\g{i}+\sum\limits_{i=2}^{\op+1}\left[\begin{array}{c}j\a_1\Eps{i-1}^{j-1}\\+\g{0}\Eps{i-1}^j\end{array}\right]\G{i}  \right)\frac{\h_n^j}{j!}  \nonumber\\
\label{taylor-expan}
&&=-\g{0}(1-\a_1) \,\h_n +\a_1(1-\a_1)\,{\h_n^2}.
\end{eqnarray}
To simplify this equation, we state the following
\begin{proposition}
\label{lemma_2}
Define $\Eps{0} \ \coloneqq\ 1-\a_1\equiv \a_2$ and take $\Eps{i}$ that are defined in \cref{Eps_i} and $\{\g{i}\,\vert\,i\in \S{0}^{\op}\}$ as a solution to System \eqref{linear_system_gi}. Then,
\begin{equation}
 \label{simplif1}
 \begin{array}{cr}
  -\sum\limits_{i=1}^{\op} \Eps{i}^{j} \g{i} = \Eps{0}^{j-1}\Big(j\a_1 + \Eps{0}\g{0}\Big),&  \quad \forall j\in\S{1}^{\op}.
 \end{array}
\end{equation}
In addition, the following identity holds:
 \begin{equation}
 \label{simplif2}
  -\sum\limits_{i=1}^{\op} \Eps{i}^{\op+1} \g{i} = \Eps{0}^{\op}\Big((\op+1)\a_1 + \Eps{0}\g{0}\Big) +(-1)^{\op+1}\a_1^{\op+1}\prod\limits_{i=1}^{\op} \eps{i}.
\end{equation}
\end{proposition}
\begin{proof}
 The proof of the first equality is obtained by replacing $\Eps{i}$ by \cref{Eps_i} and using binomial formulas with System \eqref{linear_system_gi}. The second equality is obtained also using System \eqref{linear_system_gi} and the use of \cref{lemma_1}.
\end{proof}

By using the relation $\sum\limits_{i=1}^{\op}\g{i}=-\g{0}$, we simplify the first term relative to the 0\up{th} rank ($\h_n^0$) in \eqref{taylor-expan}, we obtain:
\begin{equation}
\label{eqG0}
\sum\limits_{i=0}^{\op+1} \G{i} =0.
\end{equation}
Now, we use \cref{eqG0} to simplify the term relative to $\h_n$ in \eqref{taylor-expan} to obtain:
\begin{equation}
 \label{eqG1}
 \g{0}\sum\limits_{i=0}^{\op}\Eps{i}\G{i+1} = -\g{0}\Eps{0},
\end{equation}
which will be used, with \cref{simplif1} in \cref{lemma_2} for $j=1$, to simplify the term associated to $\h_n^2$. We continue like this for all terms multiplying $\h_n^j$ for $j \in \S{3}^{\op}$ to get
\begin{equation}
 \label{eqGj}
 \g{0}\sum\limits_{i=0}^{\op}\Eps{i}^j\G{i+1} = 0, \quad \forall j\in \S{2}^{\op}.
 \end{equation}
We end by using \cref{simplif2} to simplify the term multiplying $\h_n^{\op+1}$ and \cref{eqGj} for $j=\op$. The linear system generating coefficients $\{\G{i}\,\vert\,i\in\S{0}^{\op+1}\}$ is presented below:
\begin{equation}
 \label{linear_system_Gi}
 \MAp_{n,\op+1}\cdot \left(\G{0},\ldots,\G{\op+1}\right)^\top = \left(0,-\Eps{0},0,\ldots,0\right)^\top,
\end{equation}
where $\MAp_{n,\op+1} \in \Mat{\op+2}{\mathds{C}}$ is the matrix given below:
\begin{equation}
\label{Apnk}
\MAp_{n,\op+1}\ \coloneqq \
 \left(
 \begin{array}{ccccc}
  1&1&1& \ldots &1\\[8pt]
  0 & \Eps{0} &\Eps{1}&\ldots &  \Eps{\op}\\
    \vdots & \vdots& \vdots &\ddots &\vdots\\
  0 &\Eps{0}^{\op} &  \Eps{1}^{\op} &\ldots & \Eps{\op}^{\op}\\[8pt]
  0 &\left[\begin{array}{c}  \Eps{0}^{\op+1}+\\ \frac{(-\a_1)^{\op+1}}{\g{0}}\prod\limits_{i=1}^{\op} \eps{i}\end{array} \right] & \left[ \Eps{1}^{\op+1}\right] &\ldots & \left[\Eps{\op}^{\op+1}\right]\\
 \end{array}
 \right)
 \cdot
 \end{equation}
After assembling the linear system relative to coefficients $\{\G{i}\,\vert\,i\in\S{0}^{\op+1}\}$, we proceed to establish its solution.

\subsection{Solution to System \eqref{linear_system_Gi}}
To find the solution to the linear system associated with coefficients $\{\G{i}\,\vert\,i\in\S{0}^{\op+1}\}$, the Cramer's rule will be used. We have the following.
\begin{proposition}\label{prop1}
 Having defined $\eps{i}$ for $i\in \S{1}^{\op}$ and $\Eps{i}$ for $i\in \S{0}^{\op}$, the expression of $\{\G{i}\,\vert\,i\in\S{0}^{\op}\}$, solution to System \eqref{linear_system_Gi} are explicitly given by:
 \begin{align}
 \begin{aligned}\label{form_Gkp1}
 \G{0} &=
   \displaystyle\frac{(\g{0}-2\a_1)^{\op-1} \Big((\g{0}-2\a_1)\Eps{0}\g{0} \big(\sum\limits_{i=0}^{\op}\frac{1}{\Eps{i}}\big) - \a_1\big(\sum\limits_{i=1}^{\op}\frac{1}{\Eps{i}}\big)\Big)}{\Big( \Eps{0}\g{0}-\a_1\Big)},\\
   \G{1} &= -\displaystyle \frac{\g{0}\Eps{0}}{\big( \Eps{0}\g{0}-\a_1\big)} \prod\limits_{i=1}^{\op}{\frac{1+\r{i}}{\a_1 + \r{i}}}, \\
   \G{j+1} &=  \frac{(-1)^{j+1}\prod\limits_{\substack{i=1\\i\neq j}}^{\op}\Eps{i}}{\prod\limits_{\ell=1}^{j-1}(\Eps{j}-\Eps{\ell})} \frac{(\a_1-1)\Big[\Eps{0}^2\g{0} + \a_1^2 \eps{j}\Big]}{\Eps{j}\big(\r{j}+\a_1\big)\big[\Eps{0}\g{0}-\a_1\big]},\quad j\in\S{1}^{\op-1},\\
  \G{\op+1} &= (-1)^{\op+1} \prod\limits_{i=1}^{\op-1}\left[\frac{\Eps{i}}{\Eps{\op}-\Eps{i}}\right] \frac{(\a_1-1)\Big[\Eps{0}^2\g{0} + \a_1^2 \eps{\op}\Big]}{\Eps{\op}\big(\r{\op}+\a_1\big)\big[\Eps{0}\g{0}-\a_1\big]}.
  \end{aligned}
 \end{align}
\end{proposition}
Despite the fact that the proof should present the computation of all coefficients, we are interested in providing the procedure for $\G{\op+1}$ for what we want to cancel it by a suitable $\a_1$. The proof is presented in \ref{sec_app3}. Although the proof for other coefficients follows the same pattern as presented for $\G{\op+1}$. This proposition will allow us to generate, for any order $\op$, the algebraic formulas of $\{\G{i}\,\vert\,i\in\S{0}^{\op+1}\}$. The explicit formula of $\G{\op+1}$ for $\op\in\S{1}^{3}$ is presented in \ref{sec_app4}.

\section{Proof of \cref{comp_bdfk}}
\label{proof_theorem2}
In this Section, we state the elements that give us the $(\op+2)$\up{th} order of the approximation obtained by the twice composition of $\nf{\op}{}$ presented in \cref{sec-nov}.
It is based on the use of the difference operator defined by \cref{difference-operator} but in the adaptive version. As we start in the composition by finding $\ynmh$ using \cref{bdf_ynmh}, the associate difference operator $L_1$ is given below:
\begin{equation}
 L_1(\y,\t_n,\h_n) = \g{0}\y(\tnmh) + \sum\limits_{i=1}^{\op} \g{i}\y(\t_{n-i}) - \a_1\h_n \left.\frac{\d}{\d\t} \y(\t)\right\rvert_{\t=\tnmh}.
\end{equation}
After writing the Taylor series development of $\y(\tnmh)$, its derivative and $\y(\t_{n-i})$ in the vicinity of $\t_n$, as follows:
\begin{align}
\label{TSE_yni}
 \begin{aligned}
  \y(\tnmh) &=&  \sum\limits_{j=0}^{\infty} \frac{(-1)^j}{j!} \Eps{0}^j \h_n^j \left.\frac{\d^j}{\d\t^j} \y(\t)\right\rvert_{\t=\t_n}, \\
   \left.\frac{\d}{\d\t} \y(\t)\right\rvert_{\t=\tnmh} &=&  \sum\limits_{j=0}^{\infty} \frac{(-1)^j}{j!} \Eps{0}^j \h_n^j \left.\frac{\d^{j+1}}{\d\t^{j+1}} \y(\t)\right\rvert_{\t=\t_n} ,\\
  \y(\t_{n-i}) &=& \sum\limits_{j=0}^{\infty} \frac{(-1)^j}{j!} \Eps{i}^j \h_n^j \left.\frac{\d^j}{\d\t^j} \y(\t)\right\rvert_{\t=\t_n},
 \end{aligned}
\end{align}
we can express the difference operator $L_1$ in its Taylor development in $\t_n$:
\begin{equation}
\label{line-diff-bdfk}
 L_1(\y,\t_n,\h_n) = \sum\limits_{j=0}^{\infty}\E_j \h_n^j \left.\frac{\d^j}{\d\t^j} \y(\t)\right\rvert_{\t=\t_n},
\end{equation}
where
\begin{equation}
\E_j  \coloneqq \left\lbrace
\begin{array}{lr}
\sum\limits_{i=0}^{\op} \g{i}, & j=0\\
  \frac{(-1)^j}{j!}  \left(
 \sum\limits_{i=0}^{\op} \g{i} \Eps{i}^j + j\Eps{0}^{j-1} \a_1 \right) & j\geqslant 1.
 \end{array}
 \right.
\end{equation}
Using the relation between coefficients $\g{i}$ in system \eqref{linear_system_gi}, we have $\E_0=0$, and with \cref{lemma_2} we have $\E_j = 0, \forall j \in \S{1}^{\op}$, This is another way to see the $\op$\up{th} order of the scheme used in the first jump. It is also clear form \cref{lemma_2} that:
\begin{equation}
\label{Ek1}
 \E_{\op+1} = -\frac{\a_1^{\op+1}}{(\op+1)!}\prod\limits_{j=1}^{\op}\eps{j},
\end{equation}
which will be used to evaluate the error, according to formula (2.9) in \cite[page 372]{book:hairer}, as follows:
\begin{equation}
 \label{error_firstjump}
 \y(\tnmh) - \ynmh = \frac{1}{\g{0}}\E_{\op+1} \h_n^{\op+1} \left.\frac{\d^{\op+1}}{\d\t^{\op+1}} \y(\t)\right\rvert_{\t=\t_n} + O(\h_n^{\op+2}).
\end{equation}

After computing the approximation $\ynmh$, the formula of the difference operator regarding the \ac{BDFp} scheme used in the second jump is written as follows:
\begin{equation}
 L_2(\y,\t_n,\h_n) = \Gp{0} \y(\t_n) + \Gp{1}\ynmh + \sum\limits_{i=2}^{\op}\Gp{i}\y(\t_{n-i+1}) - \a_2 \h_n  \left.\frac{\d}{\d\t} \y(\t)\right\rvert_{\t=\t_n}.
\end{equation}
Using \cref{error_firstjump} we write $\ynmh$ in function of $y(\ynmh)$ and using \cref{TSE_yni}, we replace $\y(\tnmh)$ and $\y(\t_{n-i+1})$ in the above equation by their Taylor series expansion in the vicinity of $\t_n$. We write the difference operator $L_2$ in its Taylor development in $\t_n$ as follows:
\begin{equation}
\label{line-diff-bdfk1}
 L_2(\y,\t,\h_n) = \sum\limits_{j=0}^{\infty}\mE_j \h_n^j\left.\frac{\d^j}{\d\t^j} \y(\t)\right\rvert_{\t=\t_n}
\end{equation}
where the first $\op+1$ terms are given below:
\begin{align}
 \label{mE_j}
 \mE_j \coloneqq
 \left\lbrace
 \begin{aligned}
\sum\limits_{i=0}^\op \Gp{i} &,& j=0,\\
-\a_2 - \sum\limits_{i=0}^{\op-1} \Gp{i+1} \Eps{i}  &,& j=1,\\
\frac{(-1)^j}{j!}\sum\limits_{i=0}^{\op-1} \Gp{i+1} \Eps{i}^j  &,& j \in \S{2}^\op,\\
\frac{(-1)^j}{j!}\sum\limits_{i=0}^{\op-1} \Gp{i+1} \Eps{i}^j -\frac{\Gp{1}}{\g{0}}\E_{j} &,& j=\op +1,
 \end{aligned}
 \right.
\end{align}
By the fact that the second jump should be at least of order $\op$, we should have $\mE_j = 0$ for all $j\in \S{0}^{\op}$. This will allow us to find the formulas of coefficients $\{\Gp{i}\,\vert\,i\in\S{0}^{\op}\}$. Having $\a_2 \equiv \Eps{0}$, coefficients $\{\Gp{i}\,\vert\,i\in\S{0}^{\op}\}$ verify the system defined by the matrix ${\MAp}^{\op+2,\op+2}_{n,\op+1}$: the matrix resulting from deleting the last row and column from $\MAp_{n,\op+1}$, with the right-hand side obtained after omitting the last row of that in System \eqref{linear_system_Gi}. Using \cref{sysa1a2} and the following lemma, we show that
\begin{equation}
\label{EQ51}
\Gp{i} =\G{i}, \quad \forall i\in\S{0}^{\op}.
\end{equation}

\begin{lemma}
 \label{lemma_3}
Consider one having a linear system $\MA \sX=b$, where $\sX \ \coloneqq\  (\sx_1,\ldots,\sx_{\op+1})^\top$ and $\MA\in \GL{\op+1}{\mathds{C}}$.
If one defines a new linear system $\MA^{j,j}\sY=b^j$ where $\sY\ \coloneqq\ (\sy_1,\ldots,\sy_{\op})^\top$
and $b^{j}$ results from deleting the $j$\up{th} component from $b$. If $\sx_j = 0$, then
 \begin{equation*}
  \sy_i\equiv
  \begin{cases}
   \sx_i, & i\in\S{1}^{j-1},\\
   \sx_{i+1}, & i\in\S{j+1}^{\op}.
  \end{cases}
 \end{equation*}
\end{lemma}
\cref{lemma_3} is obtained after omitting the term relative to $\sx_j$ in the original system and removing the $j$\up{th} equation. Back to the linear system generating $\{\Gp{i}\,\vert\,i\in\S{0}^{\op}\}$: it results after deleting the last row and last column from System \eqref{linear_system_Gi}. Having relation \eqref{sysa1a2}, thus $\G{\op+1}=0$ and we conclude \cref{EQ51}. To prove the order $\op+1$, we have to show that
$$\mE_{\op+1} = 0 ,$$
which is verified in the last equation in system \eqref{linear_system_Gi} after replacing $\E_{\op+1}$ by its formula in \cref{Ek1}.
To show the error constant, we have the following asymptotic formula:
\begin{equation}
\label{error-approx}
 \y(\t_n) - \hy_n = \frac{1}{\G{0}} \mE_{\op+2}\h_n^{\op+2} \left.\frac{\d^{\op+2}}{\d\t^{\op+2}} \y(\t)\right\rvert_{\t=\t_n} + \mathcal{O}(\h_n^{\op+3}).
\end{equation}
where:
\begin{align}
 \label{mEk2}
 \begin{aligned}
 \mE_{\op+2} &=\frac{(-1)^{\op+2}}{(\op+2)!} \sum\limits_{i=0}^{\op} \G{i+1}\Eps{i}^j  - \frac{\G{1}}{\g{0}}\E_{\op+2},\\
 &=\frac{(-1)^{\op+2}}{(\op+2)!} \left[\sum\limits_{i=0}^{\op} \left( \G{i+1} - \frac{\G{1}}{\g{0}} \g{i} \right) \Eps{i}^{\op+2}  + (\op+2)\a_1 \Eps{0}^{\op+1}\right].
\end{aligned}
\end{align}

In the composed flow, we consider that $\y_n = \Re(\hy_n)$, which means that the error is approximated by the real part of the right-hand side of \cref{error-approx} with:
\begin{equation}
 \Re\big(\y(\t_n)-\hy_n\big) = \y(\t_n)-\y_n \simeq \Re\left(\frac{\mE_{\op+2}}{\G{0}}\right)\h_n^{\op+2}\left.\frac{\d^{\op+2}}{\d\t^{\op+2}} \y(\t)\right\rvert_{\t=\t_n}.
\end{equation}
Taking the imaginary part of \cref{error-approx}, we reach the second conclusion of \cref{comp_bdfk} showing that:
\begin{equation}
 \y(\t_n)-\y_n \simeq \eC_\op \times \Im(\hy_n),
\end{equation}
with
\begin{equation}
\label{error_constant}
\eC_\op \coloneqq \Re\left(\frac{\mE_{\op+2}}{\G{0}}\right)\times \left[\Im\left(\frac{\mE_{\op+2}}{\G{0}} \right)\right]^{-1}.
\end{equation}
We end it with presenting \cref{alg:BDFk_BDFk} generating the composed flow $\nfc{\op+1}{\h_n}$.
\begin{algorithm}[!ht]
\caption{The composed numerical flow $\nfc{\op+1}{\h_n} (\mT_{n-1,\op},\Y_{n-1,\op})$}\label{alg:BDFk_BDFk}
\begin{algorithmic}
\Require $\h_n$
\Require $\t_n\gets \t_{n-1}+\h_n$
\ForAll{$j\gets 1 $ to $\op$}
    \State $\r{j}\gets\displaystyle \frac{\t_{n-1}-\t_{n-j}}{\t_n-\t_{n-1}}$
\EndFor
\Function{$\eps{\op}$}{$\a_1$}
    \State \Return $\displaystyle 1+ \frac{\r{\op}}{\a_1}$
\EndFunction
\Function{$\g{0}$}{$\a_1$}
    \State \Return  $\displaystyle \sum\limits_{j=1}^{\op} \frac{\a_1}{\a_1+\r{j}}$
\EndFunction
\State $\a_1 \gets$ \textsf{Find Roots} $[\a_1\longmapsto(1-\a_1)^2\g{0} + \a_1^2 \eps{\op}=0]$
\Ensure $\Re(\a_1)>0$
\State $(\Tnmh{\op},\Ynmh{\op}) \gets \nf{\op}{\a_1\h_n} (\mT_{n-1,\op},\Y_{n-1,\op})$
\State $(\mT^{\prime}_{n,\op},\Y^{\prime}_{n,\op}) \gets
 \nf{\op}{(1-\a_1)\h_n}\big(\Tnmh{\op},\Ynmh{\op})$
 \State \Return $(\mT_{n,\op},\Y_{n,\op}), \Im(\hy_n)$
\end{algorithmic}
\end{algorithm}

\section{Convergence error and CPU for fixed time steps}
\label{sec-conv}
To show graphically the convergence order of $\nfc{\op+1}{}$, we consider the case of a fixed time step ($\h_n=\h \Rightarrow  \r{j} = j-1)$. Thus, the solution of $\G{\op+1}$, denoted here by $\Gf{\op+1}$, as it does not depend on instant $\t_n$, is given for all $\op\in\S{1}^{4}$:

\begin{align*}
\Gf{1+1}&=\frac{(\a_1 - 1)(2\a_1^2 - 2\a_1 + 1)}{\a_1(2\a_1 - 1)},\\
\Gf{2+1} & = -\frac{(\a_1-1)(3\a_1^3 - \a_1^2 + \a_1 + 1)}{2(\a_1 + 1)(3\a_1^2 - 1)},\\
\Gf{3+1} & = \frac{(\a_1 - 1)(4\a_1^4 + 5\a_1^3 + \a_1^2 + 6\a_1 + 2)}{6(2\a_1 + 1)(\a_1 + 2)(\a_1^2 + \a_1 - 1)}, \\
 \Gf{4+1} &= -\frac{(\a_1-1)(5\a_1^5 + 19\a_1^4 + 19\a_1^3 + 19\a_1^2 + 28\a_1 + 6)}{4(\a_1 + 3)(5\a_1^4 + 20\a_1^3 + 15\a_1^2 - 10\a_1 - 6)}.
\end{align*}
When the function \textsf{Find Roots} comes, we use a computer algebra software. We exhibit below the explicit expression of $\a_1$, solution to $\Gf{\op+1}=0$, when $\op\in\{1,2,3\}$, and a numerical approximation for its solution when $\op=4$ (for the sake of simplicity):
\begin{align*}
 \a_1 &= \frac{1}{2}\pm \frac{\mi}{2}; &  \op&=1\\
 \a_1 & = \left(\frac{-b}{18} + \frac{4}{9b} +\frac{1}{9} \right) \pm {\mi\sqrt{3}}\left( \frac{b}{18} + \frac{4}{9b}\right); & \op&=2\\
 &\approx0.4013648789516588 \pm \mi\times 0.7409710153124752\\
 \a_1&=\frac{d}{48 c}-\frac{5}{16} \pm \frac{\mi}{2}\sqrt{\frac{2439 c}{32d} + c^2 + \frac{7}{144c^2} -\frac{43}{96}}; & \op&=3\\
 &\approx0.3247753916537674 \pm \mi\times 0.927940112670109\\
 \a_1 &\approx 0.2675589068337956 \pm \mi\times 1.088573443182903; & \op &=4
\end{align*}
with
\begin{align*}
 b &\ \coloneqq\ \left( \frac{2\sqrt{19}}{\sqrt{3}} - {134}\right)^{1/3},
 \\c &\ \coloneqq\  \frac{1}{2}\left( 2\sqrt{\frac{9931}{6}} + \frac{2179}{27}\right)^{1/6},
& d &\ \coloneqq\  \sqrt{576 c^4 + 129c^2 + 28}.
\end{align*}

\subsection{Global error}
We solve numerically the \ac{IVP} \eqref{diff-system} with $\f\big(\t,\y\big) = -\y^3$ over the interval $]0,1[$, with $\y(0) =1$ using a \ac{BDFp} scheme and its two compositions when $\op\in\S{1}^{4}$ for different time steps $\h={1}/{\N}, \N\in\{10,20,40,80,160\}$.
Knowing the exact solution $\y(\t) \equiv \sqrt{\frac{1}{1+2\t}}$, we use it to build approximations $\y_j$ on $\t_j = j\h, j\in\S{1}^{\op-1}$ as they are crucial for initializing $\nf{\op}{}$ and $\nfc{\op+1}{\h}$. The part of \textsf{Find Roots} in \cref{alg:BDFk_BDFk} is skipped at every time step. \cref{alg:BDFk_BDFk} is then simplified by the last five lines: we apply first $\nf{\op}{(1-\a_1)\h}$ then $\nf{\op}{\a_1\h}$ to produce $\hy_n$ and take its real part to store in $\y_n$ at every time step.
Here, the exact solution is also used to compute the exact error:
\begin{equation}
 \label{error_nfc}
 \me(\t_n) \ \coloneqq\ \| \y(\t_n)-\y_n\|,\quad n\in\S{1}^{\N}.
\end{equation}
The global error is computed by using the trapezoidal rule of integration as follows:
\begin{equation}
 \label{global_error_nfc}
 \E_{\N} \coloneqq \frac{1}{\N}\left(  \sum\limits_{n=\op}^{\N-1} \me(\t_n) +\frac{\me(\t_{\N})}{2}\right).
\end{equation}
For distinction, we denote by $\E_{\N}^{\nf{\op}{}}$ the global error associated to the approximation produced by the \ac{BDFp} scheme while we denote by $\E_{\N}^{\nfc{\op}{}}$ the one associated to the approximation produced by the composed flow of order $\op$.
We present in \cref{fig1} the comparison between $\E_{\N}^{\nf{\op}{}}$ and $\E_{\N}^{\nfc{\op+1}{}}$.
\begin{figure}[!ht]%
\centering
\includegraphics[width=0.4\textwidth]{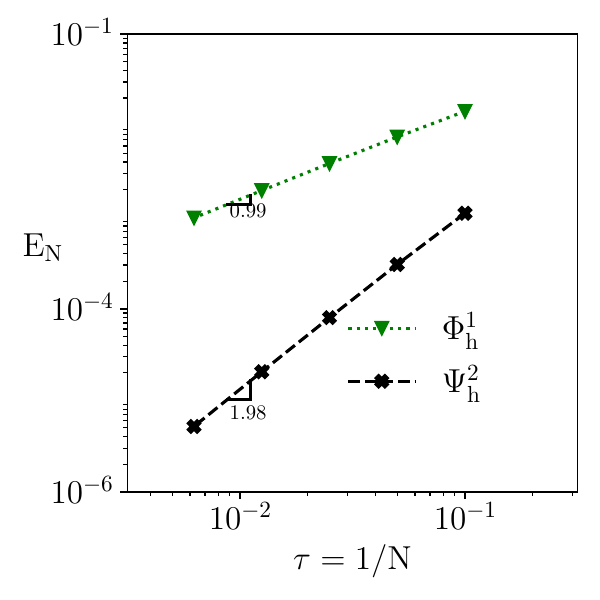}
\includegraphics[width=0.4\textwidth]{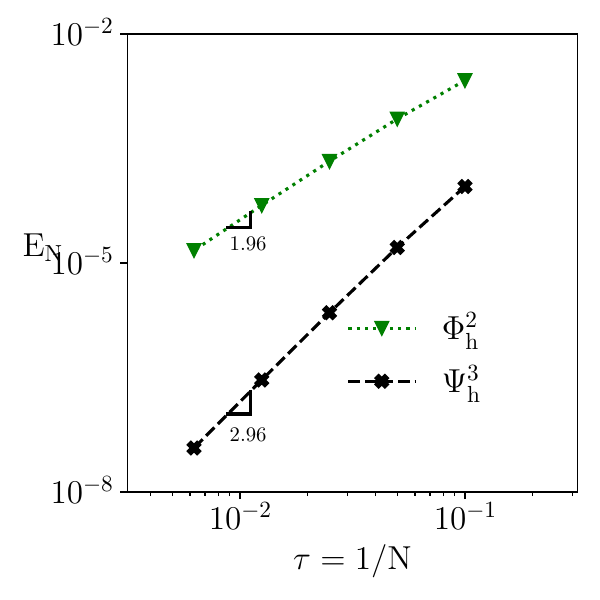}
\includegraphics[width=0.4\textwidth]{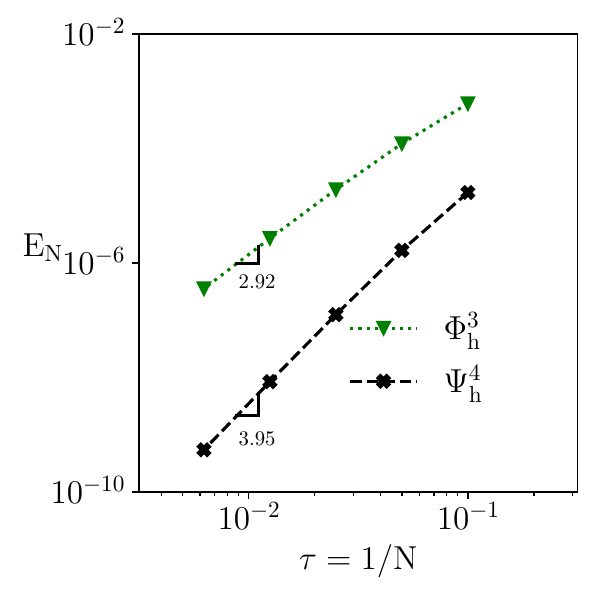}
\includegraphics[width=0.4\textwidth]{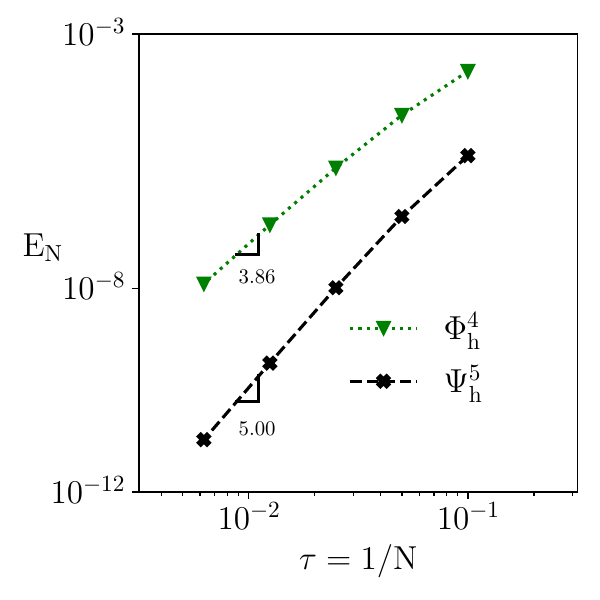}
\caption{Convergence errors of $\nf{\op}{}$ and its composed flow $\nfc{\op+1}{}$ for $\op\in\{1,2,3,4\}$.}\label{fig1}
\end{figure}

We mention that the slope for $\nf{2}{}$, the second order \ac{BDF} scheme, is $\simeq1.96$ while its twice composition $\nfc{3}{}$ designed by \cref{alg:BDFk_BDFk} has a slope of $\simeq2.96$, showing an increase by one order of convergence. Other convergence rates for \ac{BDFp} schemes and their twice compositions can be checked in the same figure. We conclude that the composition designed in this work increases by one the accuracy order. To check the utility of this composition in using less number of backward points while producing an approximation with the same order, we present in \cref{tab4_1} the global error between the classical \ac{BDFp} scheme using $\op$ backwards points and the proposed composed flow having the same order $\op$ while using $\op-1$ backward points. The comparison is done for $\op =2,3,4,5$.

\begin{table}[!ht]
\begin{center}
\begin{minipage}{\textwidth}
\caption{Comparison between the global error $\E_\N^{\nf{\op}{\h}}$ and $\E_\N^{\nfc{\op}{\h}}$ for different values $\h$.}\label{tab4_1}
\begin{tabular*}{\textwidth}{@{\extracolsep{\fill}}l?ccccc@{\extracolsep{\fill}}}
\toprule%
$\h$ & $10^{-1}$    &   $5\times10^{-2}$ & $2.5\times10^{-2}$  & $1.25\times10^{-2}$ & $6.25\times10^{-3}$ \\
\midrule
$\E_\N^{\nf{2}{}}$ & $2.46\times10^{-3}$ & $7.73\times10^{-4}$ & $2.15\times10^{-4}$ & $5.68\times10^{-5}$ & $1.45\times10^{-5}$ \\
$\E_\N^{\nfc{2}{}}$ & $1.10\times10^{-3}$ & $3.04\times10^{-4}$  & $7.99\times10^{-5}$ & $2.04\times10^{-5}$ &  $5.18\times10^{-6}$  \\
\midrule
$\E_\N^{\nf{3}{}} $ & $6.08\times10^{-4}$ & $1.21\times10^{-4}$ & $1.91\times10^{-5}$ & $2.68\times10^{-6}$ & $3.56\times10^{-7}$ \\
$\E_\N^{\nfc{3}{}}$ & $1.00\times10^{-4}$ & $1.59\times10^{-5}$ & $2.22\times10^{-6}$ & $2.93\times10^{-7}$ & $3.75\times10^{-8}$  \\
\midrule
$\E_\N^{\nf{4}{}} $ & $1.85\times10^{-4}$  & $2.53\times10^{-5}$ & $2.33\times10^{-6}$ & $1.78\times10^{-7}$ & $1.22\times10^{-8}$ \\
$\E_\N^{\nfc{4}{}}$ & $1.70\times10^{-5}$& $1.66\times10^{-6}$ & $1.24\times10^{-7}$ & $8.41\times10^{-9}$ & $5.43\times10^{-10}$ \\
\midrule
$\E_\N^{\nf{5}{}} $&$6.41\times10^{-5}$& $6.46\times10^{-6}$& $3.60\times10^{-7}$& $1.51\times10^{-8}$& $5.52\times10^{-10}$\\
$\E_\N^{\nfc{5}{}}$ &$4.06\times10^{-6}$& $2.58\times10^{-7}$& $1.02\times10^{-8}$& $3.39\times10^{-10}$& $1.06\times10^{-11}$\\
\bottomrule
\end{tabular*}
\end{minipage}
\end{center}
\end{table}

We conclude from \cref{tab4_1} that the composed flow of order $\op$, $\nfc{\op}{}$, has better accuracy than the one of \ac{BDFp}, $\nf{\op}{}$, with the same order, which means that it has a smaller error constants in the error estimation. This enables us to produce a higher-quality approximation while using a smaller number of backward points.
But how will this increase in order affect the computational cost involving real and complex parts? Does the composed numerical flow $\nfc{\op}{}$ have a low performance compared to the numerical flow of the \ac{BDFp} having the same order $\nf{\op}{}$?

\subsection{CPU comparison}

To check the computational efficiency of the composition, whether or not involving computation in real and complex parts will significantly increase the computational cost, we compare in \cref{tab4_2} the ratio of the global errors and the ratio of CPU time between $\nf{\op}{}$ and $\nfc{\op}{}$ for a set of time steps:
\begin{align}
 \label{ratio_CPU_err}
 \R_{\E_\N} & = \frac{\E_\N^{\nf{\op}{}}}{\E_\N^{\nfc{\op}{}}},& \R_{\rm CPU} &= \frac{\rm CPU^{\nf{\op}{}}}{\rm CPU^{\nfc{\op}{}}}
\end{align}

As for instance, the global error of $\nfc{\op}{}$ is smaller with a factor of $2.234$ compared to $\nf{\op}{}$ when $\op=2$ and $\h = 10^{-1}$ and the rate can go up to a factor of $52$ when $\op=5$ and $\h = 6.25\times10^{-3}$, as presented in \cref{tab4_2}. This increase in accuracy comes at a computational cost, as it involves computations in both real and complex parts. To check this, we present in the same table (\cref{tab4_2}) the CPU ratio $\R_{\rm CPU}$. We can see that, for $\op=2$ and $\h=0.1$, the decrease in the error by a factor of $2$ using $\nfc{2}{}$ implies an increase in CPU time by almost the same factor. This factor decreases when $\h$ decreases. For higher order $\op$ and smaller time steps, the increasing factor of CPU decreases and almost tends to 1, meaning that the computational time is the same for a given time step, while the error decreases faster, allowing higher precision going up to a factor of $52$ when $\op=5$ and $\h=0.00625$.
\begin{table}[!ht]
\begin{center}
\begin{minipage}{\textwidth}
\caption{Ratio of the global error and the CPU time between $\nf{\op}{\h}$ and $\nfc{\op}{\h}$.}\label{tab4_2}
\begin{tabular*}{\textwidth}{@{\extracolsep{\fill}}rl?ccccc@{\extracolsep{\fill}}}
\toprule%
&$\h$ & $10^{-1}$    &   $5\times10^{-2}$ & $2.5\times10^{-2}$  & $1.25\times10^{-2}$ & $6.25\times10^{-3}$ \\
\midrule
\multirow{2}{*}{$\op=2$} &$\R_{\E_\N}$ &2.234 & 2.539& 2.695& 2.775& 2.816 \\
&$\R_{\rm CPU}$& ${1}/{2.343}$&  ${1}/{2.295}$& ${1}/{2.291}$& ${1}/{2.242}$& ${1}/{1.625}$\\
\midrule
\multirow{2}{*}{$\op=3$} &$\R_{\E_\N}$&6.0582 & 7.629& 8.607& 9.172& 9.480 \\
&$\R_{\rm CPU}$& ${1}/{2.225}$&  ${1}/{2.037}$& ${1}/{1.915}$& ${1}/{2.103}$& ${1}/{1.968}$\\
\midrule
\multirow{2}{*}{$\op=4$} &$\R_{\E_\N}$&10.890 & 15.266& 18.734& 21.150& 22.626 \\
&$\R_{\rm CPU}$& ${1}/{1.906}$&  ${1}/{1.689}$& ${1}/{1.486}$& ${1}/{1.558}$& ${1}/{1.647}$\\
\midrule
\multirow{2}{*}{$\op=5$} &$\R_{\E_\N}$&15.773 & 24.966& 35.055& 44.662& 52.073 \\
&$\R_{\rm CPU}$& ${1}/{1.680}$&  ${1}/{1.372}$& ${1}/{1.321}$& ${1}/{1.328}$& ${1}/{1.168}$\\
\bottomrule
\end{tabular*}
\end{minipage}
\end{center}
\end{table}

We present in \cref{fig1_1} the plot of CPU time versus the associated global precision for both the classical \ac{BDFp} scheme, $\nf{\op}{}$, and the proposed composed flow $\nfc{\op}{}$ having the same order. Note that the composed flow with order $\op$ prescribes $\op-1$ backward points. We can see that the computational efficiency of the composed flow is the same when $\op=3$, allowing us to reach the same accuracy with fewer backward points, and it is even better for higher orders.
\begin{figure}[!ht]%
\centering
\includegraphics[width=0.4\textwidth]{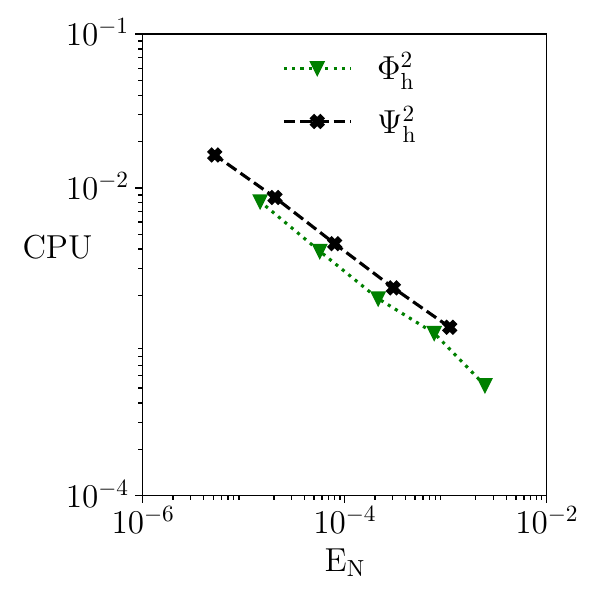}
\includegraphics[width=0.4\textwidth]{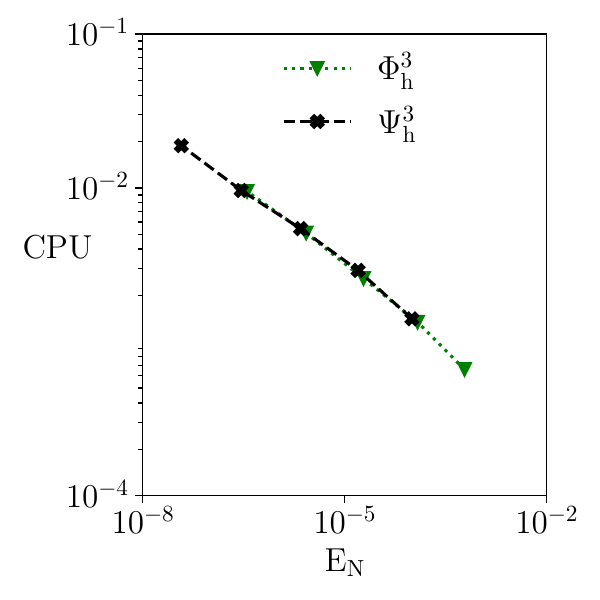}
\includegraphics[width=0.4\textwidth]{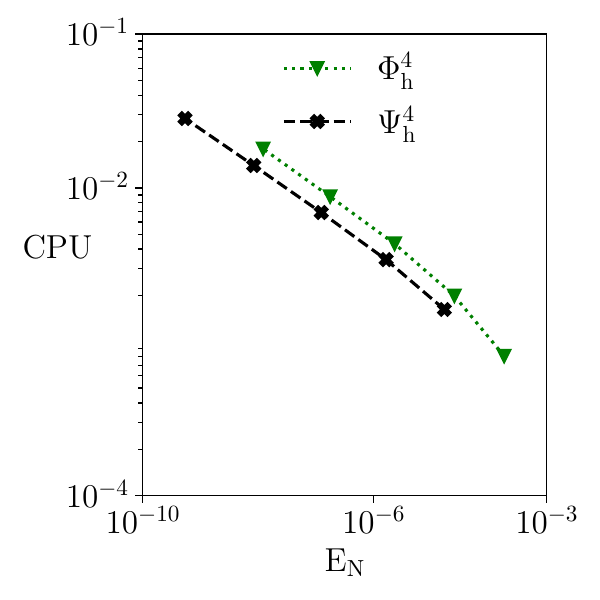}
\includegraphics[width=0.4\textwidth]{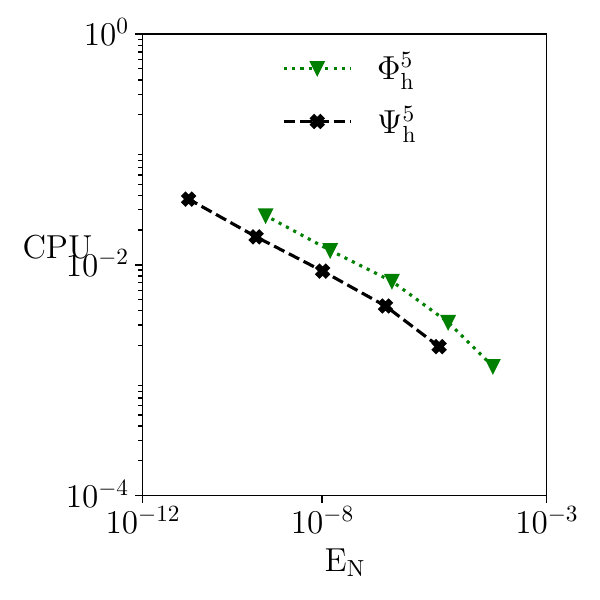}
\caption{CPU versus global errors of $\nf{\op}{}$ and the composed flow $\nfc{\op}{}$ for $\op\in\{2,3,4,5\}$.}\label{fig1_1}
\end{figure}
In the next Section, we investigate the linear stability of $\nfc{\op+1}{}$ and see how the barrier of order six of \ac{BDF} schemes is broken.

\section{Linear stability analysis}
\label{sec-stab}
The study of the stability of a linear multi-step method is conducted through its generating polynomials, as defined in \cref{generating-pol}. However, it is difficult to show the generating polynomials of the designed numerical method in this work. This is because an iterative solver is used to obtain an intermediate solution. This is why we will focus solely on the $\mA-$stability. To do that, we consider the simplest linear scalar problem $\dot{\y} = \lambda \y$, where the numerical integration using $\nfc{\op+1}{\h}$ leads us to the following relation:
\begin{equation*}
\sum\limits_{i=0}^{\op} \Theta_i\y_{n-i}=0
\end{equation*}
with
\begin{equation*}
 \Theta_i=:
 \left\lbrace
 \begin{array}{ll}
 (\a_1z - \g{0}).(\Gp{0} - (1 - \a_1)z),&i=0,\\
 \Gp{1}\g{i} + (\a_1z - \g{0})\Gp{i+1},&i\in\S{1}^{\op-1},\\
 \Gp{1}\g{\op}, &i = \op,
 \end{array}
 \right.
\end{equation*}
where $z\ \coloneqq\ \h\lambda$. We define the polynomial $p_\op(\omega,z)$ associated to $\nfc{\op+1}{}$ by the following formula:
\begin{equation}
 p_\op\big(\omega,z\big) = \sum\limits_{i=0}^{\op} \Theta_{\op-i}\omega^{i}.
\end{equation}

Consider $\omega_1(z),\ldots,\omega_\op(z)$ the $\op$ roots of $p_\op\big(\omega,z\big)$ that depend on $z$. The stability region of the numerical scheme $\nfc{\op+1}{}$ is the following region defined in the complex plane:
\begin{equation}
 \mathcal{D}_{\nfc{\op+1}{}} = \left\lbrace  z\in \mathds{C}\, \Big\rvert \,\lvert \omega_i(z)\rvert \leqslant1, \forall\,i\in\S{1}^{\op}   \right\rbrace.
\end{equation}
According to Lemma 4.7 in \cite{iserles_2008}, we have the following:
\begin{definition}
The composed scheme $\nfc{\op+1}{}$ is $\mA(\vartheta)-$stable if:
$$S_{\vartheta} \subseteq \mathcal{D}_{\nfc{\op+1}{}},$$
where $S_{\vartheta}$ is a sector of angle $2\vartheta$ bisected by the negative real axis. In addition, if $\vartheta=90^\circ$, the composed scheme is said to be $\mA-$stable.
\end{definition}
For every $\forall \,\op\in\S{1}^{8}$, the polynomial $p_\op\big(\omega,z\big)$ is constructed, then a mesh-grid of $z$ values is built in the complex plane, where for every value $z_{i,j}$, roots $\omega_{i,j,s}, s\in\S{1}^{\op}$ of $p_\op\big(w,z_{i,j}\big)=0$ are obtained numerically using the \textsf{NumPy} package in \textsf{Python} in order to check those with absolute value that is smaller than one. The domain of stability is the set of all values $z_{i,j}$, where $\lvert \omega_{i,j,s}\rvert \leqslant1$.

\begin{figure}[!ht]%
\centering
\includegraphics[width=0.45\textwidth]{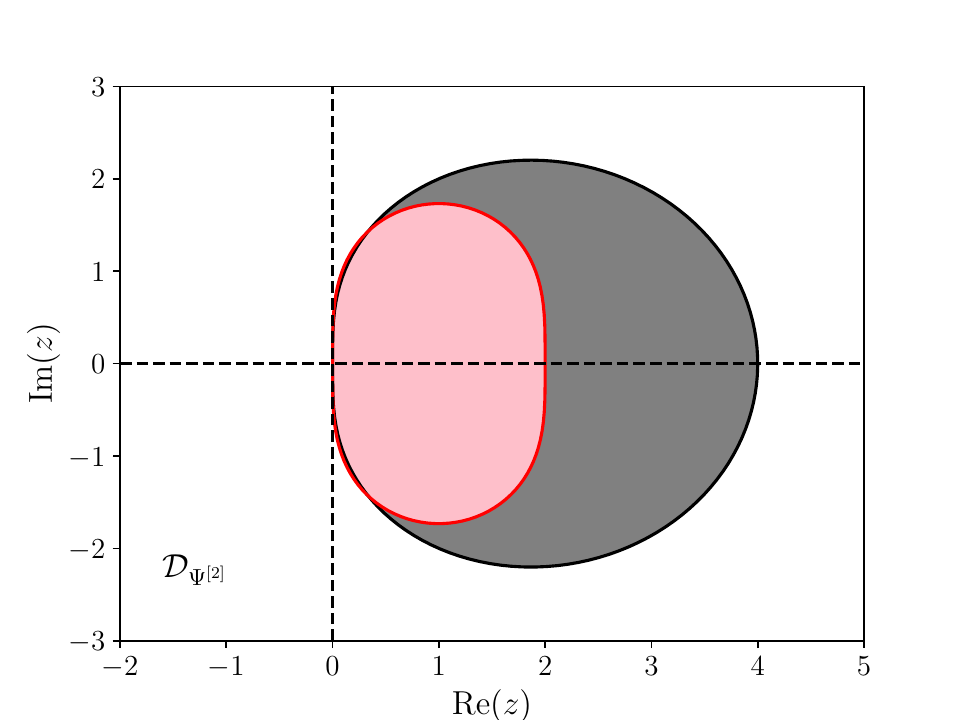}
\includegraphics[width=0.45\textwidth]{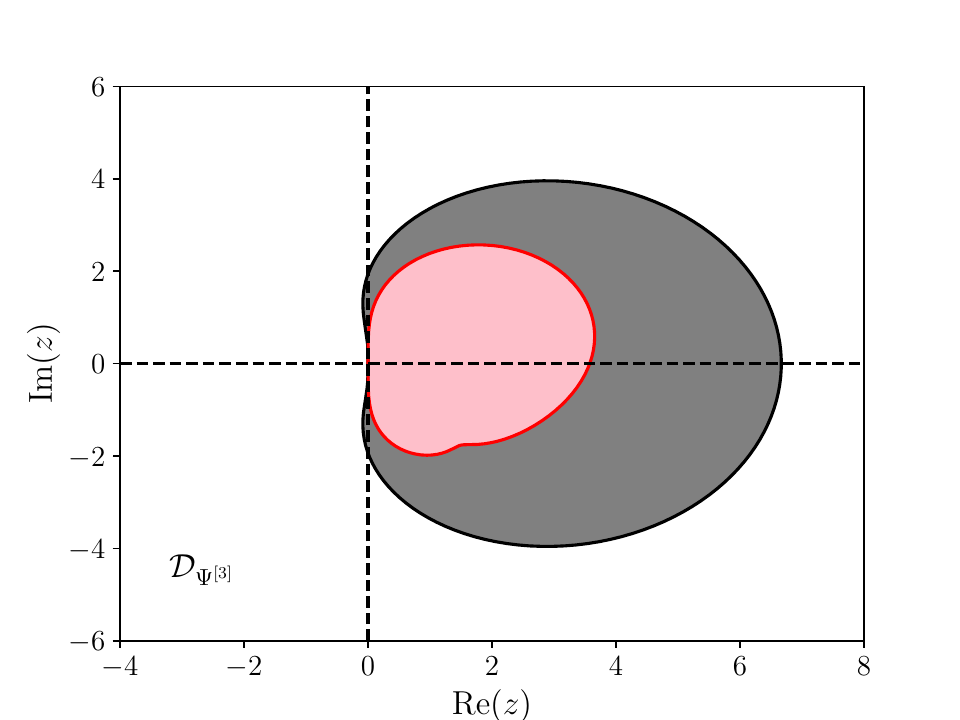}
\includegraphics[width=0.45\textwidth]{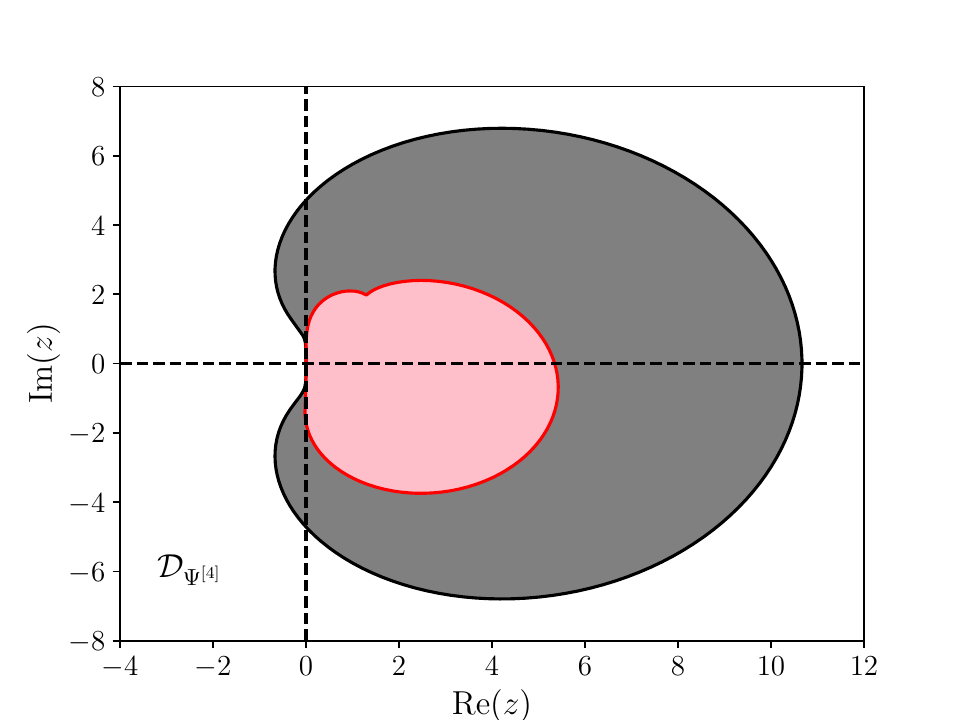}
\includegraphics[width=0.45\textwidth]{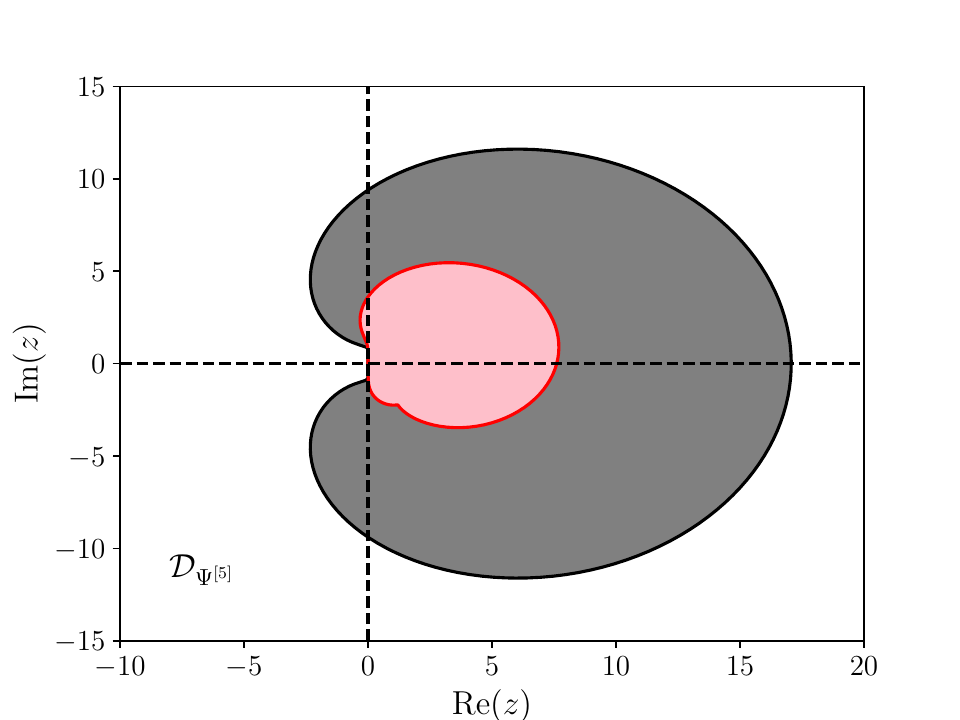}
\caption{The grey colored \textcolor{gray}{region} showed the unstable region of the \ac{BDFp} scheme of order $\op\in\{2,3,4,5\}$, while the pink colored \textcolor{pink}{region} shows the unstable region of the twice composition of the \ac{BDF}($\op-1$) scheme which is of order $\op$.}
\label{fig2}
\end{figure}

We plot the domain $\mathcal{D}_{\nfc{\op+1}{}}$ for different $\op\in\S{1}^{4}$ in \cref{fig2} and $\op\in\S{5}^{8}$ in \cref{fig3}. The black solid lines are the roots locus of the $\nf{\op}{}$ schemes and the red ones are the locus of roots to $\nfc{\op}{}$. The outer domain of these roots locus represents the stable region of the scheme. We recall here that the numerical scheme $\nfc{\op+1}{}$ is of order $\op+1$. Though the comparison of the stability region is done between two schemes of the same order: $\nf{\op}{}$ and $\nfc{\op}{}$. We can see from these pictures that $\mathcal{D}_{\nf{\op}{}}  \subseteq\mathcal{D}_{\nfc{\op}{}}$ for $\op\in\S{3}^{9}$.
\begin{figure}[!ht]%
\centering
\includegraphics[width=0.45\textwidth]{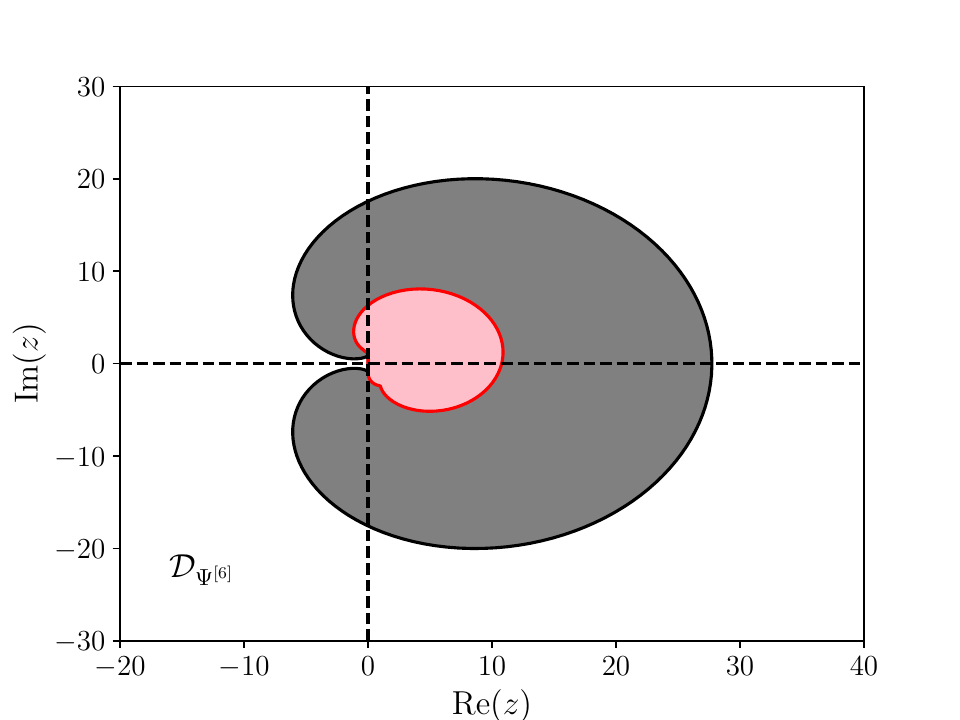}
\includegraphics[width=0.45\textwidth]{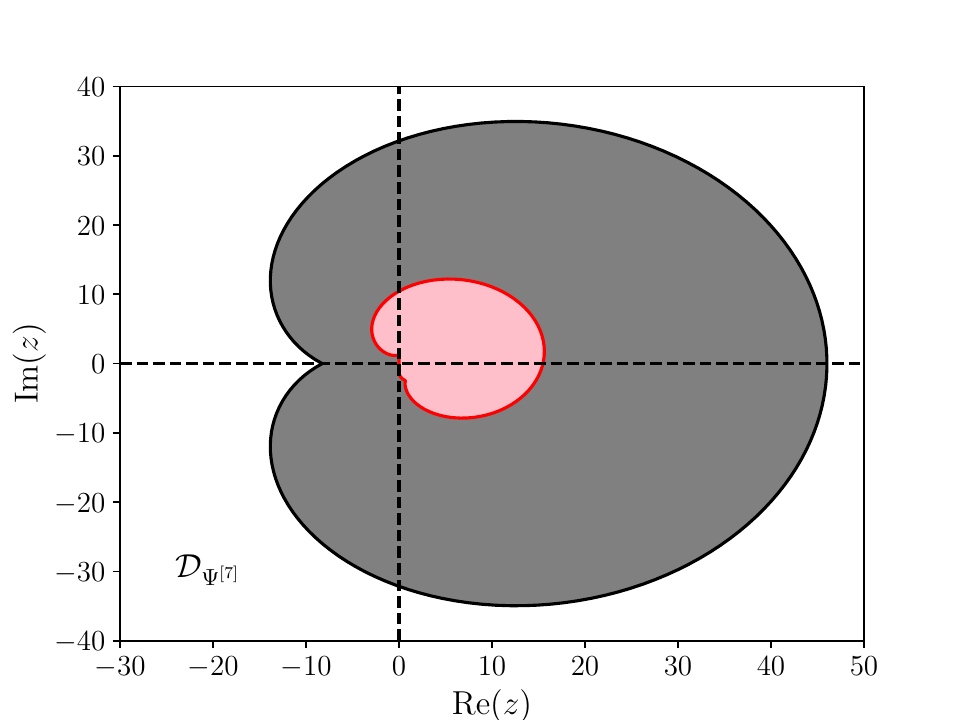}
\includegraphics[width=0.45\textwidth]{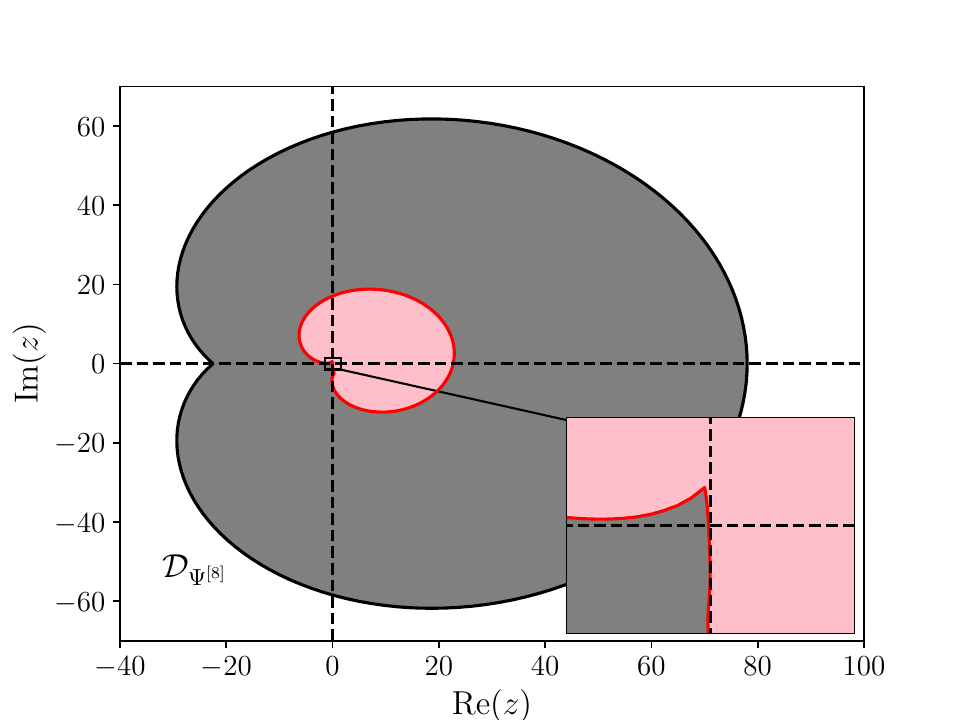}
\includegraphics[width=0.45\textwidth]{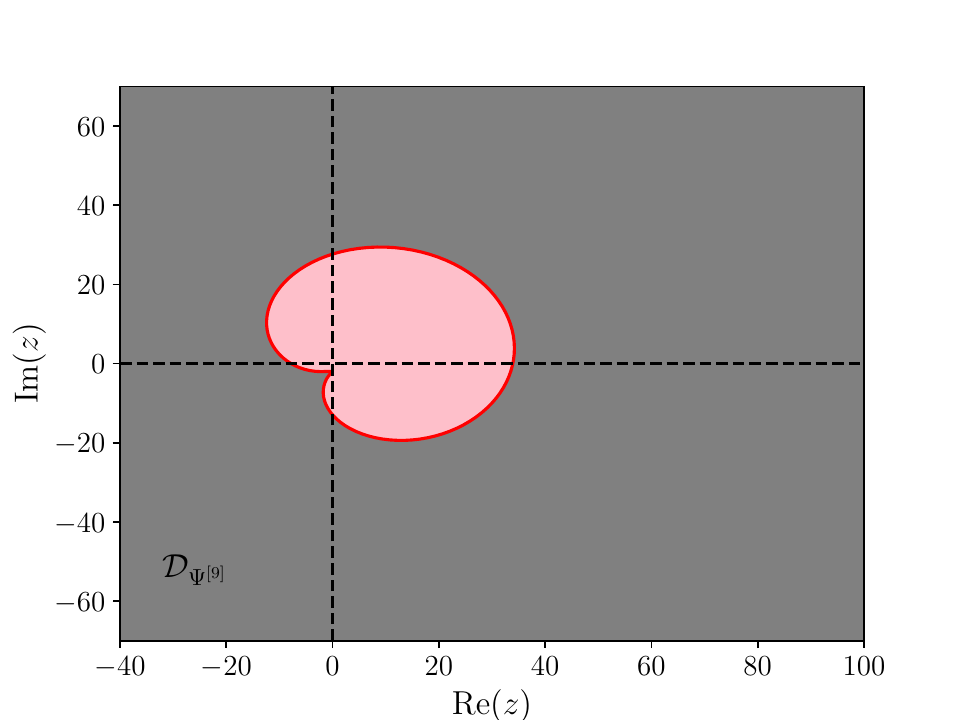}
\caption{The grey colored \textcolor{gray}{region} showed the unstable region of the \ac{BDFp}, $\op\in\{6,7,8,9\}$, while the pink colored \textcolor{pink}{region} shows the unstable region of the twice composition of the \ac{BDF}($\op-1$) which is of order $\op$.}
\label{fig3}
\end{figure}

In addition, the composed schemes of order $\op=2,3,4$ are stable, and those of order $\op\in\S{5}^{8}$ are $\mA(\vartheta)$-stable as can be seen in the shaded region of stability and the zoom in the bottom left subplot in \cref{fig3}. This in not the case for $\nf{\op}{}$ for $\op>6$.
To check the angles of stability for the composed flow $\vartheta_{\nfc{\op}{}}$, we have followed the same strategy as in \cite{gander-20} to compute them. These angles of stability regions are presented in \cref{tab4} with the ones of the $\vartheta_{\nf{\op}{}}$ for the classical \ac{BDF} schemes.
\begin{table}[!ht]
\begin{center}
\begin{minipage}{\textwidth}
\caption{Angles of the $\mA(\vartheta)-$stability domain}\label{tab4}
\begin{tabular*}{\textwidth}{@{\extracolsep{\fill}}l?cccccccc@{\extracolsep{\fill}}}
\toprule%
$\op$ & 2 & 3 & 4 & 5 & 6 & 7 & 8 &9 \\
\midrule
$\vartheta_{\nfc{\op}{}} $ & $90^{\circ}$ &90&90 & 81.511 & 67.796 & 45 & 4.146 & -  \\
$\vartheta_{\nf{\op}{}}$& $90^{\circ}$ & 86.032 &73.351& 51.839 & 17.839 & -&-&-\\
\bottomrule
\end{tabular*}
\end{minipage}
\end{center}
\end{table}

In conclusion to this section, the composition of \ac{BDFp} schemes enables us to break the Dahlquist barrier, so we can utilize higher-order \ac{BDF} families of schemes.

\section{Variable time step}
\label{sec_vts}
In this section, we explore the feasibility of implementing variable time stepping. Specifically, we examine whether certain time step values $\h_n$ are constrained to ensure the existence of a single root of the equation $\G{\op+1}=0$ with a real positive part. As System \eqref{sysa1a2} depends only on $\r{i},i\in\S{1}^{\op}$, the question to ask: is there a limit, either upper or lower, on the values of $\r{i}$ that must not be surpassed to guarantee the time marching? We will address this inquiry in the subsequent discussion. Contrary to the upper bounds restricted to the ratio of two adjacent time step found in \cite{Liao-20,Li-22} for keeping the stability of the variable time stepping of the second and third \ac{BDF} schemes, we find in our analysis lower bounds for adaptive time step when the composed flow of \ac{BDFp} schemes ($\op \in \S{2}^{8}$) is applied.

To begin, we will mention that $\a_1=1/2 \pm \mi/2$ are the roots of $\G{1+1} = 2\a_1^2 -2\a_1 +1=0$ when $\op=1$, which means that it does not depend at all on the ratio of time steps. Thus, the composed flow of order 2 has no lower bound on the ratio. In the case of $\op=2$, $\a_1$ is the root of
$$3\a_1^3 + \a_1^2(3\r{2}-4) + \a_1(\r{2}^2 - 2\r{2} + 2) + \r{2}=0,$$
where the algebraic expression of its roots as a function of $\r{2}$ is complicated. Therefore, finding the algebraic dependence of roots on $\r{2}$ to find where they have a positive real part becomes difficult to express. It is even more complicated for higher $\op$. Therefore, we conduct a numerical investigation to determine whether there is a barrier on $\h_n$ that prevents time marching, specifically the positive real part of $\a_1$ for adaptive time steps.

\subsection{The first step}
Consider first that, for a given $\op$, $\{\t_{i}\,\vert\,i\in\S{1}^{\op-1}\}$ are equidistant ($\h_i = \h, \forall i\in\S{1}^{\op-1}$).
When finding the roots of $\G{\op+1}=0$, we select among them the one with a positive real part to ensure marching in time.
To check this, we select a range of values of $\h_\op>0$ where, for every value, we calculate $\{\r{i}\,\vert\,i\in \S{1}^{\op}\}$ and then establish $\G{\op+1}$ to find its roots and check if there is one such that $\Re(\a_1)>0$.
Note that $r_{\left\lbrace 2\rvert \op\right\rbrace} \equiv \cfrac{\h}{\h_\op}$, which reflects the ratio of two adjacent time steps.
This aims to find a bound on two consecutive time steps, for which at least one of the roots of $\G{\op+1}=0$ has a positive real part. This is done numerically. We found that a lower bound on the ratio should be applied, as for instance when applying the double composition of \ac{BDF}2 -which is of order 3- and having two approximations to backwards points  $\t_0$ and $\t_0+\h_1$, the time step $\h_2$ must verify $\h_2 \geqslant 0.4506\times \h_1$ to ensure the time marching of the composition in having positive value in the real part of $\a_1$. Another example is illustrated when using the double composition of \ac{BDF}3- which is of order 4- and having approximations of backwards points $\t_0, \t_0+\h, \t_0+2\h$, the time step $\h_3$ must verify $\h_3\geqslant 0.6311 \h$.
We note that the lower bound increases as $\op$ increases. The values of the lower bounds for the rest of the composed flow are reported in \cref{tab5}.
\begin{table}[!ht]
\begin{center}
\begin{minipage}{\textwidth}
\caption{Lower bounds of $\h_\op/\h$.}\label{tab5} 
\begin{tabular*}{\textwidth}{@{\extracolsep{\fill}}l?ccccccc@{\extracolsep{\fill}}}
\toprule%
$\op$ & 2 & 3 & 4 & 5 & 6 & 7 & 8 \\
\midrule
$1/r_{\left\lbrace 2\rvert \op\right\rbrace}$ & 0.4506& 0.6311&0.7158&0.7717&0.8125&0.8454&0.8734 \\[10pt]
$1/\sqrt[2\op-3]{2}$ & 0.5&0.7937&0.8705&0.9057&0.9258&0.9389&0.9480 \\
\bottomrule
\end{tabular*}
\end{minipage}
\end{center}
\end{table}

For security and simplicity, we can generate the following rule of thumb to produce a lower bound of two adjacent time steps:
$$\frac{1}{r_{\left\lbrace 2\rvert \op\right\rbrace}} = \frac{\h_\op}{\h} \geqslant\frac{1}{\sqrt[2\op-3]{2}}, \quad \op\in\S{2}^{8}.$$
For additional information, we plot in \cref{fig8} the relation between $\h_\op$ and the real part of $\a_1$, a root of $\G{\op+1}$ for $\op\in\S{1}^{8}$ in the first step.
\begin{figure}[!ht]
\centering
\includegraphics[width=0.7\textwidth]{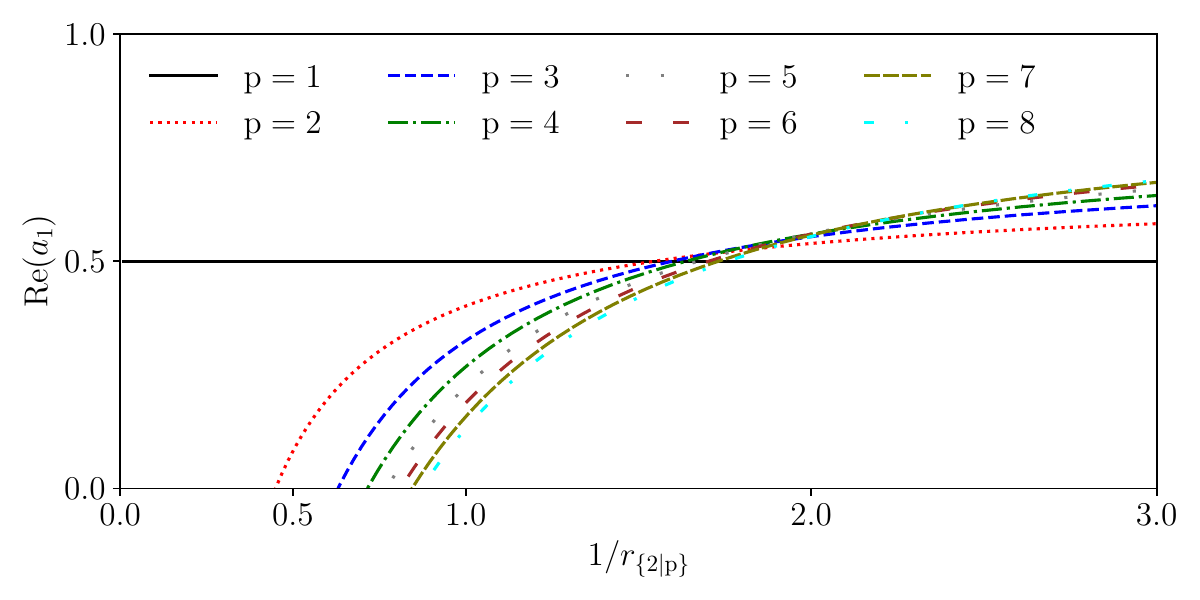}
\caption{Plot of the real part of $\a_1$ solutions of system \eqref{sysa1a2} for a range of values of $1/r_{\left\lbrace 2\rvert \op\right\rbrace}$.}
\label{fig8}
\end{figure}


\paragraph{Intermediate conclusions} Consider one has $\op$ starting points $\y_j$ approximating $\y(\t_j)$ for $\t_j = j\h, \forall \,j \in\S{0}^{\op-1}$ up order $\op$ and want to compose the numerical flow $\nf{\op}{}$. At the first step of computing an approximation $\y_\op$ to $\y(\t_\op \ \coloneqq\  \t_{\op-1}+\h_\op)$, the time step $\h_\op$ should be selected such that:
\begin{equation}\label{low_bound_h_k}
 \h_\op \geqslant
 \begin{cases}
  0& \op=1,\\
\displaystyle\frac{\h}{\sqrt[2\op-3]{2}}& \op\in \S{2}^{8}.
 \end{cases}
\end{equation}

\subsection{Simulation at the $n$\up{th} step}
Once the solution is approximated on $\t_\op = \t_0 +(\op-1)\h + \h_\op$ with $\h_\op$ verifying inequality \eqref{low_bound_h_k}, selecting the next variable time step size is crucial. For instance, users may need to decrease it for stiff dynamical problems or increase it if the system is not stiff. In the case when $\op=1$, and as it was shown before, there is no need to take care of the lower bound.

If $\op=2$, we can conclude according to \cref{tab5} and by induction, that the new time step should verify the following inequality to ensure a positive real part of the root:
\begin{equation}
\text{if}\,\,\op=2 \quad  \Longrightarrow \quad \h_n\geqslant 0.4506\times {\h_{n-1}},\quad  \forall \,n \in \mathds{Z}^+.
\end{equation}
For the numerical flow $\nfc{\op}{}$ with $\op\geqslant 3$, finding the limit bound of the time step is more complex. Thus, we follow the same procedure we applied in the first step. Results are reported in \cref{tab6}
\begin{table}[!ht]
\begin{center}
\begin{minipage}{\textwidth}
\caption{Lower bound of $\h_{\op+1}/\h_\op$ while $\h_\op\geqslant\h/\sqrt[2\op-3]{2}$.}\label{tab6}
\begin{tabular*}{\textwidth}{@{\extracolsep{\fill}}l?ccccccc@{\extracolsep{\fill}}}
\toprule%
$\op$ & 2 & 3 & 4 & 5 & 6 & 7 & 8 \\
\midrule
${1}/{r_{\left\lbrace 2\rvert \op+1\right\rbrace}}$ & 0.4506 & 0.6951 & 0.7763 & 0.8245 & 0.8613 & 0.0.8897 & 0.9131\\
\bottomrule
\end{tabular*}
\end{minipage}
\end{center}
\end{table}

We see that the values of the lower bounds of the ratio between the new time step $\h_{\op+1}$ and the last one $\h_\op$ increase, tending to unity, when $\op$ increases.
This will reduce the range of possibilities to decrease the time step, if needed, when higher-order schemes are employed for stiff dynamical problems involving different time scales. However, this may not limit their applications, as higher-order numerical schemes can capture the dynamics with higher precision, allowing for larger time steps.
We continue searching for the lower bound of the new time step $\h_{n+1}$ according to the last one $\h_n$. The following rule of thumb, as reported in \cref{lower_limit}, could be applied, as reported in \cref{tab7}.
\begin{equation}
 \label{lower_limit}
 \h_{n+1}\geqslant\frac{\h_n}{\sqrt[2\op-3]{2}}, \quad \forall\, \op \in \S{3}^{8}.
\end{equation}
\begin{table}[!ht]
\begin{center}
\begin{minipage}{\textwidth}
\caption{Lower bound of $\h_{n+1}/\h_{n}$, prescribed to $\nfc{\op+1}{~}$ while having $\h_{n}\geqslant\h_{n-1}/\sqrt[2\op-3]{2}$.}\label{tab7}
\begin{tabular*}{\textwidth}{@{\extracolsep{\fill}}l?ccccccc@{\extracolsep{\fill}}}
\toprule%
$\op$ & 2 & 3 & 4 & 5 & 6 & 7 & 8 \\
\midrule
${1}/{r_{\left\lbrace 2\rvert n+1\right\rbrace}}$ & 0.4506 & 0.6951 & 0.8063 & 0.8781 & 0.9193 & 0.9483 & 0.9709\\[5pt]
$1/\sqrt[2\op-3]{2}$ & 0.5 & 0.7937 & 0.8705  & 0.9057 & 0.9258 & 0.9389 & 0.9480\\
\bottomrule
\end{tabular*}
\end{minipage}
\end{center}
\end{table}

According to these values, we can see that the prescribed lower bound $\frac{1}{\sqrt[2\op-3]{2}}$ is a reliable solution for $\op\in \S{2}^{6}$, but not when $\op\in \{7,8\}$, where the following formula is well suited as reported in \cref{tab8}:
\begin{equation}
 \h_{n+1} \geqslant \frac{\h_n}{\sqrt[\op(\op-1)]{\op}}.
\end{equation}

\begin{table}[!ht]
\begin{center}
\begin{minipage}{\textwidth}
\caption{Lower bound of $\h_{n+1}/\h_{n}$, prescribed to $\nfc{\op+1}{~}$ while having $\h_{n}\geqslant \frac{\h_{n-1}}{\sqrt[\op(\op-1)]{\op}}$.}\label{tab8}
\begin{tabular*}{\textwidth}{@{\extracolsep{\fill}}l?ccccccc@{\extracolsep{\fill}}}
\toprule%
$\op$ & 2 & 3 & 4 & 5 & 6 & 7 & 8 \\
\midrule
${1}/{r_{\left\lbrace 2\rvert n+1\right\rbrace}}$ & 0.4501&0.6806&0.7900&0.8559&0.9019&0.9362&0.96351
\\[5pt]
$1/\sqrt[\op(\op-1)]{\op}$ & 0.7071&0.8326&0.8908&0.9226&0.9420&0.9547&0.96354\\
\bottomrule
\end{tabular*}
\end{minipage}
\end{center}
\end{table}

We mention here that there are no constraints in prescribing an upper bound for $1/r_{\left\lbrace 2\rvert n+1\right\rbrace}$, \emph{i.e.} the time step size could be taken greater than the last one with no limit. Nevertheless, studies \cite{Liao-20,Li-22} show that adaptivity of time steps for \ac{BDF} schemes is constrained by having upper bounds to maintain stability. In this paper, we do not establish upper bounds that ensure the stability of the composed flow of \ac{BDFp} for variable time, as this requires further study. Although a restriction on the upper bound will be prescribed for robustness, which aligns with the results found in \cite{Liao-20,Li-22} for the second and third \ac{BDF} schemes. We end by showing the conclusion of this section.
\paragraph{Intermediate conclusion} For stability of the computation, the upper and the lower bounds, denoted by $\ell_\op$ and $\frac{1}{\ell_\op}$ respectively, on the new time step size of the associate numerical flow $\nfc{\op+1}{\h_n}$ are given below:
\begin{equation}
\label{bounding_timestep}
\forall n:
\begin{cases}
~~~~~~~0 \leqslant \frac{\h_{n+1}}{\h_n}\leqslant 2, & \op=1,\\
 \frac{1}{\sqrt[2\op-3]{2}}\leqslant \frac{\h_{n+1}}{\h_n}\leqslant \sqrt[2\op-3]{2}, & \op \in \S{2}^{6},\\
 \frac{1}{\sqrt[\op(\op-1)]{\op}}\leqslant \frac{\h_{n+1}}{\h_n}\leqslant \sqrt[\op(\op-1)]{\op}, & \op \in \S{6}^{8}.
 \end{cases}
\end{equation}

\section{Numerical tests}
\label{sec-num}

In the composition process, $\a_1$ is a complex value, though
the approximation of the solution $\y_n$ in considered by the real part of $\hy_n$ produced by $\nfc{\op+1}{\h}$.
We demonstrate in this section, through numerical tests, that the imaginary part can represent an error estimation of the approximation. This is a powerful tool that enables us to adjust the time step to enhance stability in cases of notoriously stiff problems. In this section, we introduce some notations:
\begin{itemize}
 \item The absolute value of the exact error relative to $\nf{\op}{\h}$:
\begin{equation}
 \me(\t_n) \ \coloneqq\    \|\y(\t_n) - \y_n\| ,
\end{equation}
 \item The exact error relative to the composed numerical flow $\nfc{\op}{\h}$:
\begin{equation}
 \he(\t_n) \ \coloneqq\   \| \y(\t_n) - \Re(\hy_n) \|.
\end{equation}
 \item The imaginary part of $\hy_n$ in its absolute value
\begin{equation}
 \lvert \Im (\hy_n) \rvert.
\end{equation}
\end{itemize}
We denote by $\me_n$ and $\he_n$ approximations to $\me(\t_n)$ and $\he(\t_n)$, respectively.

\subsection{The first example}
We consider the \ac{IVP}  defined over the open interval $]0,5[$, having the initial condition $\y_0\ \coloneqq\ \y(0)$, $\f(\t,\y) = \lambda \y + g(\t)$, and $g(\t)$ a differentiable function. The exact solution is given below:
\begin{equation*}
 \y(\t) = \e^{\lambda \t} \cdot\left(\y_0 + \int_{0}^{\t}\e^{-\lambda \tau}\cdot g(\tau) d\tau \right)
.\end{equation*}
First, we consider $\lambda = -1/10$ and $g(\t) = \sin(\omega\t)$ with $\omega = 2\pi$ and $\y_0=2$, thus the exact solution is given as follows:
\begin{equation}
\label{sol_lin1}
\y(\t) = \e^{\lambda \t}\cdot\left( \y_0+ \frac{\omega - \e^{-\lambda \t}\cdot \big(\lambda \sin(\omega \t)  + \omega\cos(\omega \t)\big)}{\omega^2 + \lambda^2}\right).
\end{equation}
We will demonstrate that the imaginary part of the numerical approximation $\hy_n$, obtained by the composed numerical flow $\nfc{\op}{\h}$, represents an error estimation.  We compare the error of the numerical solution with the exact one for different orders $\op$, and we compare this error with the imaginary part. \cref{fig4} presents for several orders $\op$ and a fixed time step $\h = 0.01$ the plots of $\me_n,\he_n$ and $\Im(\hy_n)$. We can observe how the composed flow yields approximations with increased accuracy and how the imaginary part is of the same order as the exact error. It is worth noticing that when $\op=8$, the error increases with the simulation time in the case when numerical simulations are performed with \ac{BDF}. This is because the scheme is unstable for $\op \geqslant 6$, whereas the errors of approximations produced by the composed technique do not diverge, ensuring that the composed flow is stable by breaking the Dahlquist barrier.
\begin{figure}[!ht]%
\centering
\includegraphics[width=0.45\textwidth]{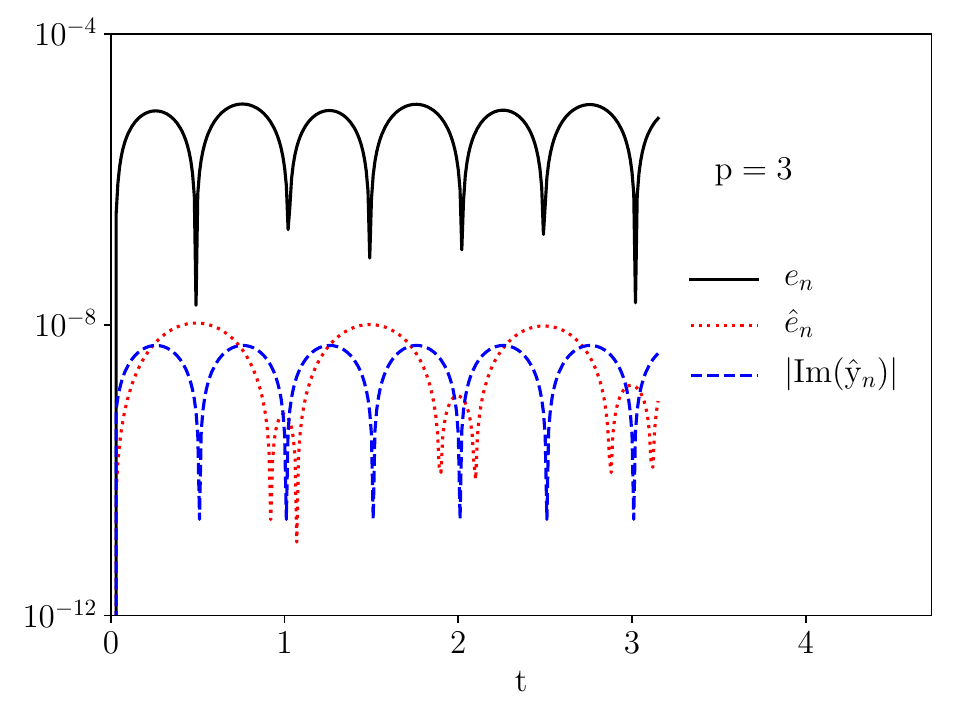}
\includegraphics[width=0.45\textwidth]{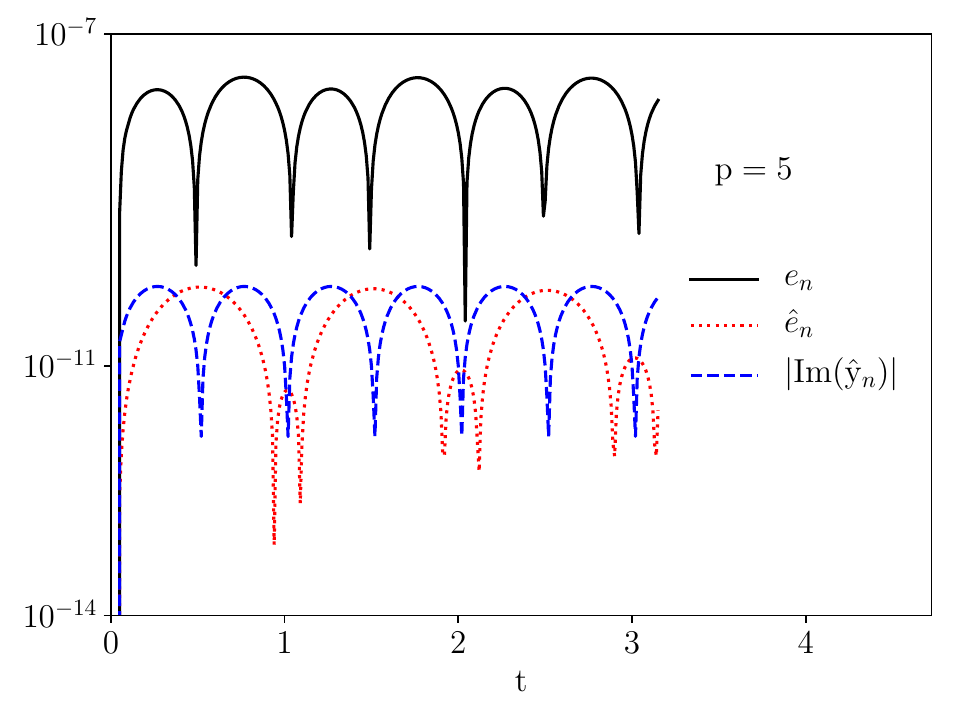}
\includegraphics[width=0.45\textwidth]{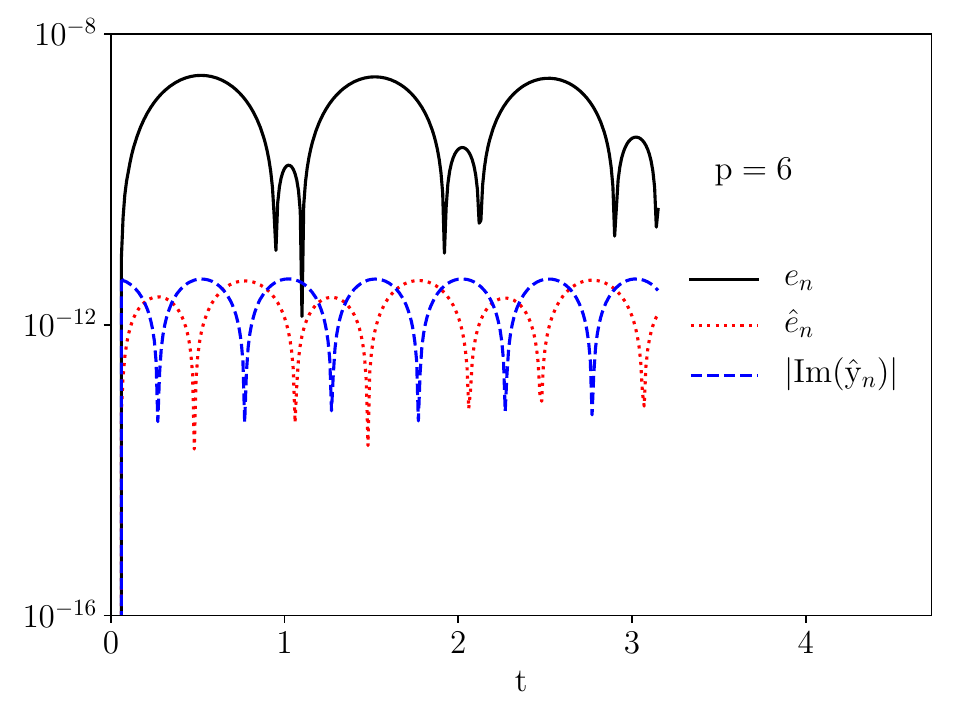}
\includegraphics[width=0.45\textwidth]{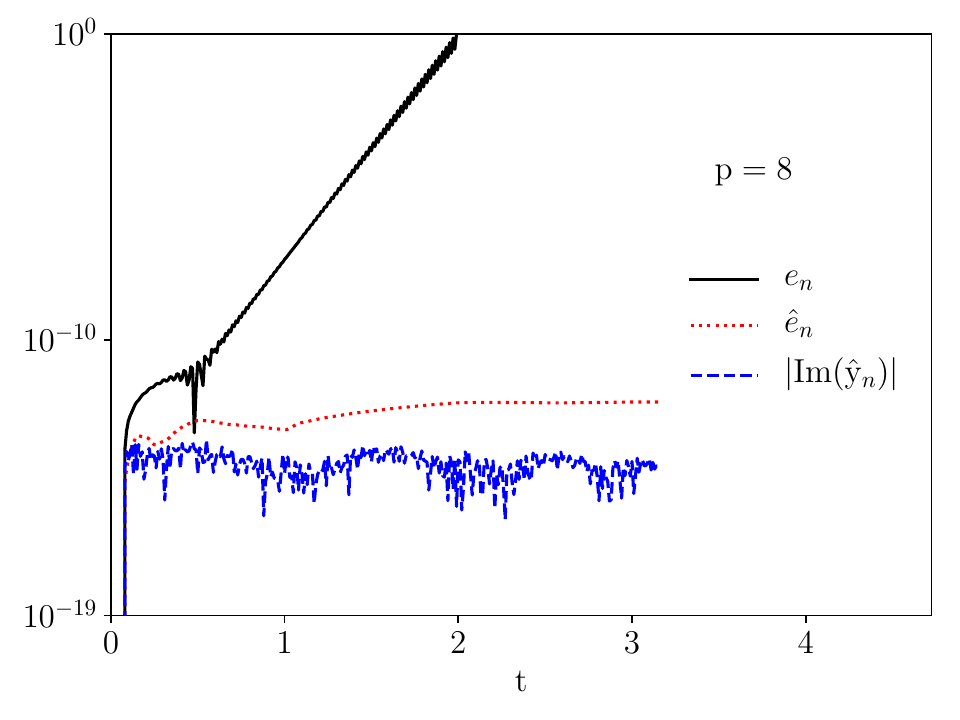}
\caption{Plots of $\me_n$ and $\he_n$ relative of the approximation to solution \eqref{sol_lin1} resulting from $\nf{\op}{}$ and $\nfc{\op+1}{}$ respectively, with the imaginary part $\Im(\hy_n)$ for $\op\in \{3,5,6,8\}$ and with $\h=0.01$.}
\label{fig4}
\end{figure}

\subsection{\ac{ODE}s with stiffness}
In this section, we present two examples of \ac{ODE}s with stiffness. The first one is a linear but non-homogeneous \ac{ODE} while the second is a nonlinear equation. The solution of the latter is represented by the Lambert function. These examples demonstrate the suitability of the imaginary part in following the error and estimating it.

\subsubsection{Linear equation}
We consider here again the last \ac{ODE} with $\lambda = -50$ and $g(\t) = -\lambda \arctan(20 \t)$. The simulation spans the segment $[0,2\pi]$ with the initial condition $\y_0 \coloneqq 1$ and uses the time step $\h = 0.01$. The solution is plotted in \cref{fig5}. It is clear that the solution exhibits a steep variation around $\t=0$ and $\t=3$, necessitating a change in the time step to follow the variation and maintain stable approximations.
\begin{figure}[!ht]%
\centering
\includegraphics[width=0.65\textwidth]{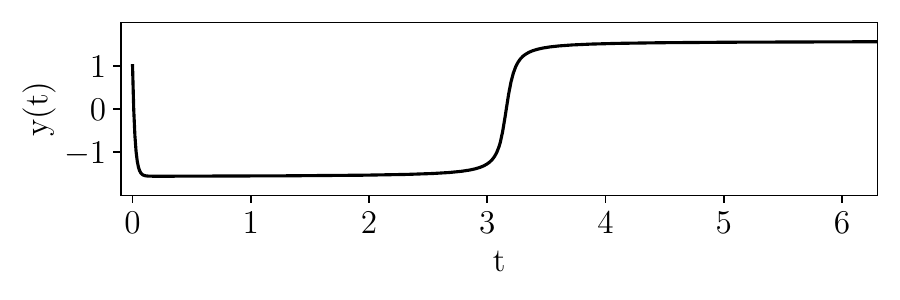}
\caption{The exact solution of the problem $\dot \y = -50 \y + 50\arctan(20\t)$.}
\label{fig5}
\end{figure}

We also add the plot of the error between the exact solution and the one obtained by the composition of the \ac{BDFp} for $\op\in\{4,5\}$. \cref{fig6} presents two evolutions of the exact errors obtained with $\nfc{4}{1/100}$ and $\nfc{5}{1/100}$. This figure also represents the imaginary part of $\hy_n$. First, we see that the accuracy of the solution increases when $\op$ increases. Second, we demonstrate that the evolution of the imaginary part follows that of the exact error while having the same order of magnitude.
\begin{figure}[!ht]%
\centering
\includegraphics[width=0.45\textwidth]{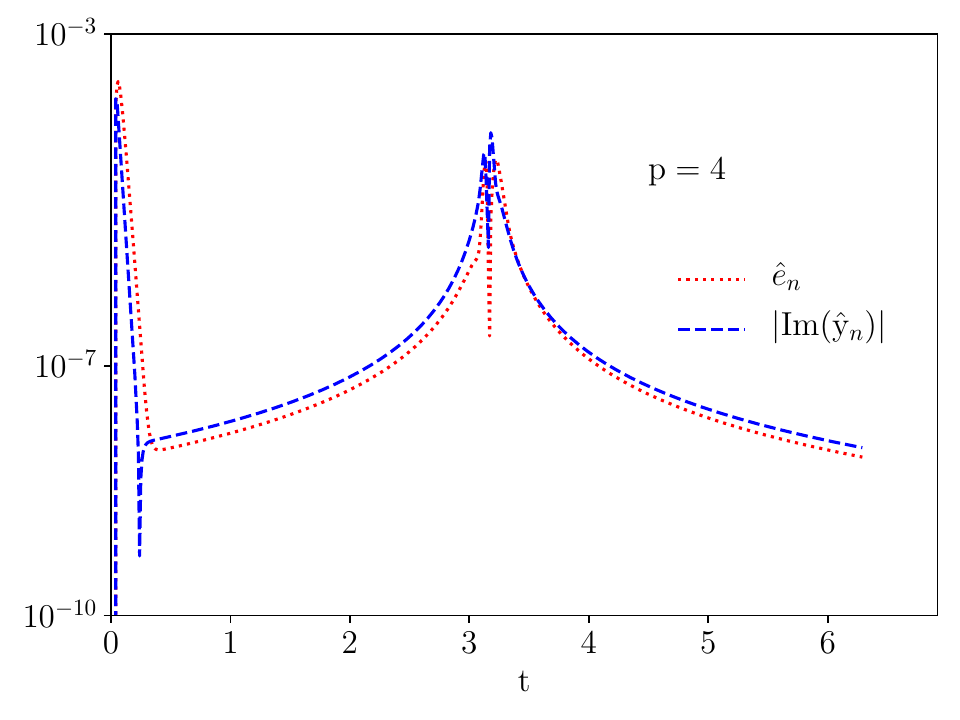}
\includegraphics[width=0.45\textwidth]{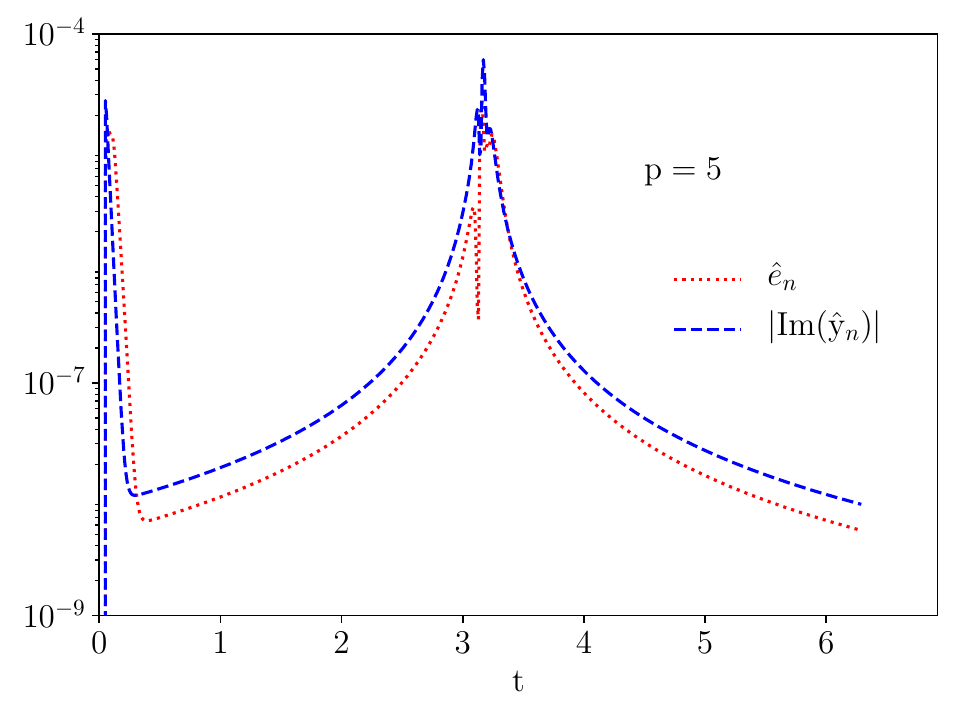}
\caption{Evolution of the exact error $\me(\t_n)$ and of the imaginary part of $\hy_n$ relative to the stiff problem defined by $g(\t) = -50 \arctan(20\t)$ and obtained by the composed flow $\nfc{\op+1}{}$ for $\op\in\{4,5\}$.}
\label{fig6}
\end{figure}

\subsubsection{Lambert function}
Consider the ODE defined by the right hand side $\f(\t,\y) \coloneqq \y^2 -\y^3$ with the initial condition $\y(0) = \delta$ over the segment $[0,2/\delta]$, where its solution is given exactly by  $\y(\t) \equiv \cfrac{1}{W(d e^{\,d-\t}) +1}$, with $d \ \coloneqq\  {1}/{\delta}-1$ and $W(z)$ being the Lambert function defined as the solution to the transcendent equation $We^W = z$. It is a classic example of testing the efficiency of the numerical scheme in a stiff case. The solution is plotted in \cref{fig12} for $\delta =0.01$.
\begin{figure}[!ht]%
\centering
\includegraphics[height=0.23\textwidth]{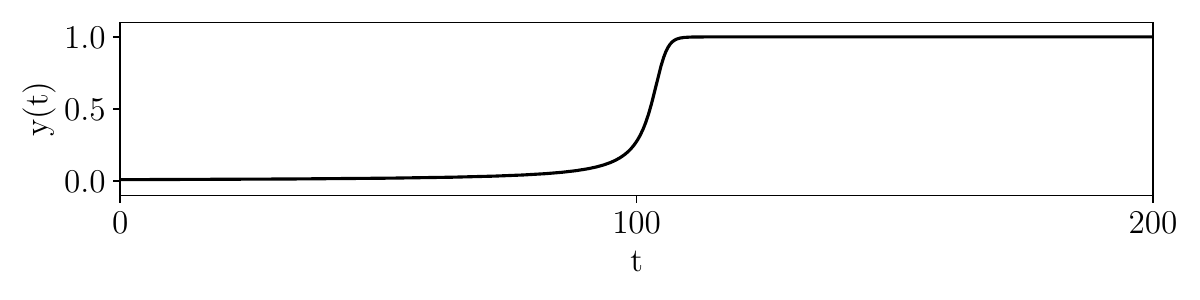}
\caption{Exact solution of the Lambert equation when $\delta = 0.01$.}
\label{fig12}
\end{figure}

\begin{figure}[!ht]%
\centering
\includegraphics[width=0.45\textwidth]{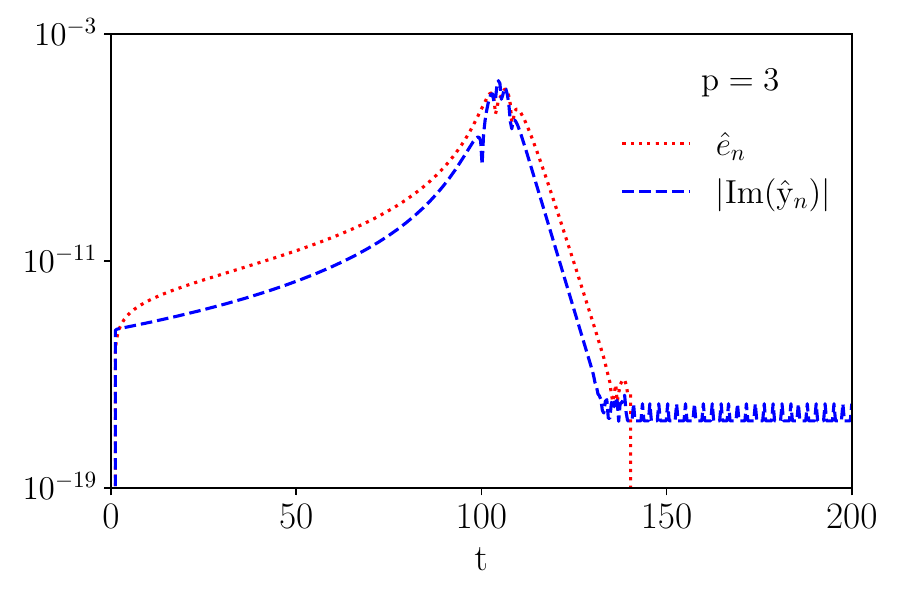}
\includegraphics[width=0.45\textwidth]{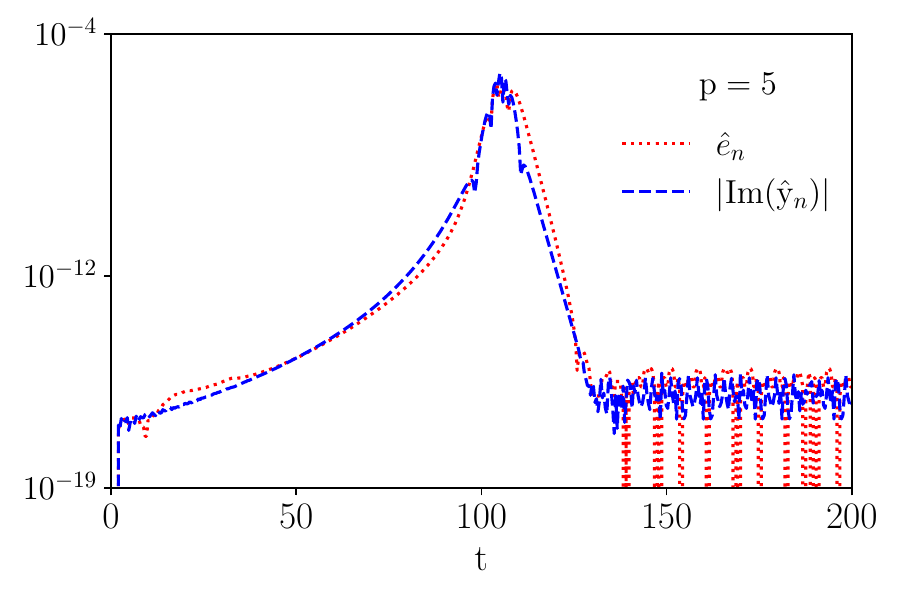}
\caption{Evolution of the exact error $\me(\t_n)$ and of the imaginary part of $\hy_n$ relative to the Lambert problem when $\delta =0.01$. It is obtained with the composed flow $\nfc{\op+1}{}$ for $\op\in\{3,5\}$.}
\label{fig7}
\end{figure}
\cref{fig7} presents the plots of the exact error evolution, for a fixed time step $\h_n = 0.01$, compared with the imaginary part of the approximation obtained by the composed flow for two orders $\op\in\{3,5\}$. It is clear that the latter follows the same pattern as the former. To check the efficiency of the rule of thumb formula in \eqref{bounding_timestep} for choosing the value of the new step size and ensure having a real positive part in the root $\a_1$, we proceed to update the time step with and without it. \cref{fig9} presents this comparison between the real parts during the simulation. The upper panel presents the case without it, where a negative real part occurs around $\t=100$, thus blocking the marching time. We can see that applying the rule of thumb presented in \eqref{bounding_timestep}, ensures advancing in time by the composed flow as all roots of $\a_1$ during the simulation have positive real parts.

\begin{figure}[!ht]%
\centering
\includegraphics[height=0.23\textwidth]{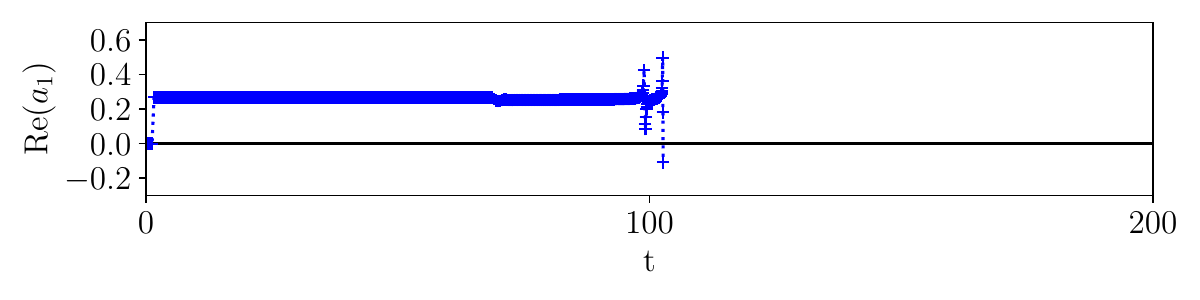}
\includegraphics[height=0.23\textwidth]{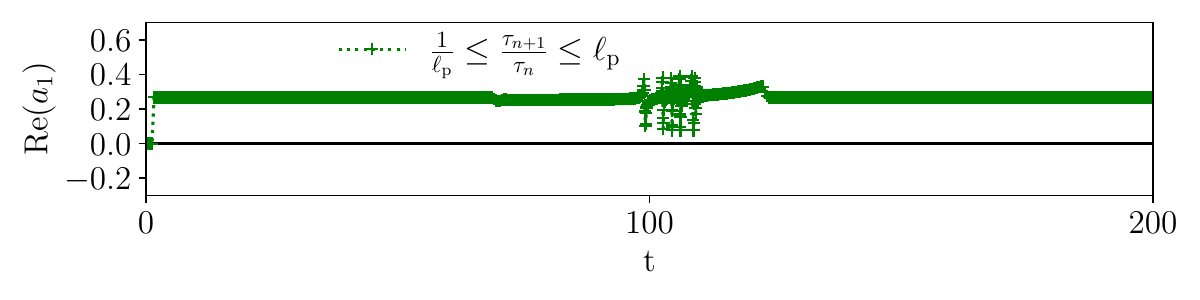}
\caption{Evolution of the maximum of the real part of roots $\a_1$ during the simulation when $\op=4$ and $\tol=10^{-12}$ using the bounding strategy in \cref{bounding_timestep} (lower panel) and without it (upper panel).}
\label{fig9}

\end{figure}\begin{figure}[!ht]%
\centering
\includegraphics[height=0.23\textwidth]{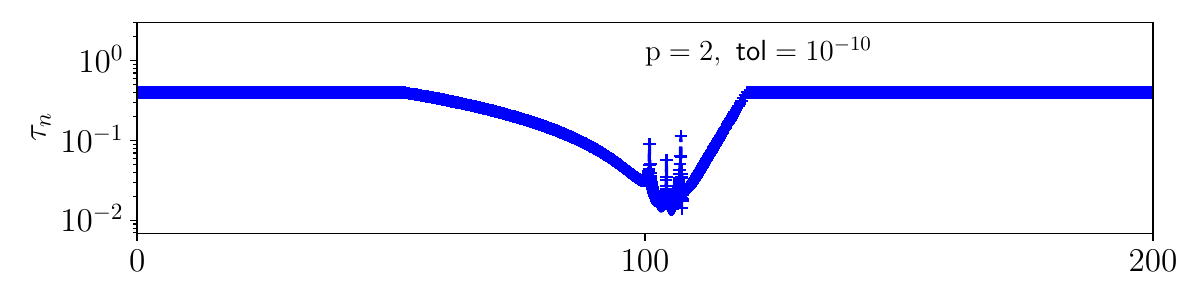}
\includegraphics[height=0.23\textwidth]{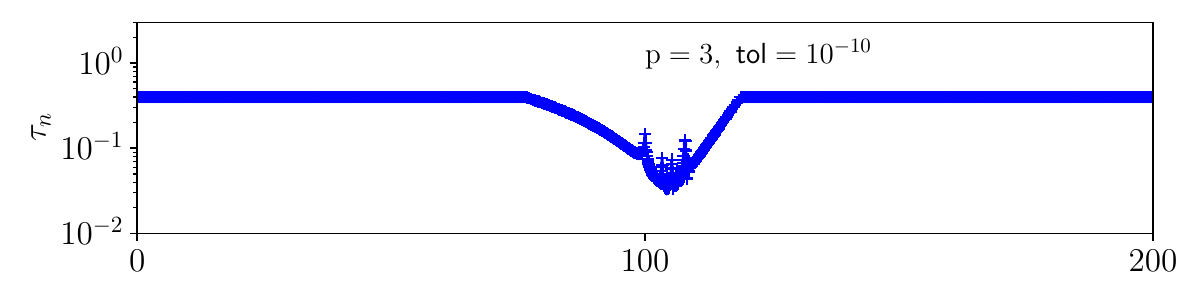}
\includegraphics[height=0.23\textwidth]{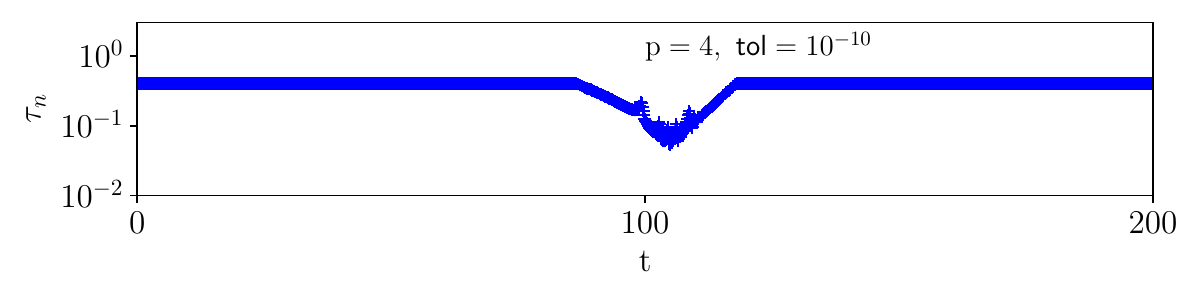}
\caption{Evolution of the $\h_n$ during the simulation with the adaptive \cref{alg-adap} for orders $\op=2,3,4$ and $\tol = 10^{-10}$.}
\label{fig10}
\end{figure}

We end by plotting the evolution of the time step $\h_n$ for different features of the simulation. \cref{fig10} presents the evolution when applying the composed flow with different orders $\op\in\{2,3,4\}$ and a fixed user tolerance $\tol$.
All simulations present at the beginning a steady time step as the solution shows no variation before reaching the neighborhood of $\t=100$, as \cref{fig12} demonstrates it.
We can see how the order of the scheme affects time step adaptivity: the higher the order, the lower the variation in the time step. However, and for a fixed order $\op$, the adaptivity of the time step size presents a high rate of variation when the tolerance is smaller. This is demonstrated in \cref{fig11} for $\op=2$ and $\tol \in\{ 10^{-7},10^{-9},10^{-12}\}$.

\begin{figure}[!ht]%
\centering
\includegraphics[height=0.23\textwidth]{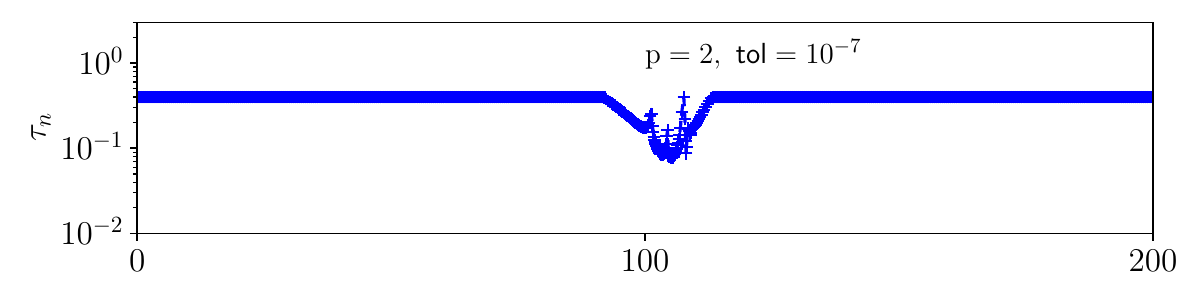}
\includegraphics[height=0.23\textwidth]{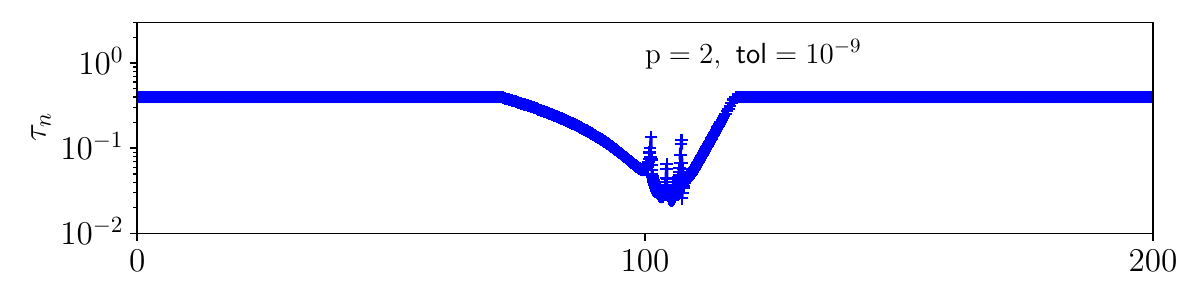}
\includegraphics[height=0.23\textwidth]{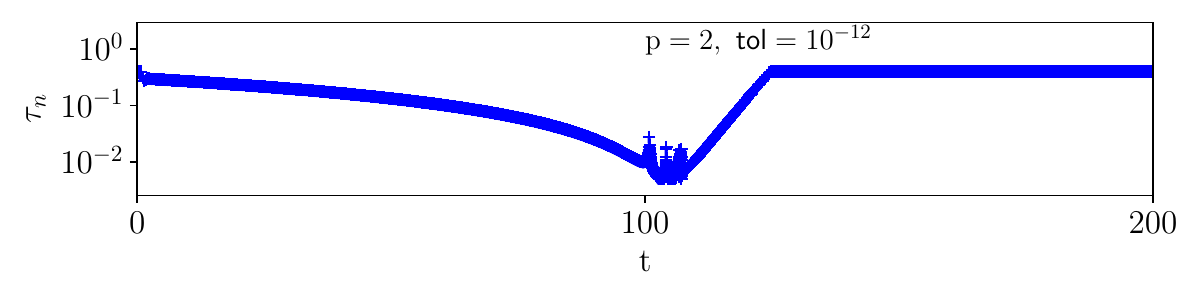}
\caption{Evolution of the $\h_n$ during the simulation with the adaptive \cref{alg-adap} for $\op=2$ and various user tolerance $\tol\in\{ 10^{-7},10^{-9},10^{-12}\}$.}
\label{fig11}
\end{figure}

\paragraph{Intermediate conclusion}
We conclude from these numerical experiments that the proposed numerical flow obtained by twice composing the \ac{BDFp} numerical flow produces a numerical scheme where the output has two parts: the real part, an approximation to the solution with an additional accuracy, and the imaginary part, an error estimate of the resulting approximation.

\section{Conclusion and perspectives}\label{sec13}
In this paper, we presented a composition technique for implicit \ac{BDFp} schemes. Unlike the cyclic composition methods proposed in the last century \cite{Hansen-69, cash-77}, our technique is akin to those developed for one-step methods \cite{blanes-01}. Specifically, it employs the same numerical flow with varying time step sizes.

We proved that two consecutive compositions of a \ac{BDF} scheme of order $\op$ yield an approximation with an additional order of accuracy. Notably, this enhancement is achieved without requiring additional backward points. To attain the higher order, an algebraic equation must be solved once for all fixed time steps for a given order $\op$, or at every time layer when variable time steps are used. The solution to this algebraic equation is a complex number, which can be computed numerically.

Furthermore, leveraging the difference operator for an \ac{LMS} method, we demonstrated that the imaginary part of the outputs provides an error estimate for the approximation generated by the composed flow. The corresponding error constant is also derived. This was illustrated through numerical experiments conducted on several classic stiff problems, both at the end of the paper and throughout the discussion.

Another advantage of the composed flow of order $\op$\up{th} is that it uses $\op-1$ backward points, compared to the \ac{BDFp} method, which requires $\op$ points. This reduction in backward points helps minimize memory allocation.

Furthermore, numerical tests have demonstrated that the computational performance of the composed flow of order $\op$ is comparable to that of the classical \ac{BDF} scheme of the same order for low orders (2 and 3) and surpasses it for higher orders. By surpassing, we mean that the composed flow achieves the same level of accuracy with lower CPU time. This improved performance is achieved with fewer backward points.

If the same number of backward points is maintained, and the computational efficiency of the composed flow is compared to the classical \ac{BDFp}, the composed flow outperforms the classical \ac{BDFp} for any order $\op$.

The linear stability of the composed scheme was investigated.
The generating polynomials were obtained in the case of a linear differential equation, and angles of stability were determined for every order.
We have proved graphically that the new scheme $\nfc{\op}{}$ is $\mA(\vartheta)$ stable up to order $\op=8$, breaking the Dahlquist barrier, equal to six, in implicit \ac{BDF} schemes.

Variable time stepping is analyzed, and conditions for ensuring a real positive part of the jumping step are presented. A lower bound formula for the ratio of two consecutive time steps is derived to guarantee stable time marching in simulations using the composed flow.

From our perspective, the developed technique of the composition should be extended to Partial Differential Equations. To achieve this, we should develop a mathematical formulation that writes the variational formulation of the problem in a mixed form, where the real and imaginary parts are the two associated unknowns, utilizing a mixed formulation strategy. In this case, the scalar field of the imaginary part associated with the variable in question is of interest, not only for adapting the time step but also for use in mesh refinement. This should be investigated in future works.

\section*{Acknowledgments}

This publication is based upon work supported by the Khalifa University of Science and Technology under Award No. FSU-2023-014.
This work has been supported by the Khalifa University of Science and Technology under award no. RIG-2023-024.

\section*{List of abbreviations}
\begin{acronym}[TDMA]
\acro{ODE}{Ordinary Differential Equations}
\acro{IVP}{Initial Value Problem}
\acro{LMS}{Linear Multi-Step}
\acro{BDF}{Backward Difference formula}
\acro{BDFp}{\ac{BDF} of order $\op$}
\acro{RK}{Runge-Kutta}
\acro{ERK}{Embedded Runge-Kutta}
\acro{IRK}{Implicit Runge-Kutta}
\acro{ETD}{Exponential-Time Differencing}
\acro{DSR}{Divergent Series Resumation}
\acro{BPL}{Borel-Padé-Laplace}
\acro{ATS}{Adaptive Time Stepping}
\acro{CN}{Crank-Nicolson}
\acro{DOC}{Discrete Orthogonal Convolution}
\end{acronym}

\nomenclature[001]{$\t$}{Time variable coordinate}
\nomenclature[002]{$\T{}$}{The final time of simulation}
\nomenclature[003]{$\y(\t)$}{The unknown function}
\nomenclature[004]{$\f$}{Right hand side of the differential system}
\nomenclature[005]{$\me(\t)$}{Local error}
\nomenclature[006]{$\phi_t$}{Exact flow of the differential system}
\nomenclature[007]{$\t_n$}{The $n$\up{th} instant of time}
\nomenclature[008]{$\h_n$}{The $n$\up{th} time step magnitude}
\nomenclature[009]{$\y_n$}{The $n$\up{th} approximation of $\y(\t_n)$}
\nomenclature[010]{$\me_n$}{Error estimate of $\me(\t_n)$}
\nomenclature[011]{$\S{i}^j$}{The set of positive integers numbers between $i$ and $j$}
\nomenclature[012]{$(.)\vert_{\t=\t_n}$}{Evaluation operator at $\t=\t_n$}
\nomenclature[013]{$\nabla^j\y_n$}{Backward Difference operator}
\nomenclature[014]{$\gf{i}$}{The $i$\up{th} coefficients of \ac{BDF} in case of fixed time step}
\nomenclature[015]{$\g{i}$}{The $i$\up{th} coefficients of \ac{BDF}($\op$) to find approximation at $\t=\t_n$}
\nomenclature[016]{$L$}{Difference operator}
\nomenclature[017]{$\tol$}{User tolerance}
\nomenclature[018]{$\mT_{n,\op}$}{Vector containing instants points set $\{\t_{n-i}\,\vert\,i\in\S{0}^{\op-1}\}$}
\nomenclature[019]{$\Y_{n,\op}$}{Vector containing elements of the set $\{\y_{n-i}\,\vert\,i\in\S{0}^{\op-1}\}$}
\nomenclature[020]{$\nf{\op}{\h_n}$}{Numerical flow of the \ac{BDFp} with a time step $\h_n$}
\nomenclature[021]{$\a_1$}{A complex number}
\nomenclature[022]{$\tnmh$}{Intermediate instant between $\t_{n-1}$ and $\t_n$}
\nomenclature[023]{$\ynmh$}{Approximation of $\y(\tnmh)$}
\nomenclature[024]{$\Tnmh{\op}$}{Vector of instants $\{\t_{n-i} \wedge \tnmh \,\vert \, i \in\S{1}^{\op-1}\}$}
\nomenclature[025]{$\Ynmh{\op}$}{Vector of approximations $\{ \y_{n-i} \wedge \ynmh \,\vert \, i \,\in \S{1}^{\op-1}\}$}
\nomenclature[026]{$\eps{j}$}{The $j$\up{th} ratio relative to the first integration in the composition}
\nomenclature[027]{$\mT^{\prime}_{n,\op}$}{Vector of instants $\{\t_{n-i} \wedge \tnmh \,\vert \, i \in\S{2}^{\op}\}$}
\nomenclature[028]{$\Y^{\prime}_{n,\op}$}{Vector of approximations $\{\y_{n-i} \wedge \ynmh \,\vert \, i \,\in\S{2}^{\op}\}$}
\nomenclature[029]{$\hy_n$}{Approximation of $\y(\t_n)$ using double composition}
\nomenclature[030]{$\Re(\hy_n)$}{The real part of $\hy_n$}
\nomenclature[031]{$\Im(\hy_n)$}{The imaginary part of $\hy_n$}
\nomenclature[032]{$\nfc{\op+1}{\h_n}$}{The numerical flow of the double composition of $\nf{\op}{}$}
\nomenclature[033]{$\ell_\op$}{Bound element in the adaptive tomestep choosing process}
\nomenclature[034]{$\Eps{j}$}{$j$\up{th} ratio relative to the second integration in the composition}
\nomenclature[035]{$\G{i}$}{$i$\up{th} coefficients of \ac{BDF}($\op+1$) to find approximation at $\t=\t_n$}
\nomenclature[036]{$L$}{Difference operator of the first integration in $\nfc{\op+1}{}$}
\nomenclature[037]{$\E_j$}{$j$\up{th} term of the Taylor development of $L_1$}
\nomenclature[038]{$L_2$}{Difference operator of the second integration in $\nfc{\op+1}{}$}
\nomenclature[039]{$\mE_j$}{$j$\up{th} term of the Taylor development of $L_2$}
\nomenclature[040]{$z$}{A complex variable}
\nomenclature[041]{$p_\op(\omega,z)$}{Polynomial associated to an \ac{LMS} method}
\nomenclature[042]{$\omega_i(z)$}{$i$\up{th} root of $p_\op$ as a function of $z$}
\nomenclature[043]{$\mathcal{D}_{\nfc{\op}{}}$}{$A$-stability domain of the numerical flow $\nfc{\op}{}$}
\nomenclature[044]{$\vartheta_{\nfc{\op}{}}$}{Stability angle of the numerical flow $\nfc{\op}{}$}

\printnomenclature

\begin{appendices}
\section{Algorithm of flow $\nf{\op}{\h_n}$}
 \label{sec_app2}
 In this section, we present the numerical flow associated with the \ac{BDFp} in \cref{alg:BDFk}.

 \begin{algorithm}[!ht]
\caption{The numerical flow $\nf{\op}{\h_n} (\mT_{n-1,\op},\Y_{n-1,\op})$ of \ac{BDFp}}\label{alg:BDFk}
\begin{algorithmic}
\Require $\h_n$, $\tol$
\State [$\g{0},\ldots,\g{\op}] \gets \mathsf{Coeff}\big([\t_{n-\op},\ldots,\t_{n-1},\h_n]\big)$ (see \cref{alg:gi})
\State $\ell \gets 0$
\State $ \y_{n,\ell} \gets \y_{n-1}$
\State  $\me_{n,\ell}\gets \tol\times 2$
\While{$\me_{n,\ell} > \tol$}
    \State $\y_{n,\ell+1} \gets \F_\op\big([\g{0},\ldots,\g{\op}], \y_{n,\ell}\big)$
    \State $\me_{n,\ell+1} \gets \|\y_{n,\ell+1} - \y_{n,\ell}\|$
    \State $ \y_{n,\ell+1} \gets \y_{n,\ell}$
    \State $\ell \gets \ell +1$
\EndWhile
\State $\y_n \gets \y_{n,\ell+1}$
 \State \Return $(\mT_{n,\op},\Y_{n,\op})$
\end{algorithmic}
\end{algorithm}

\section{Proof of \cref{prop1}}
\label{sec_app3}
In this section, the proof of \cref{prop1} is done only for $\G{\op+1}$. The proofs for other coefficients $\{\G{i}\,\vert\,i\in\S{0}^{\op}\}$ can be obtained using the same strategy.
Mainly, we use the Cramer rule to compute the solution of $\G{\op+1}$:
\begin{equation}
\label{formula_detGk}
 \G{\op+1} = \frac{\det(\MAp_{n,\op+1\vert \op+2})}{\det(\MAp_{n,\op+1})},
\end{equation}
where $\det: \Mat{n}{\mathds{C}} \rightarrow \mathds{C}$ is the application returning the determinant of a given matrix and $\MAp_{n,\op+1\vert \op+2}$ is the matrix obtained by replacing the right hand side in the $(\op+2)$\up{th} column in matrix $\MAp_{n,\op+1}$.
We can see that the last case in the second column contains only two elements: $\Eps{0}^{\op+1}$ and $(-\a_1)^{\op+1}\prod\limits_{i=1}^{\op} \eps{i}$ leading us to decompose the determinant into two parts as follows:\begin{equation}
 \det\left( \MAp_{n,\op+1} \right) = \det\left( \MB_{\op}\right) - \frac{\a_1^{\op+1}}{\g{0}} \prod\limits_{i=1}^{\op}\eps{i} \det\left( \MB_{\op}^{1,\op+1} \right).
\end{equation}
In the same context and after taking factor of $-\Eps{0}$ in the last column, we get:
\begin{equation}
\label{det_Ak_1}
 \det\left( \MAp_{n,\op+1\vert \op+2} \right) = (-1)^{\op+1}\Eps{0}\left( \det\left( \MB_{\op+1}\right) +  \frac{\a_1^{\op+1}}{\g{0}} \prod\limits_{i=1}^{\op}\eps{i}\det\left( \MB_{\op+1}^{1,\op} \right) \right),
\end{equation}
where $\MB_{\op-i}\in \Mat{\op-i}{\mathds{C}}$ is the following matrix defined in \cref{MBk} and $\MB_{\op-i}^{j,\ell}$ results from $\MB_{\op}$ after deleting the $j$\up{th} column and the $\ell$\up{th} row.
\begin{equation}
\label{MBk}
\MB_{\op-j}\ \coloneqq \
\left(
\begin{array}{ccc}
\Eps{0}^{1+j} &\ldots &  \Eps{\op-j}^{1+j}\\
\vdots& \ddots &\vdots\\
\Eps{0}^{\op} & \ldots & \Eps{\op-j}^{\op}\\[8pt]
 \end{array}
 \right)
 \cdot
 \end{equation}
After manipulating columns and rows in both $\MB_\op$ and $\MB^{1,\op}_{\op}$, and using elementary algebraic operations, we can show that:
\begin{eqnarray}
\label{det_Bk}
 \det(\MB_{\op-j}) &=& \prod\limits_{i=0}^{\op-j}\Eps{i}^{j+1}\prod\limits_{i=0}^{\op-j-1}\prod\limits_{\ell=i+1}^{\op-j} (\Eps{\ell}-\Eps{i}),\\
 \label{det_Bk1k}
 \det(\MB^{1,\op-j+1}_{\op-j}) & = &\prod\limits_{i=1}^{\op-j}\Eps{i}^{j+1}\prod\limits_{i=1}^{\op-j-1}\prod\limits_{\ell=i+1}^{\op-j} (\Eps{\ell}-\Eps{i}).
\end{eqnarray}
After using formulas \eqref{det_Bk} and \eqref{det_Bk1k} in \eqref{det_Ak_1}, and taking common factors, we get:
\begin{eqnarray*}
  \det\left( \MAp_{n,\op+1\vert \op+2} \right) &=& (-1)^{\op+1} \Eps{0} \left(  \prod\limits_{i=0}^{\op-1}\Eps{i}^{2}\prod\limits_{i=0}^{\op-2}\prod\limits_{\ell=i+1}^{\op-1} (\Eps{\ell}-\Eps{i})  \right.\\
  && + \left. \frac{\a_1^{\op+1}}{\g{0}} \prod\limits_{i=1}^{\op}\eps{i} \prod\limits_{i=1}^{\op-1}\Eps{i}^{2}\prod\limits_{i=1}^{\op-2}\prod\limits_{\ell=i+1}^{\op-1} (\Eps{\ell}-\Eps{i})\right)
 \end{eqnarray*}
 Taking common factors, we have:
 \begin{eqnarray*}
  \det\left( \MAp_{n,\op+1\vert \op+2} \right)  &=&(-1)^{\op+1}\Eps{0}\prod\limits_{i=1}^{\op+1}\Eps{i}^{2}\prod\limits_{i=1}^{\op-2}\prod\limits_{\ell=i+1}^{\op+1} (\Eps{\ell}-\Eps{i})\times\\
  &&\left(\Eps{0}^2\prod\limits_{\ell=1}^{\op+1}(\Eps{\ell}-\Eps{0})
  +\frac{\a_1^{\op+1}}{\g{0}} \prod\limits_{i=1}^{\op}\eps{i},
  \right)
 \end{eqnarray*}
 which equal to:
 \begin{equation}
  \label{det_Ak_2}
  \begin{array}{ccl}
  \det\left(  \MAp_{n,\op+1\vert \op+2}\right) &=& \displaystyle\frac{(-1)^{\op+1}}{\g{0}}\a_1^{\op+1}\Eps{0}\prod\limits_{i=1}^{\op-1}\eps{i}\Eps{i}^{2}\times\\
   &&\prod\limits_{i=1}^{\op-2}\prod\limits_{\ell=i+1}^{\op-1} (\Eps{\ell}-\Eps{i})\times\left( \g{0}\Eps{0}^2 + \a_1^2 \eps{\op} \right)
   \end{array}
 \end{equation}
We also show that:
\begin{equation}
\label{det_Ak}
 \det(\MAp_{n,\op+1})  = \displaystyle\frac{\a_1^{\op}}{\g{0}} \prod\limits_{i=1}^{\op}\eps{i}\Eps{i}\prod\limits_{i=1}^{\op-1}\prod\limits_{\ell=i+1}^{\op} {(\Eps{\ell}-\Eps{i})}\Big(\g{0}\Eps{0}-\a_1\Big).
\end{equation}
To this end, we replace Formulas \eqref{det_Ak_2} and \eqref{det_Ak} in \eqref{formula_detGk}. After simplifying, we obtain the formula of $\G{\op+1}$ in \eqref{form_Gkp1}, which completes proof for $\G{\op+1}$:
\begin{equation}
 \label{finalGk1}
 \G{\op+1} = (-1)^{\op+1} \prod\limits_{i=1}^{\op-1}\left[\frac{\Eps{i}}{\Eps{\op}-\Eps{i}}\right] \frac{(\a_1-1)\Big[\Eps{0}^2\g{0} + \a_1^2 \eps{\op}\Big]}{\Eps{\op}\big(\r{\op}+\a_1\big)\big[\Eps{0}\g{0}-\a_1\big]}
\end{equation}

\section{Closed form of $\G{\op+1}$}
\label{sec_app4}
We present in this Appendix the explicit expression of $\G{\op+1}$ when $\op\in\S{1}^{4}$

\begin{equation*}
\begin{array}{lccl}
\G{1+1}&=& & \cfrac{(\a_1 - 1)(2\a_1^2 - 2\a_1 + 1)}{\a_1(2\a_1 - 1)},\\[12pt]
\G{2+1} &=&-&\cfrac{(\a_1-1)(3\a_1^3 + \a_1^2(3\r{2}-4) + \a_1(\r{2}^2 - 2\r{2} + 2) + \r{2})}{\r{2}(\r{2}+\a_1)(\r{2} + 1)(3\a_1^2 + 2\a_1(\r{2} - 1) - \r{2})},\\[12pt]
 \end{array}
\end{equation*}
\begin{equation*}
\begin{array}{lccl}
&&&\left[
 \begin{array}{l}
 (\a_1-1)(\r{2}+1)\times\\
 \Big(4\a_1^4 + \a_1^3(3\r{2} + 4\r{3}-6)\\
 + \a_1^2\big( (\r{2}(3\r{3} -4)+ \big(\r{3}^2 -4 \r{3}+3\big) \big)\\
 + \a_1\big(\r{2}(\r{3}^2  - 2\r{3} +2)  + 2\r{3}\big) + \r{2}\r{3}\Big)
 \end{array}
 \right]\\
  \G{3+1}&=&& \rule[2pt]{10cm}{0.4pt},\\
&&& \hspace{10pt}\left[
 \begin{array}{l}
 \r{3}(\r{3}+\a_1)(\r{3} + 1)(\r{3}-\r{2}) \times \\
\Big(4\a_1^3 + 3\a_1^2(\r{2} + \r{3}-1) \\
 + 2\a_1(\r{2}\r{3} - \r{2} - \r{3}) - \r{2}\r{3}\Big)
  \end{array}
 \right]\\[20pt]
 \end{array}
\end{equation*}

\section{Algorithm computing $\g{i}$}
\label{sec_app1}
We present below steps in \cref{alg:gi} computing $\{\g{i}\,\vert\,i\in\S{}^{\op}\}$ related to \cref{bdfk_adapt} and \cref{formula_gi_adapt} approximating $\y(\t_n)$ using last approximations on $\mT_{n-1,\op}$ and with a time step $\h_n$ \cite[page 419]{book:hairer}. We denote by \textsf{Coeff} the function computing these coefficients.
\bigskip
\begin{algorithm}[!ht]
\caption{Steps to compute coefficients $\g{i}$ in \cref{bdfk_adapt}: $\mathsf{Coeff}\big([\t_{n-\op},\ldots,\t_{n-1},\h_n]\big)$}\label{alg:gi}
\begin{algorithmic}
 \Require $\t_n \gets \t_{n-1}+\h_n$
 \Require $\sfT\gets[\t_n,\t_{n-1},\ldots,\t_{n-\op}]$
 \State $\op \gets \mathsf{size}(\sfT)$
 \For{$i\gets 0$ to $\op$}
    \State $\g{i}\gets 0$
    \State $b\gets \h_n$
    \State $c \gets 1$
    \For{$m \gets 1$ to $i$}
        \If{$m< i-2$}
            \State $b \gets b \times (\sfT_0-\sfT_{m})$
        \EndIf
        \If{$m\neq i$}
            \State $c \gets c \times \displaystyle \frac{1}{\sfT_{i} - \sfT_{m}}$
        \EndIf
    \EndFor
    \For{$j \gets \max(1,i)$ to $\op$}
        \If{$j \geqslant 2$}
            \State $b \gets b \times (\sfT_0-\sfT_{j-1})$
        \EndIf
        \If{$j\neq i$}
            \State $c \gets c \times \displaystyle \frac{1}{\sfT_{i} - \sfT_{j}}$
        \EndIf
    \State $\g{i}\gets \g{i}+ (b\times c)$
    \EndFor
 \EndFor
 \State \Return $[\g{0},\ldots,\g{\op}]$
\end{algorithmic}
\end{algorithm}

\section*{Competing interests}
The research leading to these results was funded by Khalifa University under Grant Agreement No. FSU-2023-014 and No. RIG-2023-024. All authors have approved the submission and have no conflicts of interest to disclose.

\end{appendices}

\bibliography{references2}

@misc{deeb:part1,
    author = {Deeb, Ahmad and Dutykh, Denys},
    title = {Error estimation for numerical approximations of {ODE}s via composition techniques. {P}art {I}: {O}ne-step methods},
    note = {Submitted},
    doi = {https://doi.org/10.48550/arXiv.2409.10548},
    month = {Juin},
    year = {2024},
}

@article{deeb:AML1,
author = {Laadhari, Aymen and Deeb, Ahmad and Kawi, Badr},
title = {Hydrodynamics simulation of red blood cells: Employing a penalty method with double jump composition of lower order time integrator},
journal = {Mathematical Methods in the Applied Sciences},
doi = {https://doi.org/10.1002/mma.9607},
month = {August},
year = {2023}
}

@article{deeb:casson,
author = {Laadhari, Aymen and Deeb, Ahmad},
title = {Computational Modeling of Individual {R}ed {B}lood {C}ell Dynamics Using Discrete Flow Composition and Adaptive Time-Stepping Strategies},
journal = {Symmetry},
volume={15},
issue = {6},
pages = {1138},
publisher = {MDPI},
year = {2023},
doi = {https://doi.org/10.3390/sym15061138}
}

@article{ahmad_pgd_pade,
title = {Proper {G}eneralized {D}ecomposition using {T}aylor expansion for non-linear diffusion equations},
journal = {Mathematics and Computers in Simulation},
volume = {208},
pages = {71-94},
year = {2023},
issn = {0378-4754},
doi = {https://doi.org/10.1016/j.matcom.2023.01.008},
author = {Deeb, Ahmad and Kalaoun, Omar and Belarbi, Rafik}
}

@article{ahmad_comp_bpl_sfg_2015,
title ={Comparison between {B}orel-{P}ad{\'e} summation and factorial series, as time integration methods},
journal ={Discrete and Continuous Dynamical Systems - Series S},
volume ={9},
number ={2},
pages ={393-408},
year ={2016},
month={April},
issn ={1937-1632},
doi={https://doi.org/10.3934/dcdss.2016003},
author ={Deeb, Ahmad and Razafindralandy, Dina and Hamdouni, Aziz}
}

@article{DEEB_2022_bpl,
title = {Performance of {B}orel-{P}adé-{L}aplace integrator for the solution of stiff and non-stiff problems},
journal = {Applied Mathematics and Computation},
volume = {426},
publisher = {Elsevier},
year = {2022},
doi = {https://doi.org/10.1016/j.amc.2022.127118},
author = {Ahmad Deeb and Aziz Hamdouni and Dina Razafindralandy}
}

@article{ahmad_bpl_2014,
title = {Borel-{L}aplace summation method used as time integration scheme},
journal = {ESAIM: Procedings and Surveys},
year = {2014},
volume = {45},
pages = {318-327},
doi = {https://doi.org/10.1051/proc/201445033 },
author = {Deeb, Ahmad and Hamdouni, Aziz and Liberge, Erwan and Razafindralandy, Dina}
}

@inproceedings{ahmad_icnpaa_2016,
author = {Razafindralandy, Dina and Hamdouni, Aziz and Deeb, Ahmad },
title = {Considering factorial series as time integration method},
year = {2017},
Volume = {1798},
issue ={1},
series = {AIP Conference Proceedings},
booktitle = {11th International Conference on Mathematical Problems in Engineering, Aerospace and Sciences},
eventdate = {4-8},
language = {English},
address = {La Rochelle-France},
month = {July},
doi = {https://doi.org/10.1063/1.4972721},
publisher = {American Institute of Physics}
}

@article{deeb:stab-NS,
    author = {Deeb, Ahmad and Dutykh, Denys},
    title = {Numerical Integration of Navier-Stokes Equations by Time series expansion and Stabilized FEM},
    journal = {Mathematics and Computers in Simulation},
    doi = {10.1016/j.matcom.2025.01.023},
    volume = {233},
    pages={208--236},
    year = {2025},
    month={July}
}

@article{dromand_prince,
title = {A family of embedded {R}unge-{K}utta formulae},
journal = {J. Comp. App. Math.},
volume = {6},
number = {1},
pages = {19 - 26},
year = {1980},
note = {},
issn = {0377-0427},
doi = {http://dx.doi.org/10.1016/0771-050X(80)90013-3},
author = {Dormand, J.R. and Prince, P.J.}
}

@article{bogacki_shampine,
title = {A 3(2) pair of {R}unge-{K}utta formulas},
journal = {App. Math. Lett.},
volume = {2},
number = {4},
pages = {321 - 325},
year = {1989},
note = {},
issn = {0893-9659},
 author = {Bogacki, Przemysław and Shampine, Lawrence F.},
}

@book{book:butcher,
title =     {Numerical Methods for Ordinary Differential Equations},
author =    {Butcher,John C.},
publisher = {John Wiley \& Sons},
year =      {2008},
address = {West Sussex},
series =    {},
edition =   {2nd},
volume =    {}
}

@book{book:hairer,
title =     {Solving Ordinary Differential Equations I: Nonstiff Problems},
author =    {Hairer, Ernst and Norsett, Syvert P. and Wanner, Gerhard},
publisher = {Springer},
address = {Heidelberg},
year =      {2009},
series =    {Springer Series in Computational Mathematics},
edition =   {2nd},
volume =    {1}
}

@book{book:tomas,
title =     {Methods of Applied Mathematics for Engineers Scientists},
author =    {Co, Tomas B.},
publisher = {Cambridge University Press},
year =      {2013},
series =    {Michigan Technology University},
address = {New {Y}ork}
}

@book{book:hairer2,
title =     {Solving Ordinary Differential Equations II: Stiff and Differential-Algebraic Problems},
author =    {Hairer, Ernst and Wanner, Gerhard},
publisher = {Springer-Verlag},
year =      {1996},
address = {Berlin},
series =    {Springer Series in Computational Mathematics 14},
edition =   {Second Revised},
volume =    {}
}

@book{iserles_2008, 
address={Cambridge}, 
edition={2}, 
title={A First Course in the Numerical Analysis of Differential Equations}, 
publisher={Cambridge University Press},
author={Iserles, Arieh}, 
year={2008}, 
collection={Cambridge Texts in Applied Mathematics}
}

@article{Runge_1895,
author = {{R}unge, C.},
journal = {Math. Ann.},
pages = {167-178},
publisher = {National Academy of Sciences},
title = {Ueber die numerische {A}ufl\"{o}sung von Differentialgleichungen},
olume = {46},
year = {1895}
}

@article{butcher_1964, 
title={On {R}unge-{K}utta processes of high order}, 
volume={4}, 
DOI={10.1017/S1446788700023387}, 
number={2}, 
journal={Journal of the Australian Mathematical Society}, 
publisher={Cambridge University Press}, 
author={Butcher, John C.}, 
year={1964}, 
pages={179-194}}

@article{butcher_1963, 
title={Coefficients for the study of {R}unge-{K}utta integration processes},
volume={3},
number={2},
journal={Journal of Australian Mathematical Society},
publisher={Cambridge University Press},
author={Butcher, John C.},
year={1963},
pages={185–201}
}

@article{Butcher_1996,
title = {A history of {R}unge-{K}utta methods},
journal = {App. Num. Math.},
volume = {20},
number = {3},
pages = {247-260},
year = {1996},
issn = {0168-9274},
author = {Butcher, John C.}
}

@article{verwer_1996,
title = {Explicit {R}unge-{K}utta methods for parabolic {P}artial {D}ifferential {E}quations},
journal = {App. Num. Math.},
volume = {22},
number = {1},
pages = {359-379},
year = {1996},
issn = {0168-9274},
author = {Verwer, J.G.}
}

@book{hairer2002geometric,
title={Geometric Numerical Integration: Structure-Preserving algorithms for Ordinary Differential Equations},
author={Hairer, Ernst and Lubich, Christian and Wanner, Gerhard},
series={Springer series in computational mathematics},
year={2002},
publisher={Springer},
address = {Berlin},
edition = {2}
}

@article{casas_2021_complex,
title = {Compositions of pseudo-symmetric integrators with complex coefficients for the numerical integration of differential equations},
journal = { J. Comput. Appl. Math.},
volume = {381},
pages = {113006},
year = {2021},
issn = {0377-0427},
author = {Casas, Fernando and Chartier, Philippe and Escorihuela-Tomàs, Alejandro and Zhang, Yong}
}

@article{soderlind-06,
title = {Adaptive time-stepping and computational stability},
journal = { J. Comput. Appl. Math.},
volume = {185},
number = {2},
pages = {225-243},
year = {2006},
issn = {0377-0427},
author = {S\"{o}derlind, Gustaf and Wang, Lina}
}

@article{castella_2009,
title = {Splitting methods with complex times for parabolic equations},
journal = {BIT Numerical Mathematics},
volume = {49},
number = {},
pages = {487-508},
year = {2009},
author = {Castella, F. and Chartier, P. and Descombes, S. and Vilmart, G}
}

@article{Mc-1995,
author = {McLachlan, Robert I.},
title = {On the Numerical Integration of Ordinary Differential Equations by Symmetric Composition Methods},
journal = {SIAM Journal on Scientific Computing},
volume = {16},
number = {1},
pages = {151-168},
year = {1995},
doi = {10.1137/0916010}
}

@article{yoshida-1990,
title = {Construction of higher order symplectic integrators},
journal = {Physics Letters A},
volume = {150},
number = {5},
pages = {262-268},
year = {1990},
issn = {0375-9601},
doi = {10.1016/0375-9601(90)90092-3},
author = {Yoshida, Haruo}
}

@article{suzuki-1990,
title = {Fractal decomposition of exponential operators with applications to many-body theories and {M}onte {C}arlo simulations},
journal = {Physics Letters A},
volume = {146},
number = {6},
pages = {319-323},
year = {1990},
doi = {10.1016/0375-9601(90)90962-N},
author = {Suzuki, Masuo}
}

@article{casas-2008,
  title={Splitting and composition methods in the numerical integration of differential equations},
  author = {Blanes, Sergio and Casas, Fernando and Murua, Ander},
  journal={Bol. Scc. Esp. Mat. Apl.},
  pages={89-145},
  volume = {45},
  year={2008}
}

@article{casas-2006,
author = {Blanes, Sergio and Casas, Fernando and Murua, Ander},
title = {Composition Methods for Differential Equations with Processing},
journal = {SIAM Journal on Scientific Computing},
volume = {27},
number = {6},
pages = {1817-1843},
year = {2006},
doi = {10.1137/030601223}
}

@inproceedings{Hansen-69,
author = {Hansen, Eldon},
title = {Cyclic composite multistep predictor-corrector methods},
eventtitle = {ACM'69: Proceeding of the 1969 24th national conference},
pages = {135-139},
year = {1969},
month = {August},
doi = {10.1145/800195.805925}
}

@article{bickart-73,
author = {Bickart, Theodore and Picel, Zdenek},
title = {High order stiffly stable composite multistep methods for numerical integration of stiff differential equations},
journal = {BIT Numerical Mathematics},
volume = {13},
year = {1973},
pages = {272-286},
doi = {https://doi.org/10.1007/BF01951938}
}

@article{donelson-71,
author = {Donelson III, John and Hansen, Eldon},
title = {Cyclic Composite Multistep Predictor-Corrector Methods},
journal = {SIAM Journal on Numerical Analysis},
volume = {8},
number = {1},
pages = {137-157},
year = {1971},
doi = {10.1137/0708018}
}

@article{tendler-78,
author = {Tendler, Joel and Bickart, Theodore and Picel, Zdenek},
title = {A stiffly stable Integration Process Using Cyclic Composite Methods},
journal = {ACM Transactions on Mathematical Software},
volume = {4},
issue = {4},
pages = {339-368},
year = {1977},
doi = {10.1137/0708018}
}

@article{cash-77,
author = {Cash, J.R.},
title = {On a class of cyclic methods for the numerical integration of stiff systems of {O.D.E}.s},
journal = {BIT Numerical Mathematics},
volume = {17},
year = {1977},
pages = {270-280},
doi = {https://doi.org/10.1007/BF01932147}
}

@article{becker-98,
author = {Becker, J.},
title = {A second order {B}ackward {D}ifference method with variable steps for a parabolic problem},
journal = {BIT Numerical Mathematics},
volume = {38},
issue = {4},
year = {1998},
pages = {644-662},
doi = {https://doi.org/10.1007/BF02510406}
}

@article{bokanowski-21,
author = {Bokanowski, Olivier and Picarelli, Athena and Reisinger, Cristoph},
title = {Stability and convergence of second order {B}ackward {D}ifferentiation schemes for parabolic {H}amilton-{J}acobi-{B}ellman equations},
journal = {Numerische Mathematik},
volume = {148},
year = {2021},
pages = {187-222},
doi = {https://doi.org/10.1007/s00211-021-01202-x}
}

@article{wang-21,
author = {Wang, Wansheng and Mao, Mengli and Wang, Zheng},
title = {Stability and error estimates for the variable step-size {BDF}2 method for linear and semilinear parabolic equations},
journal = {Advances in Computational Mathematics},
volume = {47},
issue = {8},
year = {2021},
pages = {187-222},
doi = {https://doi.org/10.1007/s10444-020-09839-2}
}

@article{dahlquist-56,
author = {Dahlquist, Germund},
title = {Convergence and stability in the numerical integration of ordinary differential equations},
journal = {Mathematica Scandinavica},
volume = {4},
year = {1956},
pages = {33-53}
}

@article{gragg-64,
author = {Gragg, Williams and Stetter, Hans},
title = {Generalized Multistep Predictor-Corrector Methods},
journal = {Journal of ACM},
volume = {11},
issue = {2},
year = {1964},
pages = {188-209},
doi={https://doi.org/10.1145/321217.321223}
}

@article{blanes-01,
title = {High order numerical integrators for differential equations using composition and processing of low order methods},
journal = {Applied Numerical Mathematics},
volume = {37},
number = {3},
pages = {289-306},
year = {2001},
doi = {https://doi.org/10.1016/S0168-9274(00)00044-1},
author = {Blanes, Sergio}
}

@article{blanes-19,
title = {Splitting and composition methods with embedded error estimators},
journal = {Applied Numerical Mathematics},
volume = {146},
pages = {400-415},
year = {2019},
doi = {https://doi.org/10.1016/j.apnum.2019.07.022},
author = {Blanes, Sergio and Casas, Fernando and Thalhammer, Mechthild}
}

@article{blanes-10,
title = {Splitting methods with complex coefficients},
journal = {SeMA Journal: Bulletin of the Spanish Society of Applied Mathematics},
volume = {50},
pages = {47-60},
year = {2010},
doi = {https://doi.org/10.1007/BF03322541},
author = {Blanes, Sergio and Casas, Fernando and Murua, Ander}
}

@article{dahlquist-63,
title = {A special stability problem for linear multistep methods},
journal = {BIT Numerical Mathematics},
volume = {3},
pages = {27-43},
year = {1963},
doi = {https://doi.org/10.1007/BF01963532},
author = {Dahlquist, Germund G.},
}

@article{butcher-2009,
title = {General linear methods for ordinary differential equations},
journal = {Mathematics and Computers in Simulation},
volume = {79},
number = {6},
pages = {1834-1845},
year = {2009},
doi = {https://doi.org/10.1016/j.matcom.2007.02.006},
author = {Butcher, John C.}
}

@article{okounghae-12,
author = {Okounghae, R I and Ikhile, M N O},
title = {On the construction of high order {A}($\alpha$)-stable hybrid linear multistep methods for stiff {IVP}s in {ODE}s},
journal = {Numerical Analysis and Applications},
volume = {5},
pages = {231-241},
doi = {https://doi.org/10.1134/S1995423912030056}
}

@article{kirlinger-04,
title = {Linear multistep methods applied to stiff initial value problems--A survey},
journal = {Mathematical and Computer Modelling},
volume = {40},
number = {11},
pages = {1181-1192},
year = {2004},
doi = {https://doi.org/10.1016/j.mcm.2005.01.012},
author = {Kirlinger, G.}
}

@article{xiao-14,
title = {Two classes of implicit–explicit multistep methods for nonlinear stiff initial-value problems},
journal = {Applied Mathematics and Computation},
volume = {247},
pages = {47-60},
year = {2014},
doi = {https://doi.org/10.1016/j.amc.2014.08.066},
author = {Xiao, Aiguo and Zhang, Gengen and Yi, Xing}
}

@article{arkrivis-13,
title = {Implicit--explicit multistep methods for nonlinear parabolic equations},
journal = {Mathematics of Computation},
volume = {82},
issue={281},
pages = {45-68},
year = {2013},
doi = {10.1090/S0025-5718-2012-02628-7},
author = {Akrivis, Georgios}
}

@article{arkrivis-20,
author = {Akrivis, Georgios and Katsoprinakis, Emmanouil},
title = {An analogue to the  {A}$(\vartheta)$-stability concept for implicit-explicit {BDF} methods},
journal = {SIAM Journal on Numerical Analysis},
volume = {58},
number = {6},
pages = {3475-3503},
year = {2020},
doi = {10.1137/19M1275103}
}

@Inbook{Jay-15,
author={Jay, Laurent O.},
editor={Engquist, Bj{\"o}rn},
title={Lobatto Methods},
bookTitle={Encyclopedia of Applied and Computational Mathematics},
year={2015},
publisher={Springer},
address={Berlin, Heidelberg},
pages={817--826},
isbn={978-3-540-70529-1},
doi={10.1007/978-3-540-70529-1\_123},
}

@article{hairer-99,
title = {Stiff differential equations solved by {R}adau methods},
journal = {Journal of Computational and Applied Mathematics},
volume = {111},
number = {1},
pages = {93-111},
year = {1999},
issn = {0377-0427},
doi = {https://doi.org/10.1016/S0377-0427(99)00134-X},
author = {Hairer, Ernst and Wanner, Gerhard}
}

@article{radau-88,
title = {Étude sur les formules d'approximation qui servent à calculer la valeur numérique d'une intégrale définie},
journal = {Journal de Mathématiques Pures et Appliquées},
volume = {6},
number = {3},
pages = {283-336},
year = {1880},
issn = {0377-0427},
author = {Radau, Rodolphe}
}

@article{pia-17,
title = {An embedded formula of the {C}hebyshev collocation method for stiff problems},
journal = {Journal of Computational Physics},
volume = {351},
pages = {376-391},
year = {2017},
issn = {0021-9991},
doi = {https://doi.org/10.1016/j.jcp.2017.09.046},
author = {Piao, Xiangfan and Bu, Sunyoung and Kim, Dojin and Kim, Philsu}
}

@article{johnson-88,
author = {Johnson, Claes},
title = {Error Estimates and Adaptive Time-Step Control for a Class of One-Step Methods for Stiff Ordinary Differential Equations},
journal = {SIAM Journal on Numerical Analysis},
volume = {25},
number = {4},
pages = {908-926},
year = {1988},
doi = {10.1137/0725051}
}

@article{axelsson-69,
author = {Axelsson, O},
title = {A class of {A}-stable methods},
journal = {BIT Numerical Mathematics},
volume = {9},
pages = {185-199},
year = {1969},
doi = {10.1007/BF01946812}
}

@article{blanes-99,
author = {Blanes, Sergio and Casas, Fernando and Ros, J.},
title = {Symplectic Integration with Processing: {A} General Study},
journal = {SIAM Journal on Scientific Computing},
volume = {21},
number = {2},
pages = {711-727},
year = {1999},
doi = {10.1137/S1064827598332497}
}

@article{Vermeire-23,
title = {Embedded paired explicit {R}unge-{K}utta schemes},
journal = {Journal of Computational Physics},
volume = {487},
pages = {112-159},
year = {2023},
doi = {https://doi.org/10.1016/j.jcp.2023.112159},
author = {Vermeire, Brian C.}
}

@article{choi-96,
title = {Error estimates and adaptive time stepping for various direct time integration methods},
journal = {Computers \& Structures},
volume = {60},
number = {6},
pages = {923-944},
year = {1996},
doi = {10.1016/0045-7949(95)00452-1},
author = {Chang-Koon, Choi and Heung-Jin, Chung}
}

@article{jaradat-23,
title = {An Adaptive Time-Stepping Control Algorithm for Molten Salt Reactor Transient Analyses},
journal = {Annals of Nuclear Energy},
volume = {190},
pages = {109880},
year = {2023},
doi = {10.1016/j.anucene.2023.109880},
author = {Jaradat, Mustafa K. and {Sik Yang}, Won}
}

@article{soderlind-02,
title = {Automatic Control and Adaptive Time-Stepping},
journal = {Numerical Algorithms},
volume = {31},
pages = {281-310},
year = {2002},
doi = {10.1023/A:1021160023092},
author = {Söderlind, Gustaf}
}

@article{soderlind-21,
title = {Local error estimation and step size control in adaptive linear multistep methods},
journal = {Numerical Algorithms},
volume = {86},
pages = {537-563},
year = {2021},
doi = {10.1007/s11075-020-00900-1},
author = {Arévalo, Carmen and Söderlind, Gustaf and Hadjimichael Yiannis and Feket, Imre}
}

@article{meng-23,
title = {An adaptive {BDF}2 implicit time-stepping method for the no-slope-selection epitaxial thin film model},
journal = {Computational and Applied Mathematics},
volume = {42},
number= {124},
pages = {},
year = {2021},
doi = {10.1007/s40314-023-02250-9},
author = {Meng, Xiangjung and Zhang, Zhengru}
}

@article{baron-17,
title = {Adaptive multistep time discretization and linearization based on a posteriori error estimates for the {R}ichards equation},
journal = {Applied Numerical Mathematics},
volume = {112},
pages = {104-125},
year = {2017},
doi = {10.1016/j.apnum.2016.10.005},
author = {Baron, V. and Coudière, Y. and Sochala, P.}
}

@article{wang-23,
title = {An adaptive space and time method in partially explicit splitting scheme for multiscale flow problems},
journal = {Computers \& Mathematics with Applications},
volume = {144},
pages = {100-123},
year = {2023},
doi = {10.1016/j.camwa.2023.05.034},
author = {Wang , Yating and Leung, Wing Tat}
}

@article{yan-21,
title = {Adaptive time-stepping schemes for the solution of the {P}oisson-{N}ernst-{P}lanck equations},
journal = {Applied Numerical Mathematics},
volume = {163},
pages = {254-269},
year = {2021},
doi = {10.1016/j.apnum.2021.01.018},
author = {Yan, David and Pugh, M.C. and Dawson, F.P.}
}

@article{lambert-90,
title = {On the local error and the local truncation error of linear multistep methods},
journal = {BIT Numerical Mathematics},
volume = {30},
pages = {637-681},
year = {1990},
doi = {10.1007/BF01933215},
author = {Lambert, John Denholm}
}

@article{khashin-14,
title = {Estimating the error in the classical {R}unge-{K}utta methods},
journal = {Computational Mathematics and Mathematical Physics},
volume = {54},
pages = {767-774},
year = {2014},
doi = {10.1134/S0965542514050145},
author = {Khashin, S.I.}
}

@article{hochbruck-18,
title = {Error analysis of implicit {R}unge-{K}utta methods for quasilinear hyperbolic evolution equations},
journal = {Numerische Mathematik},
volume = {138},
pages = {557-579},
year = {2018},
doi = {10.1007/s00211-017-0914-6},
author = {Hochbruck, M. and Pazur, T. and Schnaubelt, R.}
}

@article{hochbruck2010exponential,
title={Exponential integrators},
author={Hochbruck, Marlis and Ostermann, Alexander},
journal={Acta Numerica},
volume={19},
pages={209--286},
year={2010},
publisher={Cambridge University Press}
}

@article{Pope_1963,
author = {Pope, David A.},
title = {An Exponential Method of Numerical Integration of Ordinary Differential Equations},
journal = {Commun. ACM},
issue_date = {Aug. 1963},
volume = {6},
number = {8},
year = {1963},
issn = {0001-0782},
pages = {491--493},
numpages = {3},
doi = {10.1145/366707.367592},
acmid = {367592},
publisher = {ACM},
address = {New {Y}ork}
}

@article{cox_2002,
title = {Exponential {T}ime {D}ifferencing for Stiff Systems},
journal = {Journal of Computational Physics},
volume = {176},
number = {2},
pages = {430 - 455},
year = {2002},
note = {},
issn = {0021-9991},
doi = {http://dx.doi.org/10.1006/jcph.2002.6995},
author = {Cox, S.M. and Matthews, P.C.}
}

@book{henrici-62,
author = {Henrici, Peter},
title = {Discrete Variable Methods in Ordinary Differential Equations},
publisher = {Wiley},
language = {English},
address = {New {Y}ork},
year = {1962}
}

@article{gander-20,
author = {Gander, M.J. and Wanner, G.},
title = {Exact {BDF} stability angles with {M}aple},
journal= {BIT Numerical Mathematics},
publisher = {Springer},
volume = {60},
pages = {615-617},
doi = {10.1007/s10543-019-00796-x},
language = {English},
year = {2020}
}

@article{ramiere-15,
title = {Iterative residual-based vector methods to accelerate fixed point iterations},
journal = {Computers \& Mathematics with Applications},
volume = {70},
number = {9},
pages = {2210-2226},
year = {2015},
issn = {0898-1221},
doi = {https://doi.org/10.1016/j.camwa.2015.08.025},
author = {Ramière, Isabelle and Helfer, Thomas}
}

@article{martinez-23,
title = {A Fixed-Point State observer with {S}teffensen-{A}itken accelerated convergence},
journal = {Journal of the Franklin Institute},
volume = {360},
number = {10},
pages = {6757-6782},
year = {2023},
issn = {0016-0032},
doi = {https://doi.org/10.1016/j.jfranklin.2023.04.023},
author = {Martinez-Guerra, Rafael and Flores-Flores, Juan Pablo}
}

@article{hairer-83,
author = {Hairer, E. and Wanner, G.},
title = {On the Instability of the {BDF} Formulas},
journal = {SIAM Journal on Numerical Analysis},
volume = {20},
number = {6},
pages = {1206-1209},
year = {1983},
doi = {10.1137/0720090}
}

@article{chen-19,
author = {Chen, Wenbin and Wang, Xiaoming and Yan, Yue and Zhang, Zhuying},
title = {A Second Order {BDF} Numerical Scheme with Variable Steps for the {C}ahn--{H}illiard Equation},
journal = {SIAM Journal on Numerical Analysis},
volume = {57},
number = {1},
pages = {495-525},
year = {2019},
doi = {10.1137/18M1206084}
}

@article{Liao-20,
author = {Liao, Hong-lin and Tang, Tao and Zhou, Tao},
title = {On Energy Stable, Maximum-Principle Preserving, Second-Order {BDF} Scheme with Variable Steps for the {A}llen--{C}ahn Equation},
journal = {SIAM Journal on Numerical Analysis},
volume = {58},
number = {4},
pages = {2294-2314},
year = {2020},
doi = {10.1137/19M1289157}
}

@article{Liao-21,
author = {Liao, Hong-lin and Zhang, Zhimin},
title = {Analysis of adaptive {BDF}2 scheme for diffusion equations},
journal = {Mathematics of Compution},
volume = {90},
pages = {1207-1226},
year = {2021},
doi = {10.1090/mcom/3585}
}

@article{Li-22,
author = {Li, Zhaoyi and Liao, Hong-lin},
title = {Stability of Variable-Step {BDF2} and {BDF3} Methods},
journal = {SIAM Journal on Numerical Analysis},
volume = {60},
number = {4},
pages = {2253-2272},
year = {2022},
doi = {10.1137/21M1462398}
}

@article{Liao-23,
author = {Liao , Hong-lin and Tang, Tao and Zhou , Tao},
title = {Discrete Energy Analysis of the Third-Order Variable-Step {BDF} Time-Stepping for Diffusion Equations},
journal = {Journal of Computational Mathematics},
year = {2023},
volume = {41},
number = {2},
pages = {325--344},
doi = {10.4208/jcm.2207-m2022-0020}
}

@article{Xu-23,
title = {The long time error estimates for the second order backward difference approximation to sub-diffusion equations with boundary time delay and feedback gain},
journal = {Mathematics and Computers in Simulation},
volume = {208},
pages = {186-206},
year = {2023},
issn = {0378-4754},
doi = {10.1016/j.matcom.2023.01.027},
author = {Da Xu}
}

@Article{CiCP-35-1327,
author = {Zhan, JiajunYang  and LeiDu, Rui and Cui , Zixuan},
title = {Towards Preserving Geometric Properties of Landau-Lifshitz-Gilbert Equation Using Multistep Methods},
journal = {Communications in Computational Physics},
year = {2024},
volume = {35},
number = {5},
pages = {1327--1351},
issn = {1991-7120},
doi = {https://doi.org/10.4208/cicp.OA-2023-0201}
}

@Article{CiCP-13-461,
author = {Ganesh, N. and Balakrishnan, N.},
title = {A h-Adaptive Algorithm Using Residual Error Estimates for Fluid Flows},
journal = {Communications in Computational Physics},
year = {2013},
volume = {13},
number = {2},
pages = {461--478},
issn = {1991-7120},
doi = {https://doi.org/10.4208/cicp.170811.210212a}
}

@article{Maset-21,
title = {Relative error analysis of matrix exponential approximations for numerical integration},
author = {Stefano Maset},
pages = {119--158},
volume = {29},
number = {2},
journal = {Journal of Numerical Mathematics},
doi = {10.1515/jnma-2020-0019},
year = {2021},
lastchecked = {2025-01-17}
}

@article{Vu-2024,
title = {Stability and error analysis of a semi-implicit scheme for incompressible flows with variable density and viscosity},
author = {An Vu and Loic Cappanera},
journal = {Journal of Numerical Mathematics},
doi = {10.1515/jnma-2024-0033},
year = {2024},
lastchecked = {2025-01-17}
}

@article{Kropielnicka-2024,
title = {Effective highly accurate time integrators for linear {K}lein--{G}ordon equations across the scales},
author = {Karolina Kropielnicka and Karolina Lademann and Katharina Schratz},
journal = {Journal of Numerical Mathematics},
doi = {10.1515/jnma-2023-0070},
year = {2024},
lastchecked = {2025-01-17}
}

@article{Diaz-2021,
title = {Convergence of time-splitting approximations for degenerate convection–diffusion equations with a random source},
author = {Roberto Díaz-Adame and Silvia Jerez},
pages = {23--38},
volume = {29},
number = {1},
journal = {Journal of Numerical Mathematics},
doi = {10.1515/jnma-2020-0012},
year = {2021},
lastchecked = {2025-01-17}
}

@article{comp1,
author = {Ayari, M. and Hammami, M.A. and Klai, Z.},
journal = {Computational and Applied Mathematics},
title = {Advanced state estimation methods with application to different classes of BMS},
volume = {45},
number = {87},
year = {2026},
doi = {10.1007/s40314-025-03575-3}
}

@article{comp2,
author = {Yong, Jinjun and Ye, Changlun and Luo, Xianbing and Sun, Shuyu},
title = {Improved error estimates of ensemble Monte Carlo methods for random transient heat equations with uncertain inputs},
journal = {Computational and Applied Mathematics},
volume = {44},
number = {58},
year = {2025},
doi = {10.1007/s40314-024-03022-9},
}

@article{comp3,
author = {Jaiswal, Aishwarya and Kumar, Shashikant and Kumar, Sunil},
title = {A priori and a posteriori error analysis for a system of singularly perturbed Volterra integro-differential equations},
journal = {Computational and Applied Mathematics},
volume = {42},
number = {278},
year = {2023},
doi = {10.1007/s40314-023-02406-7},
}

@article{comp4,
author = {Wang, Yifei and Huang, Jin and Deng, Ting and Li, Hu},
title = {An efficient numerical approach for solving variable-order fractional partial integro-differential equations},
journal = {Computational and Applied Mathematics},
volume = {41},
number = {411},
year = {2022},
doi = {10.1007/s40314-022-02131-7},
}

\end{document}